\RequirePackage{ifpdf}
\ifpdf
\documentclass[10pt,pdftex]{amsart}
\else
\documentclass[10pt,dvips]{amsart}
\fi

\usepackage[bookmarks,colorlinks=false,backref,plainpages=false]{hyperref}
\hypersetup{
  pdftitle={Mixed finite element methods for a Dirac source:
            divergence-form splitting and L\textasciicircum p error analysis},
  pdfauthor={Yueyao Wu, Shun Zhang},
  pdfsubject={65N30, 65N15, 65N50},
  pdfkeywords={mixed finite element method; Dirac measure;
               divergence-form splitting; L\textasciicircum p error estimates;
               a posteriori error estimation}
}
\usepackage[english]{babel}
\usepackage{amsmath,amssymb,amsthm,cancel}
\usepackage{graphics}
\usepackage{subfigure}
\usepackage{graphicx}
\usepackage{epsfig}
\usepackage{epstopdf}
\usepackage{framed}
\usepackage{booktabs}
\usepackage{amsaddr}

\newif\ifprintwide
\printwidetrue
\ifprintwide
\fi

\numberwithin{equation}{section}

\newtheorem{thm}{Theorem}[section]
\newtheorem{cor}[thm]{Corollary}
\newtheorem{lem}[thm]{Lemma}
\newtheorem{prop}[thm]{Proposition}
\newtheorem{assumption}[thm]{Assumption}
\theoremstyle{definition}

\newtheorem{exm}[thm]{Example}
\theoremstyle{remark}
\newtheorem{remark}[thm]{Remark}

\def\XXint#1#2#3{{\setbox0=\hbox{$#1{#2#3}{\int}$}
\vcenter{\hbox{$#2#3$}}\kern-.5\wd0}}

\renewcommand\O{\Omega}

\newcommand{\bsigma}{\mbox{\boldmath$\sigma$}}
\newcommand{\btau}{\mbox{\boldmath$\tau$}}
\newcommand{\brho}{\mbox{\boldmath$\rho$}}

\def\bxi{\boldsymbol{\xi}}

\def\ba{{\bf a}}
\def\bb{{\bf b}}
\def\bc{{\bf c}}

\def\bff{{\bf f}}

\def\bn{{\bf n}}

\def\bq{{\bf q}}

\def\bv{{\bf v}}
\def\bw{{\bf w}}
\def\bx{{\bf x}}
\def\by{{\bf y}}

\def\cE{{\cal E}}

\def\cP{{\cal P}}

\def\cR{{\cal R}}

\def\cT{{\cal T}}

\def\f12{\frac12}          %

\def\p{\partial}

\newcommand{\gradt}{\nabla\cdot}

\DeclareMathOperator{\diam}{diam}

\def\divvr{\operatorname{div}}

\newcommand{\bdm}{\begin{displaymath}}
\newcommand{\edm}{\end{displaymath}}
\newcommand{\beq}{\begin{equation}}
\newcommand{\eeq}{\end{equation}}
\newcommand{\beqa}{\begin{eqnarray}}
\newcommand{\eeqa}{\end{eqnarray}}
\newcommand{\beqas}{\begin{eqnarray*}}
\newcommand{\eeqas}{\end{eqnarray*}}

\def\O{\Omega}
\def\p{\partial}
\def\ba{\mathbf{a}}
\def\bx{\mathbf{x}}
\def\bn{\mathbf{n}}
\def\bsigma{\boldsymbol{\sigma}}
\def\btau{\boldsymbol{\tau}}

\def\brho{\boldsymbol{\rho}}
\def\bv{\mathbf{v}}
\def\bw{\mathbf{w}}
\def\bff{\mathbf{f}}
\def\bq{\mathbf{q}}

\def\bb{\mathbf{b}}
\def\jump#1{[\![#1]\!]}
\def\cT{\mathcal{T}}
\def\cN{\mathcal{N}}

\def\cE{\mathcal{E}}
\def\cP{\mathcal{P}}
\def\cR{\mathcal{R}}
\def\gradt{\nabla\!\cdot\!}
\def\divvr{\mathrm{div}}

\def\RT{\mathrm{RT}}
\def\Rone{\mathbb{R}}
\def\diam{\operatorname{diam}}

\begin{document}

\title[Mixed FEM for Dirac sources]
{Mixed Finite Element Methods for a Dirac Source:
 Divergence-Form Splitting and \texorpdfstring{$L^p$}{Lp}
 Error Analysis}

\author[Y. Wu and S. Zhang]{Yueyao Wu and Shun Zhang}
\address{Department of Mathematics, City University of Hong Kong,
Kowloon Tong, Hong Kong, China}
\email{wyueyao2-c@my.cityu.edu.hk, shun.zhang@cityu.edu.hk}
\thanks{This work was supported in part by the Research Grants Council
of the Hong Kong SAR, China, under the GRF Grant Projects
No.\ CityU 11316222 and CityU 11305025.}
\date{\today}


\begin{abstract}
For a mixed finite element method, a Dirac source is first a failure of
duality, not of regularity: the conservation equation is tested
against a Lebesgue space, and a Dirac measure lies in the dual of none.
We therefore remove the measure from the conservation law by a
divergence-form splitting: an explicit field $\bff_2$ with
$\gradt\bff_2=\delta_{x_0}$ is subtracted and
$\bsigma=-A\nabla u-\bff_2$ is taken as the mixed unknown, so that
$\gradt\bsigma=f_1$.  Equivalently, and independently of any
discretization, the problem is rewritten as
$-\gradt(A\nabla u)=f_1+\gradt\bff_2$, with data in divergence form
generated by a field of $L^p$.  The field $\bff_2$ depends on $x_0$
alone, so no coefficient-dependent singular solution and no discrete
delta are needed, only the load vector of the $\RT_0$--$P_0$ system changes,
and $x_0$ may sit anywhere relative to the mesh.  In the unmatched
$|x-x_0|^{-1}$ regime analyzed here the modified flux belongs to
$L^p(\O)^2$ for every $1<p<2$ and not to $L^2(\O)^2$, so the flux error
analysis has to leave the Hilbert scale.  We prove a quasi-best
approximation bound for the flux, and with it that on a quasi-uniform
family the flux error is exactly of order $h^{2/p-1}$, the matching
lower bound coming already from the single element carrying the pole.  Grading the mesh there restores first-order
complexity, $N^{-1/2}$ in the number of elements, and the adaptive
computations attain it.  The scalar variable is limited only by
piecewise constant approximation of $u$, which it attains.  We also
prove a
residual norm equivalence in the Lebesgue scale, yielding a computable
$L^p$ estimator, reliable and locally efficient for the mixed flux
together with a recovered potential.
\end{abstract}

\maketitle


\section{Introduction}
\label{sec:intro}
Let $\O\subset\Rone^2$ be a bounded, simply connected Lipschitz domain,
let $x_0\in\O$, and consider the diffusion problem with a Dirac source,
\[
  -\gradt(A\nabla u)=f_1+\delta_{x_0}\ \text{ in }\O,
  \qquad u=g\ \text{ on }\p\O .
\]

\subsection*{A diagnosis}

Dirac data are commonly regarded as a problem of approximation.  Already
for the Laplacian the solution has a logarithmic singularity at the
source, so that the natural Sobolev scale is $W^{1,p}(\O)$ with $p<2$
and not $H^1(\O)$; the
natural response is to work harder at approximating it, by weighted
norms, by meshes graded at the pole, or by regularizing the data.  A
substantial literature does exactly this
\cite{Scott1973, Casas1985, Eriksson1985,
ArayaBehrensRodriguez2006, ApelBenedixSirchVexler2011,
KopplWohlmuth2014, DAngelo2012, AgnelliGarauMorin2014,
GaspozMorinVeeser2017, Bertoluzza2018, FuhrerHeuerKarkulik2022}.

For the usual mixed formulation low regularity is not the first
difficulty; before approximation is considered, the formulation already
fails at the level of duality.  The low regularity of $u$ is an
inconvenience shared by every discretization of \eqref{eq:deltaeq}, and
what the conforming treatments have in common is that they avoid this
earlier difficulty of duality: they test the equation against a space of
\emph{continuous} functions, and $\delta_{x_0}$ is a bounded
functional on $W^{1,p'}_0(\O)$ as soon as $p'>2$, because
$W^{1,p'}_0(\O)\hookrightarrow C(\overline\O)$ in two dimensions.

A mixed method cannot use that escape.  It tests the conservation
equation against the scalar space, and in a mixed method that space is
by design a Lebesgue space: no derivative is imposed on $u$, which is
the entire point of the formulation.  Since $\delta_{x_0}$ is not
represented by a locally integrable function, it lies in the dual of no
Lebesgue space, and the pairing $\langle\delta_{x_0},v\rangle$ is
undefined for a general scalar test function.  What fails is therefore
\emph{duality}, not approximation; it fails already at the continuous
level, before any mesh is introduced, and no choice of mesh, of
element, or of boundary condition repairs it.  Replacing $\delta_{x_0}$ by a discrete
representer does not repair it either, as Section~\ref{subsec:whyf2}
shows: it requires the mesh to keep $x_0$ interior to an element
\cite{HoustonWihler2012}, which conflicts with fitting the mesh to the
coefficient interfaces.

\subsection*{The remedy, and what it costs}

A failure of duality is repaired by changing the formulation, not the
mesh and not the approximation.  The measure has to leave the
conservation equation, and there is a way to remove it.  It is
the one introduced in \cite[Sec.~4]{Zhang2023} for a least-squares method
with a load in $H^{-1}(\O)$: such a load is divergence of a
square-integrable field, and subtracting that field from the physical
flux restores an $L^2(\O)$ conservation equation at no cost.  The
present paper is what happens when the same substitution is applied to data
for which no square-integrable field exists.  Let $\bff_2$ be the
explicit field \eqref{eq:f2} built from the fundamental solution of the
Laplacian, so that $\gradt\bff_2=\delta_{x_0}$, and take as the mixed
unknown the \emph{modified flux}
\[
  \bsigma:=-A\nabla u-\bff_2 ,
  \qquad\text{so that}\qquad
  \gradt\bsigma=f_1 .
\]
The Dirac measure is gone from the conservation law; the singular field
survives only in the constitutive equation, where it is paired with a
\emph{flux} test function and H\"older's inequality suffices.  The
mixed formulation is then the standard one.

Eliminating $\bsigma$ shows what the substitution is, stripped of the method
that suggested it.  The problem \eqref{eq:deltaeq} becomes
\[
  -\gradt(A\nabla u)=f_1+\gradt\bff_2\ \ \text{in }\O,
  \qquad u=g\ \ \text{on }\p\O ,
\]
which is \eqref{eq:deltaeq} itself (the identity
$\gradt\bff_2=\delta_{x_0}$ is exact, no parameter is introduced and no
datum is perturbed), but whose right-hand side is now a distribution
in divergence form generated by a field of $L^p(\O)^2$.  That is the
standard shape of data for a second-order elliptic problem; a Dirac
measure is not.  The substitution therefore does more than repair the
mixed formulation: it returns the Dirac problem to a form for which
every discretization may use, or develop, its own theory for data in
divergence form, and not a formulation built for a measure.  Conforming, discontinuous Galerkin, hybridizable, finite
volume and nonconforming methods are all in that position
(Remark~\ref{rem:other-methods}).  The mixed method analyzed here is
one application among them, and it is the one carried through because
$\bsigma$ is the unknown whose regularity the substitution changes.

What matters about $\bff_2$ is what it does \emph{not} depend
on.  It depends on $x_0$ alone: not on
$A$, so nothing is recomputed when the coefficient changes and no
coefficient-dependent singular solution need be known; not on the
equation, so the same substitution serves the full second-order
operator with lower-order terms (Section~\ref{sec:conclusions}); and not
on the discretization, so the same substitution is available whatever
the method.  At the discrete level it changes the representation of the
singular load, while the operator-dependent part of the method, the
bilinear form and whatever numerical fluxes or penalty terms it carries,
is untouched.  It is also indifferent to the position
of $x_0$: no point value of a discrete function is ever taken, so
$x_0$ may be a vertex.

What is given up is that the singular structure of the operator is not
canceled.  Section~\ref{subsec:phiA} makes the comparison exact.  The
ideal subtracted field is $-A\nabla\phi_A$, built from a singular
solution of the operator itself; it removes the singularity from the
modified flux, and under the Hilbert-scale boundary regularity stated
below it leaves $\bsigma\in H(\divvr;\O)$, where the classical
Hilbert-space theory applies unchanged.  We call the splitting \emph{matched} at $x_0$ when the
modified flux is square-integrable there; exact cancellation of the
singular flux by $\bff_2$ is one way of realizing it.  For a coefficient
which is a single constant matrix $A_0$ near $x_0$ the condition is that
$A_0$ be a scalar multiple of the identity, but for a discontinuous
coefficient it is a condition on the traces from all sides, and not one
that either trace satisfies by itself
(Section~\ref{subsec:regularity}).  In the simple
cases where $\phi_A$ is known one should use it.  In less simple ones it can still
be found, at the cost of a computation carried out afresh for each
coefficient; for an equation with lower-order terms no comparably simple
coefficient-dependent singular field is available in general, and
obtaining one amounts to solving a separate singular problem for each
operator.  The point of using $\bff_2$ is that the analysis
does not depend on which of these is the case.

The cost is therefore paid in the regularity of $\bsigma$: when the
mismatch retains a nonzero homogeneous $|x-x_0|^{-1}$ part, as it does
in every problem computed in Section~\ref{sec:numerics},
\[
  \bsigma\in L^p(\O)^2\quad\text{for every }1<p<2,
  \qquad\text{but}\qquad \bsigma\notin L^2(\O)^2 ,
\]
and if that leading part vanishes a better behavior is not excluded.
This is a statement about the \emph{norms}, not about the method: the
discrete problem is algebraically the standard $\RT_0$--$P_0$ scheme,
the exponent $p$ enters neither the matrix nor its solvability, and
accordingly we do not develop a Banach-space well-posedness theory.
But it does mean that the entire error analysis has to be redone below
the Hilbert exponent, where orthogonality, best approximation in an
inner product and the usual coercivity arguments are no longer
available.

\subsection*{The cost is affordable}

The rest of the paper establishes that.  Each of the following is
needed because the corresponding $L^2$ argument is unavailable.

\emph{(i) A quasi-best approximation bound for the flux.}  The analytic
input is a single $h$-uniform stability estimate for the weighted
projection onto the discrete solenoidal subspace, which at $p=2$ is an
orthogonal projection and for $p<2$ is the estimate of Dur\'an
\cite{Duran1988}.  Under it, Theorem~\ref{thm:flux-quasibest} gives a
quasi-best approximation bound over the equilibrated affine subspace,
with no data-oscillation term.
What Dur\'an proves is in fact a two-term estimate, in which the
divergence of the argument is measured in a negative norm; taken in that
form it gives instead Theorem~\ref{thm:flux-two-term}, whose leading
constant does not degenerate as $p\downarrow1$, the $p$-dependence that
does being moved to an oscillation term of one higher order in $h$.

\emph{(ii) A discrete $L^{p'}$ inf--sup condition for the divergence.}
The naive route to the scalar variable combines the standard
$H(\divvr)$ inf--sup condition with an inverse inequality and loses the
factor $h^{-(2/p-1)}$, which is exactly the flux rate: no convergence
at all.  Corollary~\ref{cor:discrete-Lp-infsup} instead lifts the
divergence directly in the $L^{p'}$-graph norm dual to the $L^p$-flux
error, by applying the canonical Raviart--Thomas interpolant to a right
inverse of the divergence, and yields an $L^2$-estimate for $u-u_h$ free of
any inverse inequality (Theorem~\ref{thm:scalar-basic}); a duality
argument then gains the factor $h^{2/p'}$, which converts the flux rate
$h^{2/p-1}$ into the first-order rate $h$
(Theorem~\ref{thm:scalar-refined}).

\emph{(iii) A comparison flux the canonical interpolant cannot
provide.}  We exhibit a piecewise anisotropic coefficient for which the
canonical Raviart--Thomas normal moments of $\bsigma$ are not
integrable on the edges meeting $x_0$
(Proposition~\ref{prop:moments-fail}): the degrees of freedom
themselves are undefined, not just inaccurate.  An equilibrated
comparison flux is constructed instead by local flux balance on the
patch of the pole (Section~\ref{subsec:patch}).

\emph{(iv) A barrier on quasi-uniform meshes, and grading past it.}
The rate $h^{2/p-1}$ is not an artifact of the construction:
Theorem~\ref{thm:sharp-local} attains it from below already by the free
local $L^p$-best approximation on the element carrying the pole, so it
is a property of the singularity and not of the method.  The
barrier is local, and grading removes it.  Balancing the local
contributions gives graded meshes on which the flux error is
$O(N^{-1/2})$ in the number of elements
(Proposition~\ref{prop:equidistribution}), the
first-order complexity of the space, and the adaptive computations of
Section~\ref{sec:numerics} attain it.  For the scalar variable the
duality argument shows that the discrete part $Q_hu-u_h$ is $O(h)$ and
therefore negligible, so that $\|u-u_h\|_{L^2(\O)}$ is asymptotically
the best $P_0$ approximation of $u$, of order $h|\log h|^{1/2}$ and
independent of $p$ (Corollary~\ref{cor:scalar-rate}).  That logarithm is
again a property of the singularity and not of the method, but of a
different kind: it comes from all scales between $h$ and $1$ at once,
not from the element at the pole.

\emph{(v) A residual norm equivalence, and an estimator.}  The a
posteriori theory rests on Theorem~\ref{thm:norm-equivalence}, an
equivalence in the Lebesgue scale between the error of a pair
$(\btau,v)\in H^q(\divvr;\O)\times W_0^{1,q}(\O)$ and the residual of
the first-order system it produces.  This is the counterpart of the norm
equivalence underlying first-order system least-squares methods, with
the $W^{1,q}$ well-posedness of Assumption~\ref{ass:A3} in place of the
Hilbert-space coercivity that is unavailable for $p<2$.  Filling the
flux slot with the discrete solution gives a computable $L^p$
estimator, reliable and locally efficient for the augmented
flux--potential pair, with no bubble functions, no jump terms and
no flux reconstruction (Theorem~\ref{thm:estimator}); reliability for
the flux alone follows, while efficiency for it alone would need control
of the recovery.
The estimate holds for \emph{any} $w\in W^{1,p}(\O)$ carrying the
boundary data of $u$; one is produced in
Section~\ref{subsec:potential-field}.

\subsection*{Relation to other work}

Error analysis of mixed methods in norms other than the energy norm has
a long history, global estimates in \cite{DouglasRoberts1985} and
maximum-norm estimates in \cite{GastaldiNochetto1987}.  The $L^p$
theory, in the full range $1\le p\le\infty$, is due to Dur\'an
\cite{Duran1988}, and its stability estimate is what the flux analysis
below rests on.  To
our knowledge no $H(\divvr)$-conforming mixed analysis of a Dirac
source in the $L^p$ scale with $p<2$ exists.

A second route replaces the singular load by a computable regularization
and solves with the perturbed datum.  It works, and gives
quasi-optimality and convergence rates for least-squares and
discontinuous Petrov--Galerkin methods \cite{FuhrerHeuerKarkulik2022}
and, in weaker norms, for the lowest-order mixed method itself
\cite{Fuhrer2024}.  What separates the two routes is not scope.  For a
load in
$H^{-1}(\O)$, equivalently for divergence-form data generated by an
$L^2$ field, the substitution used here settles the question outright: such a load is the
divergence of a square-integrable field, subtracting that field from the
physical flux returns the conservation equation to $L^2(\O)$ and the
modified flux to $H(\divvr;\O)$, and the standard theory then applies
with no regularization, no operator to construct and no perturbation of
the data.  That is \cite[Sec.~4]{Zhang2023}, where the substitution is
introduced and where the lower-order terms $\bb$ and $c$ are carried
along.  A Dirac source is the first case in which it cannot be done, and
the reason is exact: $\delta_{x_0}\notin H^{-1}(\O)$ says precisely
that no square-integrable field has divergence $\delta_{x_0}$.  That
forces the lifting itself, and with it the constitutive pairing it
generates, below the Hilbert exponent; whether the modified flux also
leaves $L^2(\O)^2$ depends on its cancellation against the constitutive
singularity, and in the unmatched configurations studied here it does.

The two routes then differ in where the failure of duality is
absorbed.
There it is absorbed by the datum: one solves with $Q_h^\star f$ in
place of $f$, and the distance between the two is quantified and
carried through the estimates.  Here it is absorbed by the norm: the
datum is untouched, because $\gradt\bff_2=\delta_{x_0}$ is an identity
and no operator has to be constructed, and the price appears in the
Lebesgue exponent instead.  One consequence of the second choice should
be stated, since it is what makes it necessary for our purpose.
Whenever the modified flux is not square-integrable, as in the
configurations studied here, $\|\bsigma-\bsigma_h\|_{L^2(\O)}$ is
infinite for every $\bsigma_h$ and the flux error is not a quantity an
$L^2$ theory can measure;
accordingly the point-load results of
\cite{FuhrerHeuerKarkulik2022} concern the scalar variable, and
\cite{Fuhrer2024} works with $H^{-1}(\O)$ loads, for which
$\bsigma\in L^2(\O)^2$.  In those configurations, measuring the modified
flux is what forces the exponent below $2$, and the flux is the quantity
the paper analyzes.

We do not invoke a Banach-space well-posedness theory: the discrete
problem here is the classical one \cite{BrezziFortin1991}, and $p$
enters only through the norms in which the error is measured.  It would
in any case not be the relevant relaxation, since a Banach setting
weakens the divergence slot, whereas here $\gradt\bsigma=f_1$ inherits
whatever regularity the data have and it is the field that leaves
$L^2(\O)^2$.

\subsection*{Outline}

Section~\ref{sec:splitting} introduces the splitting and establishes
the regularity of $\bsigma$.  Section~\ref{sec:mixed-reformulation}
sets out the mixed formulations.  Section~\ref{sec:discrete} sets up the
discrete problem and proves the quasi-best approximation bound for the
flux, reducing the analysis to a constrained approximation problem;
Section~\ref{sec:approx} solves that approximation problem, with
matching upper and lower bounds for the flux error.  Section~\ref{sec:scalar} transfers these estimates to
the scalar variable.  Section~\ref{sec:aposteriori} develops the a
posteriori estimator, and Section~\ref{sec:numerics} reports numerical
experiments.

\section{The Dirac problem and the singular-field splitting}
\label{sec:splitting}

The splitting itself is not specific to mixed methods; we nevertheless
follow its effect on the mixed variables, that being the discretization
analyzed below.  What is done here
is to rewrite \eqref{eq:deltaeq} as a problem with standard data, by
subtracting from the flux an explicit field whose divergence is the
measure.  The rewritten problem is available to any discretization; its
analysis is whatever theory is available, or has to be developed, for
$W^{-1,p}$ data in divergence form for that method, and
Section~\ref{subsec:scope} says what this means for conforming,
discontinuous Galerkin, finite volume and nonconforming methods.  The mixed
formulation enters only in Section~\ref{sec:mixed-reformulation}.

Let $\O\subset\Rone^2$ be a bounded, simply connected Lipschitz domain,
let $x_0\in\O$, and consider
\begin{equation}\label{eq:deltaeq}
  -\gradt(A\nabla u)=f_1+\delta_{x_0}\ \ \text{in }\O,
  \qquad
  u=g\ \ \text{on }\p\O ,
\end{equation}
where $A\in L^\infty(\O;\Rone^{2\times2})$ is symmetric and uniformly
elliptic,
\begin{equation}\label{eq:ellipticity}
  \alpha|\xi|^2\le\xi^{T}A(x)\xi\le\beta|\xi|^2
  \qquad\forall\,\xi\in\Rone^2,\ \text{a.e. }x\in\O .
\end{equation}
Throughout, $1<p<2$ and $p'=p/(p-1)\in(2,\infty)$, and the standing data
assumptions are $f_1\in L^2(\O)$ and $g\in W^{1-1/p,p}(\p\O)$.
The space $W^{1-1/p,p}(\p\O)$ is the trace space of $W^{1,p}(\O)$, which
is where the potential is sought; since $p<2$ it is strictly larger than
$H^{1/2}(\p\O)$.
As $\O$ is bounded and $p<2$, we also have
$L^2(\O)\hookrightarrow L^p(\O)$, so both exponents are available for
$f_1$.  The flux analysis of Sections~\ref{sec:discrete} and
\ref{sec:approx} and the a posteriori estimate of
Section~\ref{sec:aposteriori} use only $f_1\in L^p(\O)$; the $L^2$
regularity is used in Proposition~\ref{prop:mixed-identities}, where it
fixes the test space of the conservation identity, and in
Section~\ref{subsec:scalar-refined} and Corollary~\ref{cor:scalar-rate},
where the term $\|f_1-Q_hf_1\|_{L^2(\O)}$ appears; it is used once more
in the comparison with the ideal splitting in
Section~\ref{subsec:phiA}.

\subsection{Why the difficulty lies in the formulation}
\label{subsec:whyf2}

Write the physical flux as $\bsigma_{\rm phys}:=-A\nabla u$, so that
$\gradt\bsigma_{\rm phys}=f_1+\delta_{x_0}$.

At the continuous level this pair has no mixed formulation of the usual
kind.  A mixed method puts the flux in a graph space and tests the
conservation equation against the scalar space, which is by design a
Lebesgue space: no derivative is imposed on $u$, and that is the whole
point of the formulation.  But $\delta_{x_0}$ is not represented by a
locally integrable function, so
\begin{equation}\label{eq:delta-not-in-Ls}
  \delta_{x_0}\notin L^{r}(\O)=\bigl(L^{r'}(\O)\bigr)'
  \qquad\text{for every }1<r\le\infty ,
\end{equation}
and therefore $\gradt\bsigma_{\rm phys}\notin L^{r}(\O)$ for any $r$: the physical flux
lies in no graph space $H^{r}(\divvr;\O)$, and the pairing
$\langle\delta_{x_0},v\rangle$ is undefined for a general
$v\in L^{r'}(\O)$.  The measure can be kept only by requiring the
scalar test space to consist of functions possessing a value at $x_0$,
which is no longer a mixed formulation of the kind we wish to
discretize.  On $P_0(\cT_h)$ such a value can be assigned only after
saying how $x_0$ sits in the mesh: it is unambiguous when $x_0$ is
interior to one element, and not otherwise.

A scheme can nevertheless be written, and two standard ways of doing so
act on the datum: one keeps the point functional, which a continuous
space evaluates exactly \cite{Scott1973}, or replaces it by a discrete
representer supported on the element containing $x_0$
\cite{HoustonWihler2012}; the other by a computable
regularization \cite{FuhrerHeuerKarkulik2022, Fuhrer2024}.  Both work.
The first is exact where the point value is available, that is on a
conforming space and on a discontinuous one when $x_0$ lies inside an
element, and then nothing is perturbed at all; the second replaces the
datum and quantifies the perturbation.

One consequence of the first is worth stating, because it bears on the
configuration studied here.  It presupposes that $x_0$ is
interior to an element, since otherwise $v_h(x_0)$ is not defined; if
$x_0$ is taken to be the centroid of an initial element, that position
is maintained under red refinement \cite{BrennerCarstensen2017}, the
middle sub-triangle having the same centroid as its parent.  But when interfaces meet at $x_0$, as in the
example of Section~\ref{subsec:sector}, an element containing $x_0$ in
its interior is cut by them, so the mesh is not fitted to the
coefficient as assumed throughout Section~\ref{sec:approx}.  Fitting the
mesh instead makes $x_0$ a vertex.  The two requirements are in
conflict.

There is a third place to put the difficulty, and it is the one used
here: leave the datum alone and change the unknown.  The next subsection
carries out this divergence-form splitting.  For a Dirac source in two dimensions every exact
divergence lifting lies outside $L^2(\O)^2$; whether the modified flux
does so as well is a separate question, settled by its cancellation
against the constitutive singularity.

\subsection{The splitting and the modified flux}
\label{subsec:splitting-def}

We avoid all of this by moving the measure out of the conservation
equation, following \cite[Sec.~4]{Zhang2023}.  What is available there is
not available here, and the reason is sharp.  The divergence maps
$L^2(\O)^2$ onto $H^{-1}(\O)$ (for $F\in H^{-1}(\O)$ let
$z\in H^1_0(\O)$ solve $-\Delta z=F$, and take $\bff:=-\nabla z$), so
a load in $H^{-1}(\O)$ can always be written as the divergence of a
square-integrable field, and subtracting that field costs nothing.  That
is what is done in \cite[Sec.~4]{Zhang2023}, where a least-squares method
is set up for a load $f_1-\gradt(A\bff_2)$ with $f_1\in L^2(\O)$ and
$\bff_2\in L^2(\O)^2$, at no cost in regularity.  A Dirac measure admits no such field:
$\delta_{x_0}\notin H^{-1}(\O)$ in two dimensions is precisely the
statement that no $\bff\in L^2(\O)^2$ has $\gradt\bff=\delta_{x_0}$.
Whatever field is subtracted must therefore lie outside $L^2(\O)^2$, so
that for the explicit choice below the constitutive load is read through
$L^p$--$L^{p'}$ duality with $1<p<2$, whatever else happens.  That by
itself does not put the modified flux outside $L^2(\O)^2$: what decides
that is how the subtracted singularity meets the constitutive one.  We take the explicit one: let
\begin{equation}\label{eq:f2}
  \bff_2(x):=\frac{1}{2\pi}\frac{x-x_0}{|x-x_0|^2},
  \qquad\text{so that}\qquad
  \gradt\bff_2=\delta_{x_0}\ \text{ in }\mathcal D'(\Rone^2),
\end{equation}
and $\bff_2\in L^p(\O)^2$ for every $1<p<2$.  Define the
\emph{modified flux}
\begin{equation}\label{eq:modified-flux}
  \bsigma:=-A\nabla u-\bff_2 ,
\end{equation}
so that \eqref{eq:deltaeq} is equivalent to the first-order system
\begin{equation}\label{eq:first-order-system}
  A^{-1}\bsigma+\nabla u=-A^{-1}\bff_2,
  \qquad
  \gradt\bsigma=f_1 ,
\end{equation}
in which no Dirac distribution acts on the scalar test space.  We call
\eqref{eq:f2}--\eqref{eq:modified-flux} the \emph{divergence-form
splitting} of the Dirac source: the measure is moved out of the
conservation law into the prescribed field $\bff_2$, and what is left
for the flux unknown is $\gradt\bsigma=f_1$.  What has
been bought is a change of pairing.  The conservation equation in
\eqref{eq:first-order-system} has right-hand side $f_1$ and may be
tested against the full Lebesgue space, as a mixed method requires.  The singular field survives
only in the constitutive equation, where it meets a \emph{flux} test
function, and there H\"older's inequality applies:
\begin{equation}\label{eq:f2-holder}
  \bigl|(A^{-1}\bff_2,\btau)\bigr|
  \le\alpha^{-1}\|\bff_2\|_{L^p(\O)}\|\btau\|_{L^{p'}(\O)}
  \qquad\forall\,\btau\in L^{p'}(\O)^2 .
\end{equation}
No continuity of $\btau$ is required, and no point value of it is ever
taken; \eqref{eq:f2-holder} holds for every $\btau_h\in\RT_0(\cT_h)$
whatever the position of $x_0$ relative to the mesh, and in particular
when $x_0$ is a vertex.  A pairing that needed a test function
continuous at a point has been traded for one that needs only
integrability.  It is worth noting that the two constructions become
admissible at the same threshold and for the same reason: $\bff_2\in
L^p(\O)^2$ and $\delta_{x_0}\in W^{-1,p}(\O)$ both hold throughout
$1<p<2$, the borderline being the failure of $W^{1,2}$ to embed into
$C^0$ in two dimensions.

Eliminating $\bsigma$ from \eqref{eq:first-order-system} returns a
scalar equation,
\begin{equation}\label{eq:divform}
  -\gradt(A\nabla u)=f_1+\gradt\bff_2\ \ \text{in }\O,
  \qquad
  u=g\ \ \text{on }\p\O ,
\end{equation}
which is \eqref{eq:deltaeq} again, with the measure written in
divergence form.  This is the most portable consequence of the
splitting, and it is worth stating separately from the mixed method
that this paper analyzes.  Equation \eqref{eq:divform} is exactly
equivalent to \eqref{eq:deltaeq}: no parameter is introduced, no datum
is perturbed, and $\bff_2$ is given in closed form by \eqref{eq:f2}.
But its right-hand side is a distribution in divergence form generated
by a field of $L^p(\O)^2$, which is the standard shape of data for a
second-order elliptic problem, whereas a Dirac measure is not.  A method
applied to \eqref{eq:divform} therefore needs no theory developed for
measure data; what it needs is a theory for data in $W^{-1,p}$ of
divergence form.  Conforming, discontinuous Galerkin,
finite volume and nonconforming discretizations are all in that
position, and Remark~\ref{rem:other-methods} says what the passage
looks like for each.  The mixed method is one application among them;
it is the one carried through here because $\bsigma$ is the unknown
whose regularity the splitting changes, and the analysis below is
accordingly carried out on \eqref{eq:first-order-system}.  The weak form
of \eqref{eq:divform} is
\begin{equation}\label{eq:divform-weak}
  (A\nabla u,\nabla v)=(f_1,v)-(\bff_2,\nabla v)
  \qquad\forall\,v\in W_0^{1,p'}(\O),
\end{equation}
again with $u\in W^{1,p}(\O)$ and $u|_{\p\O}=g$.  The singular field
meets a gradient here just as it meets a flux in
\eqref{eq:first-order-system}, so the last term is an $L^p$--$L^{p'}$
duality and H\"older's inequality bounds it exactly as in
\eqref{eq:f2-holder}.  The distributional action of $\delta_{x_0}$ is
thereby represented by an $L^p$--$L^{p'}$ pairing, so that no point
evaluation and no measure-valued load is left in the formulation.

\subsection{The ideal splitting}
\label{subsec:phiA}

The field $\bff_2$ is not the only admissible choice in
\eqref{eq:modified-flux}: any field whose divergence is $\delta_{x_0}$
will do, and there is one that is better.  Let $\phi_A$ be a known
particular solution of $-\gradt(A\nabla\phi_A)=\delta_{x_0}$ in $\O$,
no boundary condition being imposed on it, and subtract the field
$-A\nabla\phi_A$, whose divergence is $\delta_{x_0}$.  The modified flux
is then $\bsigma=-A\nabla u+A\nabla\phi_A=-A\nabla w$, where
$w:=u-\phi_A$,
with $-\gradt(A\nabla w)=f_1\in L^2(\O)$.  The Dirac singularity is
removed from the flux as well as from the conservation law; under the
Hilbert-scale boundary regularity $g-\phi_A|_{\p\O}\in H^{1/2}(\p\O)$,
which the standing assumption $g\in W^{1-1/p,p}(\p\O)$ does not supply,
one has $w\in H^1(\O)$ and hence $\bsigma\in H(\divvr;\O)$, the
classical Hilbert-space mixed theory applies verbatim to the pair
$(\bsigma,w)$, and $u_h:=w_h+\phi_A$ is recovered at the end.  None of the analysis of the
present paper would be needed.

When $\phi_A$ is available this is what one should do.  For $A=I$,
$\phi_I(x)=-\tfrac1{2\pi}\log|x-x_0|$; if $A\equiv A_0$ is a constant
symmetric positive definite matrix, then
\begin{equation}\label{eq:frozen}
  \phi_{A_0}(x)=-\frac{1}{2\pi\sqrt{\det A_0}}
  \log\Bigl(\sqrt{(x-x_0)^{T}A_0^{-1}(x-x_0)}\Bigr),
\end{equation}
which for $A_0=\alpha_0I$ reduces to $\alpha_0^{-1}\phi_I$.  If $A=A_0$
only in a neighborhood of $x_0$, the same formula gives not $\phi_A$ but
its principal singular part there, which determines the leading mismatch
considered below.  The corresponding \emph{field} is
\[
  -A_0\nabla\phi_{A_0}(x)
  =\frac{1}{2\pi\sqrt{\det A_0}}\,
   \frac{x-x_0}{(x-x_0)^{T}A_0^{-1}(x-x_0)} ,
\]
which equals $\bff_2$ precisely when $A_0^{-1}=(\det A_0)^{-1/2}I$, that
is, when $A_0$ is a scalar multiple of the identity.  If $A\equiv\alpha_0I$
throughout $\O$ the splitting used in this paper is then already the
ideal one, whatever the value of $\alpha_0$; if $A=\alpha_0I$ only near
$x_0$, what follows is that $\bff_2$ matches the principal singular flux
there.  In either case the normalization of $\bff_2$ needs no
adjustment, because the flux of the fundamental solution carries the
unit mass however the coefficient is scaled.
  If $A$ is discontinuous and $x_0$ lies on a straight
interface separating $A=\alpha_-I$ from $A=\alpha_+I$, a flux balance on
a small circle gives the radially symmetric exact local solution
\begin{equation}\label{eq:interface-phiA}
  \phi_A(x)=-\frac{1}{\pi(\alpha_-+\alpha_+)}\log|x-x_0| ,
\end{equation}
and the two conditions one might compare it against are not the same
one.  The singular \emph{potential} coincides with $\phi_I$ when
$\alpha_-+\alpha_+=2$.  The singular \emph{flux}
$-A\nabla\phi_A=\frac{\alpha_\pm}{\pi(\alpha_-+\alpha_+)}
(x-x_0)/|x-x_0|^2$ coincides with $\bff_2$ on both sides when
$2\alpha_\pm=\alpha_-+\alpha_+$, that is when $\alpha_-=\alpha_+$,
whatever their common value.  For this model it is the second condition
that decides matching, since any nonzero flux mismatch here is
homogeneous of degree $-1$ and so not square-integrable; matching a
singular potential is not matching a singular flux, and
Section~\ref{subsec:regularity} computes the mismatch the first
condition leaves.

Beyond such cases $\phi_A$ is not read off a formula but computed.  It
can still be found when the coefficient depends only on the angle near
$x_0$:
Section~\ref{subsec:sector} does this for an anisotropic coefficient
with two sectors meeting at the pole, where the ansatz
$-\gamma\log r+U(\theta)$ reduces the problem to an ordinary
differential equation in $\theta$ and two constants are then fixed by
periodicity and by the Dirac mass.  Each coefficient requires its own
such computation, and one may or may not wish to carry it out.
For an equation with lower-order terms, that is, for
$-\gradt(A\nabla u+\bb\,u)+cu=f_1+\delta_{x_0}$, no comparably simple
coefficient-dependent closed form is available in general: the
lower-order terms introduce length scales, the ansatz above no longer
closes, and already for $A=I$ with constant $\bb$ and $c$ the
fundamental solution ceases to be elementary.  Producing an ideal
singular field would then amount to solving a separate singular problem,
which is what a coefficient-independent splitting is meant to avoid.

\begin{remark}[What the subtraction changes]
\label{rem:what-subtraction-changes}
It is worth setting out what the choice of subtraction does to the two
unknowns.  There are three
arrangements, all of them producing the same $\RT_0$--$P_0$ matrix and
differing only in the load:

\begin{center}
\begin{tabular}{@{}lll@{}}
\toprule
subtracted from the flux & mixed unknowns
  & Dirac singularity carried by \\
\midrule
$\bff_2$, this paper
  & $(-A\nabla u-\bff_2,\ u)$
  & the potential; the flux too if the\\
  & & mismatch is not square-integrable \\
$-A\nabla\phi_A$, potential kept
  & $(-A\nabla w,\ u)$ & the potential \\
$-A\nabla\phi_A$, potential changed
  & $(-A\nabla w,\ w)$ & neither \\
\bottomrule
\end{tabular}
\end{center}

\noindent
\end{remark}

The purpose of this paper is to provide an analysis that does not rest
on any of this.  The field $\bff_2$ of \eqref{eq:f2} is explicit and
independent of $A$: it is not required to reproduce the singular
structure of the operator, only to carry the correct Dirac mass, and it
is used whether or not $\phi_A$ can be computed.  The gain is largest
for methods whose scalar test space is not continuous at $x_0$: for a
conforming method the two loads coincide, as
Remark~\ref{rem:other-methods} says, and subtracting $\phi_I$ instead
of $\phi_A$ cancels the principal logarithmic singularity exactly when
$\phi_I$ is the principal singular potential of the operator: for a
single constant matrix $A_0$ only for $A_0=I$, and for the interface
coefficient of \eqref{eq:interface-phiA} when $\alpha_-+\alpha_+=2$,
which is not the flux-matching condition $\alpha_-=\alpha_+$.  Matching
the potential is not matching the flux, and it is the flux that the
mixed method carries.  The price is that the
singular structure is not canceled, so that $\bsigma$ need not lie in
$H(\divvr;\O)$, and the whole of the error analysis has to be carried
out in a Lebesgue scale below the Hilbert exponent.  How far below is
the subject of Section~\ref{subsec:regularity}.

\subsection{Regularity of the modified flux}
\label{subsec:regularity}

If $A=\alpha_0I$ in a neighborhood of $x_0$, then by
Section~\ref{subsec:phiA} the field $\bff_2$ matches the principal
singular flux there, and the local difficulty created by the Dirac
source disappears; under the boundary regularity stated in that
subsection the modified flux then lies in $H(\divvr;\O)$, and the
classical mixed theory applies after passing to the formulation in the
regular remainder $w$, the third row of
Remark~\ref{rem:what-subtraction-changes}.  In general the cancellation is only
partial.  What is left is the singular mismatch
$\bq_{\rm mis}:=-A\nabla\phi_A-\bff_2$: writing $u=\phi_A+w$ as above,
\begin{equation}\label{eq:sigma-split}
  \bsigma=-A\nabla u-\bff_2=\bq_{\rm mis}-A\nabla w ,
\end{equation}
the second term being the flux of the remainder.  In the interface
situation \eqref{eq:interface-phiA} with $\alpha_-\ne\alpha_+$ one
computes explicitly
\[
  \bq_{\rm mis}(x)=\Bigl(\frac{\alpha(x)}{\pi(\alpha_-+\alpha_+)}
                    -\frac{1}{2\pi}\Bigr)
  \frac{x-x_0}{|x-x_0|^2},
\]
which is nonzero on both sides of the interface.  A homogeneous
$|x-x_0|^{-1}$ field is not square-integrable in two dimensions and
cannot be cancelled by a remainder whose flux is, so whenever
$A\nabla w\in L^2(\O)^2$, in particular in the model problems of
Section~\ref{sec:numerics}, where $f_1=0$ and the datum is chosen so
that $w=0$ and $\bsigma=\bq_{\rm mis}$,
\begin{equation}\label{eq:sigma-Lp-only}
  \bsigma\in L^p(\O)^2\quad\forall\,1<p<2,
  \qquad
  \bsigma\notin L^2(\O)^2 .
\end{equation}
The standard $H(\divvr;\O)$ framework is therefore unavailable, while
the divergence constraint $\gradt\bsigma=f_1$ remains perfectly regular.
This is the situation the paper is about.

The example also fixes the correct form of the criterion.  Here $A$ is a
scalar matrix on each side of the interface, so no value of $A$ at $x_0$
is a non-scalar matrix, and for a discontinuous coefficient $A(x_0)$
need not be a single matrix at all, yet $\bsigma\notin L^2(\O)^2$.
What decides the question is the modified flux itself.  We call the
splitting \emph{matched} when the subtraction returns $\bsigma$ to the
Hilbert scale, that is when $\bsigma$ is square-integrable near $x_0$;
this refers to $\bsigma=-A\nabla u-\bff_2$ alone and to no choice of
$\phi_A$.  For a coefficient which is a single constant matrix $A_0$
near $x_0$ it reduces, by Section~\ref{subsec:phiA}, to the requirement
that $A_0$ be a scalar multiple of the identity, while
\eqref{eq:interface-phiA} shows that the reduction is genuinely local to
that case, the condition there coupling the traces from the two sides.
A nonzero homogeneous $|x-x_0|^{-1}$ part of $\bq_{\rm mis}$ therefore
prevents matching.  Its vanishing is necessary, and is not sufficient
without an $L^2$ bound on the mismatch that remains.  The question does
not arise in the setting of \cite[Sec.~4]{Zhang2023}: there the
lifting used in the formulation is in $L^2(\O)^2$, so it leaves $\bsigma$
square-integrable and the splitting is matched.

A scalar coefficient that is Lipschitz in a neighborhood of $x_0$ is
matched, and the
calculation is short.  Put $\alpha_0:=\alpha(x_0)$ and
$\phi_0:=\alpha_0^{-1}\phi_I$, so that $-\nabla\phi_I=\bff_2$.  The
mismatch left by that singular field is
\begin{equation}\label{eq:lipschitz-mismatch}
  -\alpha\nabla\phi_0-\bff_2=\frac{\alpha-\alpha_0}{\alpha_0}\,\bff_2 ,
\end{equation}
which is bounded near $x_0$ because $|\alpha(x)-\alpha_0|\le C|x-x_0|$
while $|\bff_2(x)|=\frac1{2\pi}|x-x_0|^{-1}$; in particular it has no
homogeneous $|x-x_0|^{-1}$ part.  Writing $u=\phi_0+w$ gives
$\bsigma=(-\alpha\nabla\phi_0-\bff_2)-\alpha\nabla w$ with
$-\gradt(\alpha\nabla w)=f_1-\gradt(-\alpha\nabla\phi_0-\bff_2)$, whose
data lie in $H^{-1}$ near the pole, since
\eqref{eq:lipschitz-mismatch} is in $L^2_{\rm loc}$.  Standard local
Hilbert-scale regularity for a uniformly elliptic scalar coefficient
that is Lipschitz in a neighborhood of $x_0$ then gives
$w\in H^1_{\rm loc}$, hence $\alpha\nabla w\in L^2_{\rm loc}$, so
$\bsigma$ is square-integrable at the pole.  Constant $\alpha$ is the
special case in which \eqref{eq:lipschitz-mismatch} vanishes.

\subsection{Other discretizations}
\label{subsec:scope}

\begin{remark}[Other discretizations]
\label{rem:other-methods}
Neither \eqref{eq:first-order-system} nor \eqref{eq:divform} is tied to
the method used below.  Both are exactly equivalent to
\eqref{eq:deltaeq}, in both the measure has been traded for the field
$\bff_2\in L^p(\O)^2$, and a discretization may be built on whichever of
the two suits it: on the first-order system if it carries a flux
unknown, on the scalar equation if it does not.  The gain is largest for
methods whose scalar test space fails to be continuous at $x_0$, which
is where the difficulty stated in \eqref{eq:delta-not-in-Ls} is felt:
discontinuous Galerkin and hybridizable DG methods, finite volume
schemes, and nonconforming elements such as Crouzeix--Raviart, whose
functions are single-valued at edge midpoints but not at vertices.  In
each case the change is to the right-hand side only: the bilinear form,
the numerical fluxes and the penalty terms are untouched, so in an
existing code only the load assembly changes.

For a primal discontinuous Galerkin method, for instance, testing
$\gradt\bsigma=f_1$ elementwise against $v_h$ and inserting
$\bsigma=-A\nabla u-\bff_2$ replaces the formal point-source load
$(f_1,v_h)+v_h(x_0)$, whose second term is ambiguous for a
discontinuous $v_h$, by the well-defined
\begin{equation}\label{eq:dg-load}
  (f_1,v_h)-\sum_{K\in\cT_h}(\bff_2,\nabla v_h)_K
  +\sum_{e\in\cE_h}\bigl\langle\bff_2\cdot\bn_e,\jump{v_h}\bigr\rangle_e ,
\end{equation}
where $\bn_e$ is a fixed unit normal on each edge and
$\jump{v_h}:=v_h^+-v_h^-$ with $v_h^\pm$ the traces from the two sides,
$\bn_e$ pointing from $+$ to $-$; on a boundary edge $\jump{v_h}:=v_h$
and $\bn_e$ is outward.  This is \eqref{eq:divform-weak} with the
integration by parts performed elementwise.  Every term is finite: the volume integrals by \eqref{eq:f2-holder}
applied elementwise, and the edge integrals because $\bff_2\cdot\bn_e$
vanishes identically on any straight edge containing $x_0$ (there
$x-x_0$ is parallel to the edge, so $\bff_2$ is tangential to it) and
is bounded on every other edge.  The difficulty is therefore not
relocated from the elements to the skeleton.

Two limits of the claim should be stated.  First, for a
\emph{conforming} method the two right-hand sides give the same discrete
load: if $v_h\in W^{1,p'}_0(\O)$ then the jumps in \eqref{eq:dg-load}
vanish and $-\int_\O\bff_2\cdot\nabla v_h=\langle\gradt\bff_2,v_h\rangle
=v_h(x_0)$, so one recovers the classical load vector exactly.  Nothing
changes in the assembly, and the point evaluation was legitimate there
to begin with, $v_h$ being continuous.  What \eqref{eq:divform-weak}
still supplies is the $L^p$--$L^{p'}$ duality in place of the
$W^{-1,p}$--$W_0^{1,p'}$ one, so that the analysis may be run with a
datum in divergence form and not with a measure; whether that is
worth doing depends on the norm in which the conforming method is
analyzed.  Second, and more
substantially, the passage to \eqref{eq:divform} repairs the
formulation, not the regularity.  The splitting does not change $u$, so
whatever point singularity it has remains, logarithmic, with a
gradient of order $|x-x_0|^{-1}$, in the model configurations analyzed
below, so that the approximation difficulties the point singularity
creates remain, although each of these methods would need an error
analysis of its own, in a norm adapted to it and with its own
approximation mechanism.

\end{remark}


\section{Mixed formulation of the Dirac problem}
\label{sec:mixed-reformulation}
We now set out the mixed formulations used in the finite element
analysis.
Following Douglas and Roberts \cite{DouglasRoberts1985}, we assume that
\eqref{eq:deltaeq} has a unique solution and derive the mixed
formulation from it, as a reformulation of the original elliptic
problem and not as a first-order system with a well-posedness theory
of its own; no continuous inf--sup analysis is needed below.

For \(1<r<\infty\) we write
\begin{equation}\label{eq:Hr-div}
\begin{aligned}
H^{r}(\divvr;\O)&:=\{\btau\in L^{r}(\O)^2:\gradt\btau\in L^{r}(\O)\},\\
\|\btau\|_{H^{r}(\divvr;\O)}
 &:=\|\btau\|_{L^{r}(\O)}+\|\gradt\btau\|_{L^{r}(\O)} ;
\end{aligned}
\end{equation}
the constitutive identity below uses \(r=p'\), while
Sections~\ref{sec:approx} and \ref{sec:aposteriori} use \(r=p\).
On a bounded Lipschitz domain, the normal trace extends continuously from \(H^{p'}(\divvr;\O)\) to \(W^{-1/p',p'}(\p\O)\), and Green's formula
\begin{equation}\label{eq:green-Lp}
\langle\btau\cdot\bn,\varphi\rangle_{\p\O}=(\btau,\nabla\varphi)+(\gradt\btau,\varphi)
\end{equation}
holds for \(\btau\in H^{p'}(\divvr;\O)\) and
\(\varphi\in W^{1,p}(\O)\); see
\cite[Sec.~1.2.6]{BernardiEtAl2024}.  The trace space of
\(W^{1,p}(\O)\) is \(W^{1-1/p,p}(\p\O)=W^{1/p',p}(\p\O)\), so the
boundary pairing below is defined.

\begin{prop}[Mixed formulations]
\label{prop:mixed-identities}
Let \(u\in W^{1,p}(\O)\), \(u|_{\p\O}=g\), be the assumed solution of
\eqref{eq:deltaeq}, and set \(\bsigma:=-A\nabla u-\bff_2\), where
\(\gradt\bff_2=\delta_{x_0}\).  Then \(\bsigma\in L^p(\O)^2\) and
\(\gradt\bsigma=f_1\in L^2(\O)\), so that in particular \(\bsigma\in H^{p}(\divvr;\O)\), and
\begin{equation}\label{eq:mixed-continuous}
\begin{aligned}
(A^{-1}\bsigma,\btau)-(\gradt\btau,u)&=-(A^{-1}\bff_2,\btau)-\langle\btau\cdot\bn,g\rangle_{\p\O}
&&\forall\,\btau\in H^{p'}(\divvr;\O),\\
(\gradt\bsigma,v)&=(f_1,v)
&&\forall\,v\in L^{2}(\O).
\end{aligned}
\end{equation}
Conversely, let \(u\in W^{1,p}(\O)\) with \(u|_{\p\O}=g\) and \(\bsigma\in L^p(\O)^2\) with \(\gradt\bsigma\in L^2(\O)\). If \((\bsigma,u)\) satisfies \eqref{eq:mixed-continuous}, then \(u\) solves \eqref{eq:deltaeq} in the sense of distributions.
\end{prop}

\begin{remark}[Where $f_1\in L^2(\O)$ is used]
\label{rem:f1-L2}
The standing $L^2$ hypothesis enters \eqref{eq:mixed-continuous} only
through the test space of the second identity.  If $f_1$ is merely in
$L^p(\O)$, the same identity holds for every $v\in L^{p'}(\O)$, and
Sections~\ref{sec:discrete}--\ref{sec:approx} and
Section~\ref{sec:aposteriori} go through unchanged; the $L^2$
regularity is needed again only in Section~\ref{subsec:scalar-refined}.
\end{remark}

\begin{proof}
Since \(A^{-1}\bsigma=-\nabla u-A^{-1}\bff_2\) with \(\nabla u\in L^p(\O)^2\)
and \(\bff_2\in L^p(\O)^2\), we have \(\bsigma\in L^p(\O)^2\).  Its
distributional divergence is
\(\gradt\bsigma=-\gradt(A\nabla u)-\gradt\bff_2
=(f_1+\delta_{x_0})-\delta_{x_0}=f_1\), the two measures canceling.  Since \(f_1\in L^p(\O)\), the divergence of
\(\bsigma\) is represented by an \(L^p\) function, and
\(\bsigma\in H^{p}(\divvr;\O)\) by \eqref{eq:Hr-div}.  Apart from the
converse statement, this is the only step in which a distributional
identity occurs: from here on both \(\bsigma\) and \(\gradt\bsigma\) are
functions.

Pairing the constitutive identity with \(\btau\in H^{p'}(\divvr;\O)\) and
applying \eqref{eq:green-Lp} with \(\varphi=u\) gives the first identity
in \eqref{eq:mixed-continuous}.  The second is \(\gradt\bsigma=f_1\)
tested against \(v\in L^2(\O)\); since \(p'>2\) and \(\O\) is bounded,
\(L^{p'}(\O)\subset L^2(\O)\), so it holds in particular for every
\(v\in L^{p'}(\O)\).

Conversely, testing the first identity with \(\btau\in\mathcal D(\O)^2\),
for which the boundary term vanishes, gives
\(A^{-1}\bsigma+\nabla u=-A^{-1}\bff_2\) almost everywhere in \(\O\), all
three fields lying in \(L^p(\O)^2\), while the second gives
\(\gradt\bsigma=f_1\).  Hence
\(-\gradt(A\nabla u)=\gradt(\bsigma+\bff_2)=f_1+\delta_{x_0}\) in
\(\mathcal D'(\O)\), and the boundary condition is contained in the
prescribed trace of \(u\).
\end{proof}

\section{The mixed finite element method and the flux error in
         \texorpdfstring{$L^p$}{Lp}}
\label{sec:discrete}

Let $\cT_h$ be a shape-regular triangulation of $\O$. For each
$K\in\cT_h$, let
\[
\begin{aligned}
P_0(K)&:=\{v:\ v\text{ is constant on }K\},\\
\RT_0(K)&:=P_0(K)^2+\bx P_0(K)
=\{\ba+b\bx:\ \ba\in\Rone^2,\ b\in\Rone\}.
\end{aligned}
\]
The global spaces of the lowest-order Raviart--Thomas pair
\cite{RaviartThomas1977} (see also \cite[Ch.~III]{BrezziFortin1991})
are
\begin{eqnarray*}
\RT_0(\cT_h)&:=&\{\btau_h\in H(\divvr;\O):\ \btau_h|_K\in\RT_0(K)
\ \forall K\in\cT_h\},\\
P_0(\cT_h)&:=&\{v_h\in L^2(\O):\ v_h|_K\in P_0(K)
\ \forall K\in\cT_h\}.
\end{eqnarray*}
We define $Q_hf\in P_0(\cT_h)$ by the elementwise mean
$(Q_hf)|_K:=\frac1{|K|}\int_K f\,dx$.%
This definition is meaningful for every $f\in L^1(\O)$, and in
particular under the weaker hypothesis $f_1\in L^p(\O)$ to which the
flux analysis of this section and of Section~\ref{sec:approx} is
confined; under the standing assumption $f_1\in L^2(\O)$ it coincides,
in the present lowest-order setting, with the usual $L^2$-orthogonal
projection onto $P_0(\cT_h)$.

We denote by $\Pi_h^{\rm rt}$ the canonical Raviart--Thomas interpolation
operator. Whenever the required normal moments are well defined,
$\Pi_h^{\rm rt}\btau\in\RT_0(\cT_h)$ is determined by
$\int_e\Pi_h^{\rm rt}\btau\cdot\bn_e\,ds=\int_e\btau\cdot\bn_e\,ds$ for
every edge $e$, and satisfies $\gradt\Pi_h^{\rm rt}\btau=Q_h(\gradt\btau)$.

\subsection{The mixed discrete problem}
\label{subsec:discrete-problem}

The finite element approximation is: find
$(\bsigma_h,u_h)\in\RT_0(\cT_h)\times P_0(\cT_h)$ such that
\begin{equation}\label{eq:mixed-discrete}
  \begin{aligned}
    (A^{-1}\bsigma_h,\btau_h)-(\gradt\btau_h,u_h)
      &=-(A^{-1}\bff_2,\btau_h)-\langle\btau_h\cdot\bn,g\rangle_{\p\O}
      &&\forall\,\btau_h\in\RT_0(\cT_h),\\
    (\gradt\bsigma_h,v_h)&=(f_1,v_h)
      &&\forall\,v_h\in P_0(\cT_h) .
  \end{aligned}
\end{equation}
The singular volume term $(A^{-1}\bff_2,\btau_h)$ is well defined because $\bff_2\in L^p(\O)^2\subset L^1(\O)^2$ and every $\btau_h\in\RT_0(\cT_h)$ is bounded elementwise. When $A$ is elementwise constant, this integral can be evaluated exactly; for a general coefficient it can be computed by a singular quadrature based on the Duffy transformation \cite{Duffy1982}.
More importantly, \eqref{eq:mixed-discrete} is
algebraically the standard $\RT_0$--$P_0$ mixed method: the exponent
$p$ enters neither the matrix nor its solvability, but only the
regularity of the exact flux and the norms used to measure the error.

\begin{prop}[Standard discrete well-posedness]
\label{prop:discrete-wellposed}
For every triangulation $\cT_h$, problem \eqref{eq:mixed-discrete} has
a unique solution, and
\begin{equation}\label{eq:discrete-conservation}
  \gradt\bsigma_h=Q_hf_1
\end{equation}
elementwise.
\end{prop}

The proof is standard and the second equation is
\eqref{eq:discrete-conservation} because $\gradt(\RT_0(\cT_h))=P_0(\cT_h)$. %

Subtracting \eqref{eq:mixed-discrete} from
\eqref{eq:mixed-continuous} gives the error equations
\begin{equation}\label{eq:mixed-error-equations}
  \begin{aligned}
    (A^{-1}(\bsigma-\bsigma_h),\btau_h)-(\gradt\btau_h,u-u_h)&=0
      &&\forall\,\btau_h\in\RT_0(\cT_h),\\
    (\gradt(\bsigma-\bsigma_h),v_h)&=0
      &&\forall\, v_h\in P_0(\cT_h) .
  \end{aligned}
\end{equation}
The data have canceled: both right-hand sides of
\eqref{eq:mixed-discrete} are exactly the restrictions to
$\RT_0(\cT_h)\times P_0(\cT_h)$ of those in \eqref{eq:mixed-continuous}.  In
particular $\gradt(\bsigma-\bsigma_h)=f_1-Q_hf_1$.
Define
\[
\begin{aligned}
  K_h&:=\{\btau_h\in\RT_0(\cT_h):\gradt\btau_h=0\},\\
  \RT_0^{f_1}(\cT_h)&:=\{\btau_h\in\RT_0(\cT_h):\gradt\btau_h=Q_hf_1\} .
\end{aligned}
\]
The affine set $\RT_0^{f_1}(\cT_h)$ is nonempty because
$\gradt(\RT_0(\cT_h))=P_0(\cT_h)$, and
$\bsigma_h\in\RT_0^{f_1}(\cT_h)$ by
\eqref{eq:discrete-conservation}.  On $K_h$ every graph norm reduces to
a Lebesgue norm, since the divergence vanishes identically there; this
is used repeatedly below.  Restricting the first identity in
\eqref{eq:mixed-error-equations} to $\btau_h\in K_h$ removes the second
term and gives the kernel orthogonality
\begin{equation}\label{eq:kernel-orthogonality}
  (A^{-1}(\bsigma-\bsigma_h),\btau_h)=0\qquad\forall\,\btau_h\in K_h .
\end{equation}

\subsection{The solenoidal projection and the flux error}
\label{subsec:duran}

For $\brho\in L^p(\O)^2$ define $R_h\brho\in K_h$ by
\begin{equation}\label{eq:Rh-def}
  (A^{-1}R_h\brho,\btau_h)=(A^{-1}\brho,\btau_h)
  \qquad\forall\,\btau_h\in K_h .
\end{equation}
For each fixed mesh $R_h$ is well defined, since
$(A^{-1}\cdot,\cdot)$ is positive definite on the finite-dimensional
space $K_h$.  The analytic input of the whole $L^p$-error analysis is an
$h$-uniform bound on $R_h$, which we state in two forms.  The second
measures the divergence of the argument in a negative norm, so we recall
that, for \(1<q<\infty\) and \(q'=q/(q-1)\), \(W^{-1,q}(\O)\) is the
dual of \(W^{1,q'}_0(\O)\), with
\beq \label{def:W-1norm-q}
  \|F\|_{W^{-1,q}(\O)}
  :=\sup_{0\neq\chi\in W^{1,q'}_0(\O)}
   \frac{\langle F,\chi\rangle}{\|\nabla\chi\|_{L^{q'}(\O)}} ,
\eeq
the gradient norm being equivalent to the full \(W^{1,q'}\) norm on
\(W^{1,q'}_0(\O)\) by Poincar\'e's inequality.  For
\(\bv\in L^q(\O)^2\) the divergence \(\gradt\bv\) is defined as a
distribution by \(\langle\gradt\bv,\chi\rangle
=-\int_\O\bv\cdot\nabla\chi\,dx\) for \(\chi\in C^\infty_c(\O)\),
whence, by H\"older's inequality,
$ |\langle\gradt\bv,\chi\rangle|
  \le\|\bv\|_{L^q(\O)}\,\|\nabla\chi\|_{L^{q'}(\O)}$.
Since \(C^\infty_c(\O)\) is dense in \(W^{1,q'}_0(\O)\), the functional
\(\gradt\bv\) extends to that space with the same bound, so that
\begin{equation}\label{eq:div-Wm1q}
  \|\gradt\bv\|_{W^{-1,q}(\O)}\le\|\bv\|_{L^q(\O)}
  \qquad\forall\,\bv\in L^q(\O)^2 .
\end{equation}
This requires nothing of \(\bv\) beyond \(\bv\in L^q(\O)^2\); in
particular \(\gradt\bv\) need not be a function.  Both are used here at
\(q=p\) and again in Section~\ref{sec:aposteriori} at a general
exponent.

\begin{quote}
\textbf{Hypothesis (SP) (solenoidal projection).}
\emph{There is $C_{\mathrm{SP}}>0$, independent of $h$, such that}
\begin{equation}\label{eq:Duran-stability}
  \|R_h\brho\|_{L^p(\O)}\le C_{\mathrm{SP}}\|\brho\|_{L^p(\O)}
  \qquad\forall\,\brho\in L^p(\O)^2 .
\end{equation}

\medskip
\textbf{Hypothesis (SP$'$) (two-term form).}
\emph{There are $C_0>0$ and $C_{\mathrm{SP}}'\ge0$, independent of $h$,
such that}
\begin{equation}\label{eq:Duran-two-term}
  \|R_h\brho\|_{L^p(\O)}
  \le C_0\|\brho\|_{L^p(\O)}
   +C_{\mathrm{SP}}'\|\gradt\brho\|_{W^{-1,p}(\O)}
  \qquad\forall\,\brho\in L^p(\O)^2 .
\end{equation}
\end{quote}

By \eqref{eq:div-Wm1q} at $q=p$, (SP$'$) implies (SP) with
$C_{\mathrm{SP}}=C_0+C_{\mathrm{SP}}'$, and (SP) is the case
$C_{\mathrm{SP}}'=0$ of (SP$'$); as existence statements the two are
therefore the same, and what distinguishes them is the size of the
constants.  It is the two-term form that Dur\'an proves, with a constant
$C_0$ that does not degenerate as $p\downarrow1$ while
$C_{\mathrm{SP}}'$ does, and the distinction survives in the error
bounds: Theorem~\ref{thm:flux-quasibest} carries $C_{\mathrm{SP}}$ in
front of the best-approximation term and
Theorem~\ref{thm:flux-two-term} carries only $C_0$, the remainder being
moved into a data-oscillation term of one higher order in $h$.  Both
hypotheses are discussed in Remark~\ref{rem:Duran-scope} below.

\begin{thm}[Equilibrated $L^p$ quasi-best approximation]
\label{thm:flux-quasibest}
Assume that Hypothesis~(SP) holds.  Then
\begin{equation}\label{eq:flux-quasibest}
  \|\bsigma-\bsigma_h\|_{L^p(\O)}
  \le(1+C_{\mathrm{SP}})
  \inf_{\btau_h\in\RT_0^{f_1}(\cT_h)}\|\bsigma-\btau_h\|_{L^p(\O)} .
\end{equation}
\end{thm}

\begin{proof}
Let $\bsigma_I\in\RT_0^{f_1}(\cT_h)$ be arbitrary.  Then
$\bsigma_h-\bsigma_I\in K_h$, and by
\eqref{eq:kernel-orthogonality},
$(A^{-1}(\bsigma_h-\bsigma_I),\btau_h)=(A^{-1}(\bsigma-\bsigma_I),\btau_h)$
for every $\btau_h\in K_h$.  By the definition \eqref{eq:Rh-def} of $R_h$ this says exactly
\begin{equation}\label{eq:error-representation}
  \bsigma_h-\bsigma_I=R_h(\bsigma-\bsigma_I),
  \qquad\text{hence}\qquad
  \bsigma-\bsigma_h=(I-R_h)(\bsigma-\bsigma_I).
\end{equation}
Therefore
$\|\bsigma-\bsigma_h\|_{L^p}\le(1+C_{\mathrm{SP}})\|\bsigma-\bsigma_I\|_{L^p}$ by
\eqref{eq:Duran-stability}, and taking the infimum over
$\bsigma_I\in\RT_0^{f_1}(\cT_h)$ proves \eqref{eq:flux-quasibest}.
\end{proof}

Under (SP$'$) the same representation gives a two-term bound.  The point
is that the projection is applied in \eqref{eq:error-representation} to
a field whose divergence is $f_1-Q_hf_1$, which has vanishing mean on
every element, so that its negative norm is one power of $h$ smaller
than its $L^p$ norm.  The following lemma is standard; we state it here
for completeness.

\begin{lem}[Elementwise oscillation in the negative norm]
\label{lem:osc-negative}
Let $1<q<\infty$ and let $g\in L^q(\O)$ satisfy $\int_Kg\,dx=0$ for
every $K\in\cT_h$.  Then
\begin{equation}\label{eq:osc-negative}
  \|g\|_{W^{-1,q}(\O)}
  \le C_{\rm P}\Bigl(\sum_{K\in\cT_h}h_K^{q}\|g\|_{L^q(K)}^{q}\Bigr)^{1/q}
  \le C_{\rm P}\,h\,\|g\|_{L^q(\O)} ,
\end{equation}
where $h:=\max_{K\in\cT_h}h_K$ and the Poincar\'e constant $C_{\rm P}$
depends only on $q$ and on the shape regularity of $\cT_h$.
\end{lem}

\begin{proof}
Let $\chi\in W^{1,q'}_0(\O)$.  Since $g$ has vanishing mean on each
element, $\langle g,\chi\rangle=\int_\O g\,(\chi-Q_h\chi)\,dx$.  By the
elementwise H\"older and Poincar\'e inequalities
$\|\chi-Q_h\chi\|_{L^{q'}(K)}\le C_{\rm P}h_K\|\nabla\chi\|_{L^{q'}(K)}$,
followed by H\"older's inequality for sums,
\[
  \begin{aligned}
  |\langle g,\chi\rangle|
  &\le C_{\rm P}\sum_{K\in\cT_h}h_K\|g\|_{L^q(K)}\|\nabla\chi\|_{L^{q'}(K)}\\
  &\le C_{\rm P}\Bigl(\sum_{K\in\cT_h}h_K^{q}\|g\|_{L^q(K)}^{q}\Bigr)^{1/q}
   \|\nabla\chi\|_{L^{q'}(\O)} .
  \end{aligned}
\]
Dividing by $\|\nabla\chi\|_{L^{q'}(\O)}$ and taking the supremum gives
the first bound in \eqref{eq:osc-negative}; the second follows from
$h_K\le h$.
\end{proof}

\begin{thm}[Two-term $L^p$ flux estimate]
\label{thm:flux-two-term}
Assume Hypothesis~(SP$'$), \eqref{eq:Duran-two-term}, and
$f_1\in L^p(\O)$.  Then
\begin{equation}\label{eq:flux-two-term}
  \|\bsigma-\bsigma_h\|_{L^p(\O)}
  \ \le\ (1+C_0)\!\!\inf_{\btau_h\in\RT_0^{f_1}(\cT_h)}\!\!
     \|\bsigma-\btau_h\|_{L^p(\O)}
   +C_{\rm P}\,C_{\mathrm{SP}}'\,h\,\|f_1-Q_hf_1\|_{L^p(\O)} ,
\end{equation}
with the Poincar\'e constant $C_{\rm P}$ of Lemma~\ref{lem:osc-negative}.
\end{thm}

\begin{proof}
Let $\bsigma_I\in\RT_0^{f_1}(\cT_h)$.  By \eqref{eq:error-representation}
and \eqref{eq:Duran-two-term},
\[
  \begin{aligned}
  \|\bsigma-\bsigma_h\|_{L^p(\O)}
  &\le\|\bsigma-\bsigma_I\|_{L^p(\O)}+\|R_h(\bsigma-\bsigma_I)\|_{L^p(\O)}\\
  &\le(1+C_0)\|\bsigma-\bsigma_I\|_{L^p(\O)}
   +C_{\mathrm{SP}}'\|\gradt(\bsigma-\bsigma_I)\|_{W^{-1,p}(\O)} .
  \end{aligned}
\]
Here $\gradt\bsigma=f_1$ and $\gradt\bsigma_I=Q_hf_1$, so
$\gradt(\bsigma-\bsigma_I)=f_1-Q_hf_1$ has vanishing mean on every
element, and Lemma~\ref{lem:osc-negative} with $q=p$ bounds its negative
norm by $C_{\rm P}h\|f_1-Q_hf_1\|_{L^p(\O)}$.  This last term does not depend on
$\bsigma_I$, so the infimum may be taken in the first one, which gives
\eqref{eq:flux-two-term}.
\end{proof}

\begin{remark}[Scope of (SP) and (SP$'$)]
\label{rem:Duran-scope}
Estimate \eqref{eq:Duran-two-term} is, as far as we are aware, the only
$L^p$-stability result available for the solenoidal projection of a
mixed method.  We set out its scope in detail, and isolate it as a
hypothesis so that this is visible and not buried in the proofs.

For \(1<p<2\), \eqref{eq:Duran-two-term} is known in two dimensions
under the assumptions of Dur\'an \cite[Thm.~3.1, Cor.~3.1 and
Rem.~3.5]{Duran1988}: a smooth simply connected domain, quasi-uniform
triangulations, and a Lipschitz scalar coefficient bounded away from
zero.  It is stated there in Theorem~3.1 for \(2\le p<\infty\) and
extended to \(1<p\le2\) in Remark~3.5, the constant
\(C_{\mathrm{SP}}'\) degenerating as \(p\downarrow1\) while \(C_0\) does
not; the passage to \eqref{eq:Duran-stability} through
\eqref{eq:div-Wm1q} is Corollary~3.1 there.

Two features of that setting are not incidental.  The argument reduces
\eqref{eq:Duran-stability} to the \(W^{1,p}\)-stability of a scalar
Ritz projection by means of a stream function, and this reduction is
available only in two dimensions; \cite{Duran1988} says so explicitly,
and \cite[p.~105]{GastaldiNochetto1989} says that the technique does
not extend beyond two dimensions.  Moreover, for the Dirichlet problem
\eqref{eq:deltaeq} the scalar problem so produced carries
\emph{Neumann} boundary conditions, and the required
\(W^{1,p}\)-stability is obtained in \cite{Duran1988} by asserting that
the proofs of Nitsche \cite{Nitsche1975} and of Rannacher--Scott
\cite{RannacherScott1982}, which are for the Dirichlet problem, carry
over with minor modifications, an assertion whose safety rests on
the smoothness of \(\O\) and on the Lipschitz continuity of the
coefficient.  These analytic assumptions already fail in the setting of
interest here.
We are not aware of any result giving
\eqref{eq:Duran-stability} on a polygon, on general adaptive meshes, or
for the discontinuous and anisotropic coefficients considered here.
Outside Dur\'an's setting we keep it as a hypothesis.
\end{remark}

\begin{remark}[Relation to the Hilbert case]
\label{rem:Duran-role}
Hypothesis~(SP) is used only for the \(L^p\)-error analysis and is not needed for the solvability of the discrete mixed problem. When \(p=2\), \(R_h\) is the \(A^{-1}\)-orthogonal projection onto \(K_h\), so (SP) is automatic and the flux estimate reduces to the classical best-approximation identity
\[
\|A^{-1/2}(\bsigma-\bsigma_h)\|_{0,\O}
=\inf_{\btau_h\in\RT_0^{f_1}(\cT_h)}
\|A^{-1/2}(\bsigma-\btau_h)\|_{0,\O}.
\]
Nothing is needed for this beyond \eqref{eq:kernel-orthogonality},
which at \(p=2\) is an orthogonality in the \(A^{-1}\) inner product,
together with \(\bsigma_h\in\RT_0^{f_1}(\cT_h)\), which is what makes the
infimum attained.  The corresponding inequality, with the infimum over
the \emph{equilibrated} set and with a constant independent of the
coefficient, is Theorem~2 of \cite{Zhang:20mixed}, and Remark~5 there
for a matrix coefficient as here.
Theorem~\ref{thm:flux-quasibest} is its counterpart in \(L^p\), with the
equality weakened to a bound with constant \(1+C_{\mathrm{SP}}\):
Hypothesis~(SP) is precisely what replaces the orthogonality that makes
the case \(p=2\) an identity.

\end{remark}

\begin{remark}[The two estimates compared]
\label{rem:Duran-forms}
Theorem~\ref{thm:flux-quasibest} is a quasi-best approximation bound in
the strict sense: one term, and no reference to the data.
Theorem~\ref{thm:flux-two-term} is not, but it separates the two
constants, and the separation is the useful one.  All the
\(p\)-dependence that degenerates as \(p\downarrow1\) is carried by
\(C_{\mathrm{SP}}'\), which multiplies only the oscillation; the
constant in front of the approximation term is \(1+C_0\), which stays
bounded.  The oscillation is also of one higher order in \(h\): on quasi-uniform
meshes it is \(O(h^2)\) if \(f_1|_K\in W^{1,p}(K)\) on every element,
whereas Corollary~\ref{cor:global-rate} bounds the flux error by
\(Ch^{2/p-1}\) with \(2/p-1<1\), a rate that is sharp under the
nonvanishing leading-order hypothesis of
Corollary~\ref{cor:global-lower}.  It vanishes
altogether when
\(f_1\) is elementwise constant, in particular for the pure Dirac
problem \(f_1\equiv0\).  We use \eqref{eq:flux-quasibest} in what
follows, since the rates of Section~\ref{sec:approx} are unaffected by
which of the two is taken; \eqref{eq:flux-two-term} is the sharper
statement when \(p\) is close to \(1\).

\end{remark}

%
%
%
%
%
%
%
%
%
%

\section{Approximation of the singular flux}
\label{sec:approx}

Theorem~\ref{thm:flux-quasibest} reduces the flux analysis to the
constrained approximation problem
$\inf_{\btau_h\in\RT_0^{f_1}(\cT_h)}\|\bsigma-\btau_h\|_{L^p(\O)}$,
which is the subject of this section.

Two things stand in the way.
The first is that on the element carrying the pole the error is
$h_{K^{x_0}}^{2/p-1}$, that element's diameter to the power $2/p-1$, for
every field of $\RT_0(\cT_h)$, provided the
leading $|x-x_0|^{-1}$ part of $\bsigma$ does not vanish there.  This is a property of
the singularity and of the local Raviart--Thomas space, not of
the method, it holds whatever the coefficient, and no construction
removes it; it is stated for a pole element that is, after translation
and scaling, of a fixed shape (Section~\ref{subsec:sharpness}).  It
applies to the canonical interpolant like anything else, and the radial
interface flux of Section~\ref{subsec:regularity}, with $x_0$ a vertex
of the fitted mesh, gives a coefficient for which $\Pi_h^{\rm rt}\bsigma$
is perfectly well defined and vanishes on the pole element, so that it
misses the singular field entirely.
The second is that $\Pi_h^{\rm rt}\bsigma$ may not be defined at all: in
the sector class of Section~\ref{subsec:sector} a nonzero angular
component of the leading $r^{-1}$ field makes the normal moment on an
edge through the pole non-integrable.  The first is present in all three
problems of Section~\ref{sec:numerics}, the second only in~(S).

What is needed near the pole is therefore not a more accurate
comparison field but a defined and equilibrated one.
Section~\ref{subsec:patch} constructs it by local flux balance, and its
error on the patch has the scale $h_{x_0}^{2/p-1}$, which
Section~\ref{subsec:sharpness} shows to be unavoidable when the leading
homogeneous $r^{-1}$ term does not vanish; in that regime nothing is
given up in exchange.  Sections~\ref{subsec:global-rate}
and~\ref{subsec:equidistribution} derive the global rate from it, on
quasi-uniform and on graded families.

We work throughout with a partition of $\O$ adapted to the coefficient
and with a local hypothesis on $\bsigma$ near the pole.
Let $\cP=\{\O_j\}_{j=1}^{J}$ be a finite partition of $\O$ into
pairwise disjoint Lipschitz subdomains, subordinate to the coefficient
interfaces, so that
$\overline{\O}=\bigcup_{j=1}^{J}\overline{\O_j}$.  The point $x_0$ may
lie in the interior of one subdomain, on an interface, or at an
interface junction.  Write $\nabla_{\cP}\bv:=\nabla(\bv|_{\O_j})$ in
$\O_j$ for the gradient taken subdomain by subdomain, which is defined
for fields that jump across the interfaces.
All triangulations are assumed to be fitted to $\cP$: every $K\in\cT_h$
is contained in the closure of a single subdomain $\O_j$.  Piecewise
regularity with respect to $\cP$ is then elementwise regularity, which is
what the Raviart--Thomas approximation estimate uses.

\begin{assumption}[Local structure of the singular flux]
\label{ass:S}
There exist $\rho_0>0$ and $C_{\bsigma}>0$ such that
\begin{equation}\label{eq:S-pointwise}
  |\bsigma(x)|\le C_{\bsigma}|x-x_0|^{-1}
  \qquad\text{for a.e.\ }x\in B_{\rho_0}(x_0)\setminus\{x_0\} .
\end{equation}
No continuity of the tangential component of $\bsigma$ across the
interfaces is assumed.
\end{assumption}

This is a bound on the size of $\bsigma$ alone: it involves no
derivative, and it is all that Section~\ref{subsec:sharpness}, the patch
construction of Section~\ref{subsec:patch} and
Lemma~\ref{lem:patch-interpolant} require, together with the geometry of
$\omega_{x_0}$.  What the rate of
Section~\ref{subsec:global-rate} needs in addition is derivative
information away from the pole, which is stated there as
\eqref{eq:away-derivative}.

\subsection{The cost at the pole}
\label{subsec:sharpness}

By Assumption~\ref{ass:S}, if $\omega\subset B_{Ch}(x_0)$ with
$Ch\le\rho_0$, polar coordinates give
\begin{equation}\label{eq:sigma-patch-size}
  \|\bsigma\|_{L^p(\omega)}
  \le C_{\bsigma}\Bigl(\int_0^{Ch}r^{-p}r\,dr\Bigr)^{1/p}
  \le CC_{\bsigma}h^{2/p-1},
\end{equation}
where $p<2$ is used in the last step.  This subsection shows that for a
nonzero homogeneous $r^{-1}$ leading term $h^{2/p-1}$ is a lower bound
as well, for every field of $\RT_0$.  The statement is independent of
the method and of the choice of interpolant; the only geometric input
is the shape of the pole element after scaling.  We begin on a single
element.

\begin{thm}[Local $\RT_0$ approximation of an $r^{-1}$ singularity]
\label{thm:sharp-local}
Let $1<p<2$, let $K_1$ be a fixed triangle with $0\in\overline{K_1}$, and
let $\bq\in L^p(K_1)^2$ satisfy the scaling identity
$\bq(t\bx)=t^{-1}\bq(\bx)$ for $0<t\le1$ and a.e.\ $\bx\in K_1$, and let
$\bq$ be not identically zero on $K_1$.  Then
\begin{equation}\label{eq:sharp-local}
\inf_{\btau_h\in\RT_0(hK_1)}\|\bq-\btau_h\|_{L^p(hK_1)}\simeq h^{2/p-1},
\end{equation}
with constants independent of $h$.  Only the scaling identity is used;
nothing else about $\bq$ enters.
\end{thm}

\begin{proof}
For $\btau_h(\bx)=\ba+b\bx\in\RT_0(hK_1)$ define $\btau_1(\by):=h\btau_h(h\by)=h\ba+bh^2\by\in\RT_0(K_1)$. This is a bijection between $\RT_0(hK_1)$ and $\RT_0(K_1)$. Since $\bq(h\by)=h^{-1}\bq(\by)$ and $dx=h^2dy$,
$\|\bq-\btau_h\|_{L^p(hK_1)}^p=h^{2-p}\|\bq-\btau_1\|_{L^p(K_1)}^p$.
Taking the infimum gives \eqref{eq:sharp-local}. The reference-element infimum is finite because $\bq\in L^p(K_1)^2$ by hypothesis, and it is positive because $\RT_0(K_1)$ is finite dimensional and hence closed in $L^p(K_1)^2$, while $\bq\notin\RT_0(K_1)$: if $\bq$ agreed a.e.\ with some $\ba+b\bx\in\RT_0(K_1)$, the scaling identity would give $t^{-1}(\ba+b\bx)=\ba+bt\bx$ for every $0<t\le1$ and a.e.\ $\bx\in K_1$, forcing $\ba=0$ and $b=0$.
\end{proof}

The same lower bound holds for the error over the whole domain, with the
size of the pole element in place of $h$.

\begin{cor}[No $\RT_0$ field approximates the singularity better]
\label{cor:global-lower}
Let $1<p<2$ and let $K^{x_0}$ be an element of $\cT_h$ whose closure
contains $x_0$, of the form $K^{x_0}=x_0+h_{K^{x_0}}K_1$ with $K_1$ a
fixed triangle satisfying $0\in\overline{K_1}$ and $\diam K_1=1$, so
that $h_{K^{x_0}}$ is the diameter of $K^{x_0}$.  Let $\bq$ be a fixed field on a neighborhood of $x_0$, lying
in $L^p$ there, with $\bq\bigl(x_0+t(x-x_0)\bigr)=t^{-1}\bq(x)$ for
$0<t\le1$ and not identically zero on $K^{x_0}$, a condition
independent of $h_{K^{x_0}}$, by that identity.  Assume that on
$K^{x_0}$
\begin{equation}\label{eq:leading-singular-part}
\bsigma=\bq+\widetilde{\bsigma},\qquad
\|\widetilde{\bsigma}\|_{L^p(K^{x_0})}=o(h_{K^{x_0}}^{2/p-1}).
\end{equation}
Then there is $c>0$ such that
\begin{equation}\label{eq:global-lower}
\|\bsigma-\btau_h\|_{L^p(\O)}\ge c\,h_{K^{x_0}}^{2/p-1}
\qquad\text{for every }\btau_h\in\RT_0(\cT_h)
\end{equation}
and all $h_{K^{x_0}}$ small enough.  In particular
$\|\bsigma-\bsigma_h\|_{L^p(\O)}\ge c\,h_{K^{x_0}}^{2/p-1}$.
\end{cor}

\begin{proof}
Restrict to $K^{x_0}$ and use that $\btau_h|_{K^{x_0}}\in\RT_0(K^{x_0})$:
\[
\|\bsigma-\btau_h\|_{L^p(\O)}
\ge\|\bsigma-\btau_h\|_{L^p(K^{x_0})}
\ge \inf_{\btau\in\RT_0(K^{x_0})}\|\bsigma-\btau\|_{L^p(K^{x_0})} .
\]
Fix $h_*>0$ with $x_0+h_*K_1$ inside the neighborhood on which $\bq$ is
defined, and set $\bq_1(y):=h_*\bq(x_0+h_*y)$ for $y\in K_1$.  For
$h:=h_{K^{x_0}}\le h_*$ and $t:=h/h_*\le1$ the scaling identity gives
$\bq(x_0+hy)=t^{-1}\bq(x_0+h_*y)$, so that $h\,\bq(x_0+hy)=\bq_1(y)$:
one and the same fixed nonzero field on $K_1$ serves every pole element.
Theorem~\ref{thm:sharp-local} applied to it, after translation to $x_0$,
therefore bounds the last infimum below by $c\,h_{K^{x_0}}^{2/p-1}$ with
$c$ depending only on $\bq_1$, $K_1$ and $p$; subtracting
$\|\widetilde{\bsigma}\|_{L^p(K^{x_0})}$ and using
\eqref{eq:leading-singular-part} leaves $c'h_{K^{x_0}}^{2/p-1}$ for
$h_{K^{x_0}}$ small enough.

\end{proof}

A refinement strategy that produces only finitely many pole-element
shapes up to translation and scaling, as newest-vertex bisection does,
is covered as well, provided the same fixed $\bq$ is nonzero on each
shape that occurs: the corollary then applies to each of them with a
positive constant, and the smallest of those finitely many constants
serves for all.

\subsection{A sector example with a non-integrable normal moment}
\label{subsec:sector}

Where $\Pi_h^{\rm rt}\bsigma$ is defined it is subject to
Corollary~\ref{cor:global-lower} like every other field of $\RT_0$, and
it can do no better: for the radial flux of
Section~\ref{subsec:regularity} it is defined on every element, and on
the element carrying the pole all three of its moments vanish when
$f_1=0$, so it is zero there and misses the singular field entirely.
This subsection gives a coefficient for which its moments do not exist
at all.

Write the term $\bq$ of Section~\ref{subsec:sharpness} in polar form, $\bq=r^{-1}S(\theta)$ with
$S=S_re_r+S_\theta e_\theta$, $e_r=(\cos\theta,\sin\theta)^T$ and
$e_\theta=(-\sin\theta,\cos\theta)^T$.  Section~\ref{subsec:sharpness} used
only the size of $\bq$, to which both parts contribute.  The angular
part decides something further, namely whether the canonical moments
exist.  An edge $e$ through the pole lies at a fixed angle $\theta_e$
and has $\bn_e=\pm e_\theta$, so that the leading contribution to
$\int_e|\bsigma\cdot\bn_e|\,ds$ is
$|S_\theta(\theta_e)|\int_0^{h_e}r^{-1}\,dr$: the moment on that edge is
lost as soon as $S_\theta(\theta_e)\neq0$.

\begin{exm}[An explicit interface singularity]
\label{ex:sector}
Let $x_0=0$ and use polar coordinates $(r,\theta)$. Set
\begin{equation}\label{eq:sector-A}
A(\theta)=
\begin{cases}
\mathrm{diag}(a,1), & 0<\theta<\pi/2,\\
I, & \pi/2<\theta<2\pi ,
\end{cases}
\qquad a>0 .
\end{equation}
Then $-\gradt(A\nabla\Phi)=\delta_0$ admits the explicit solution
\begin{equation}\label{eq:sector-solution}
\Phi(r,\theta)=-\gamma\log r+U(\theta),
\end{equation}
where
\[
U(\theta)=
\begin{cases}
-\dfrac{\gamma}{2}\log\bigl(1+(a-1)\sin^2\theta\bigr)
 -\dfrac{C_A}{\sqrt a}\arctan(\sqrt a\tan\theta), & 0<\theta<\pi/2,\\[6pt]
-\dfrac{\gamma}{2}\log a-\dfrac{C_A\pi}{2\sqrt a}
 -C_A\bigl(\theta-\dfrac{\pi}{2}\bigr), & \pi/2<\theta<2\pi,
\end{cases}
\]
and
\begin{equation}\label{eq:sector-constants}
\gamma=\left(\frac{\pi(3+\sqrt a)}{2}
 +\frac{\sqrt a\,(\log a)^2}{2\pi(3\sqrt a+1)}\right)^{-1},\qquad
C_A=-\frac{\sqrt a\,\log a}{\pi(3\sqrt a+1)}\,\gamma .
\end{equation}
The constant that matters below is $C_A$, and it vanishes exactly when
$a=1$: the difficulty is carried by $\log a$ alone, and it is present
for every anisotropy however mild.  The experiments of
Section~\ref{sec:numerics} take $a=4$, for which
$\gamma=\bigl(\tfrac{5\pi}{2}+\tfrac{4(\log2)^2}{7\pi}\bigr)^{-1}$ and
$C_A=-\tfrac{4\log2}{7\pi}\gamma$, that is
$\gamma=1.2592\cdot10^{-1}$ and $|C_A|=1.5876\cdot10^{-2}$.
Taking $\O=B_1(0)$, $f_1=0$ and $g=\Phi|_{\p\O}$, the function $u=\Phi$ is
exactly the solution of \eqref{eq:deltaeq}, so the discussion below
applies to the model problem of this paper and not only to a
whole-plane model.
Here $\arctan(\sqrt a\tan\theta)$ is taken continuously on $(0,\pi/2)$. The constants in \eqref{eq:sector-constants} are determined by the periodicity of $U$ and the normalization $\int_{\partial B_\varepsilon}-A\nabla\Phi\cdot\bn\,ds=1$. A direct calculation shows that $U$ is bounded and Lipschitz, and that $|\nabla\Phi|\simeq r^{-1}$ near the pole.
\end{exm}

\begin{prop}[Canonical Raviart--Thomas moments may be undefined]
\label{prop:moments-fail}
For the coefficient \eqref{eq:sector-A} with $a\neq1$, let $\bsigma$ be the modified flux corresponding to the explicit solution above. Then $\int_e|\bsigma\cdot\bn_e|\,ds=\infty$ on any straight edge $e$ emanating radially from $x_0$, so that the degree of freedom $\int_e\bsigma\cdot\bn_e\,ds$ is not
defined.  The same integral is a degree of freedom of the
Brezzi--Douglas--Marini family \cite{BrezziDouglasMarini1985}, so the
failure is not particular to $\RT_0$.
\end{prop}

\begin{proof}
The construction above gives $-A\nabla\Phi\cdot e_\theta=C_A/r$, with $C_A\neq0$ by \eqref{eq:sector-constants} because $a\neq1$: since $\Phi=-\gamma\log r+U(\theta)$ and $A$ is constant in each sector, $A\nabla\Phi=r^{-1}(V_r(\theta)e_r+V_\theta(\theta)e_\theta)$ and $\gradt(A\nabla\Phi)=r^{-2}V_\theta'(\theta)$, so $V_\theta$ is constant in each sector, and continuity of the conormal flux across the rays $\theta=0$ and $\theta=\pi/2$ makes the two constants equal; in the isotropic sector $V_\theta=U'=-C_A$. Since $\bff_2=(2\pi r)^{-1}e_r$ is purely radial, $\bsigma\cdot e_\theta=(-A\nabla\Phi-\bff_2)\cdot e_\theta=C_A/r$. On a radial edge $\bn_e=\pm e_\theta$, and therefore
$\int_e|\bsigma\cdot\bn_e|\,ds=|C_A|\int_0^{h_e}r^{-1}\,dr=\infty$.
\end{proof}

\begin{remark}
Proposition~\ref{prop:moments-fail} concerns the canonical degrees of
freedom, not the $L^p$ approximation problem itself; the free
$L^p$-best approximation by $\RT_0$ remains well posed for every
$1<p<2$, and by Section~\ref{subsec:sharpness} it is of order
$h_{K^{x_0}}^{2/p-1}$ here, as in every case covered by
Corollary~\ref{cor:global-lower}.  Problem~(S) of
Section~\ref{sec:numerics} is this coefficient with $a=4$; problems~(I)
and~(A) have $S_\theta=0$, so their canonical moments all exist.
\end{remark}

\subsection{An equilibrated comparison flux near the pole}
\label{subsec:patch}

Throughout this subsection we assume \eqref{eq:away-regularity} below,
so that the edge moments of $\bsigma$ used here exist: on a general
$\btau\in H^{p}(\divvr;\O)$ the normal trace is only a distribution in
$W^{-1/p,p}(\p\O)$ and cannot be integrated over a single edge.  The
same difficulty, in the $L^2$ setting and for a solution of low
regularity, is what rules out a direct application of the standard
mixed theory in \cite[Rem.~16]{Zhang:20mixed}.

We construct $\bsigma_I\in\RT_0^{f_1}(\cT_h)$ which coincides with $\Pi_h^{\rm rt}\bsigma$ away from a small patch of the pole and is defined there by local flux balance. Choose an edge-connected patch $\omega_{x_0}$, one in which any two elements are joined by a chain of elements sharing full edges, containing every element whose closure contains $x_0$ and, if necessary, enlarge it by a fixed number of element layers. Edge-connectedness is what makes the balance system below solvable under its single compatibility condition, and a vertex or element star has it.  Write
$h_{x_0}:=\max\{h_K:\ K\subset\omega_{x_0}\}$ for the mesh size at the
pole, so that
\begin{equation}\label{eq:patch-geometry}
\diam(\omega_{x_0})\simeq h_{x_0},\qquad
\mathrm{dist}(x_0,\partial\omega_{x_0})\simeq h_{x_0},
\end{equation}
with a uniformly bounded number of elements.  Shape regularity and the
bounded number of layers make all of them comparable, so that
$h_{K^{x_0}}\simeq h_{x_0}$ for the pole element of
Corollary~\ref{cor:global-lower}.  Estimates written in either scale
therefore have the same order; we keep $h_{K^{x_0}}$ for the local lower
bound and $h_{x_0}$ for the patch estimates. Outside $\omega_{x_0}$ set $\bsigma_I:=\Pi_h^{\rm rt}\bsigma$. On every boundary edge $e\subset\partial\omega_{x_0}$ prescribe $\int_e\bsigma_I\cdot\bn\,ds=\int_e\bsigma\cdot\bn\,ds$, and choose the interior edge fluxes so that
\begin{equation}\label{eq:patch-balance}
\sum_{e\subset\partial K}\int_e\bsigma_I\cdot\bn_K\,ds=\int_KQ_hf_1\,dx\qquad\forall\,K\subset\omega_{x_0}.
\end{equation}
The system is compatible because summing over the patch gives
\[
\int_{\partial\omega_{x_0}}\bsigma_I\cdot\bn\,ds=\int_{\partial\omega_{x_0}}\bsigma\cdot\bn\,ds=\int_{\omega_{x_0}}f_1\,dx=\int_{\omega_{x_0}}Q_hf_1\,dx.
\]
The middle equality is the divergence theorem on a region containing the
pole, and it is legitimate precisely because the splitting has removed
the measure: $\bsigma\in L^p(\omega_{x_0})^2$ with
$\gradt\bsigma=f_1\in L^p(\omega_{x_0})$, so
$\bsigma\in H^{p}(\divvr;\omega_{x_0})$ and no boundary term survives at
$x_0$.
Indeed, \eqref{eq:patch-balance} is a linear system for the fluxes on
the interior edges of $\omega_{x_0}$, with one equation per element and
with matrix entries $0$ and $\pm1$; the right-hand side collects the
cell terms and the prescribed boundary fluxes. Summing the equations
cancels every interior flux, so the identity above is the only
condition the right-hand side must satisfy, and the system is solvable.
We take the solution of smallest Euclidean norm; a patch problem of this
kind, with a construction that fixes the remaining kernel, is used in the
equilibrated residual estimator of \cite{CaiZhang2012}. Its matrix is
determined by the combinatorics of $\omega_{x_0}$ alone, and
\eqref{eq:patch-geometry} together with shape regularity leaves only
finitely many possibilities, so the resulting interior fluxes are
bounded by the boundary fluxes and the cell terms with a constant
depending only on the shape-regularity parameter.
That the resulting field is $H(\divvr)$-conforming and satisfies $\gradt\bsigma_I=Q_hf_1$, so that $\bsigma_I\in\RT_0^{f_1}(\cT_h)$, is verified in Lemma~\ref{lem:patch-interpolant}.

The next lemma shows that the patch modification has the natural $L^p$
size of the singular flux, namely the $h_{x_0}^{2/p-1}$ of
Section~\ref{subsec:sharpness}, together with the source contribution
$h_{x_0}\|f_1\|_{L^p(\omega_{x_0})}$.

\begin{lem}[Patch $\RT_0$ comparison flux]
\label{lem:patch-interpolant}
Let $1<p<2$, let Assumption~\ref{ass:S} hold, and let $\omega_{x_0}$ satisfy \eqref{eq:patch-geometry} and $\omega_{x_0}\subset B_{\rho_0}(x_0)$, which holds once $h_{x_0}$ is small enough. Assume in addition that
\begin{equation}\label{eq:away-regularity}
\bsigma|_K\in W^{s_K,p}(K)^2\quad\text{with }s_K\in(1/p,1]\qquad\text{for every }K\in\cT_h\text{ not contained in }\omega_{x_0},
\end{equation}
so that the normal trace of $\bsigma$ on $\p K$ lies in $L^1(\p K)$, the canonical interpolant $\Pi_h^{\rm rt}\bsigma$ is defined on $\O\setminus\omega_{x_0}$, and the normal moments on $\partial\omega_{x_0}$ used in the patch construction are well defined. Then the field $\bsigma_I$ constructed in Section~\ref{subsec:patch} belongs to $\RT_0^{f_1}(\cT_h)$ and satisfies
\begin{equation}\label{eq:patch-estimate}
\|\bsigma-\bsigma_I\|_{L^p(\O)}
\le \|\bsigma-\Pi_h^{\rm rt}\bsigma\|_{L^p(\O\setminus\omega_{x_0})}
+C C_{\bsigma}h_{x_0}^{2/p-1}+Ch_{x_0}\|f_1\|_{L^p(\omega_{x_0})},
\end{equation}
where $C$ depends only on $p$, the shape-regularity parameter, and the constants in \eqref{eq:patch-geometry}.
\end{lem}
\begin{proof}
Outside $\omega_{x_0}$ we have $\bsigma_I=\Pi_h^{\rm rt}\bsigma$. Hence the commuting property of the canonical Raviart--Thomas interpolant gives $\gradt\bsigma_I=\gradt(\Pi_h^{\rm rt}\bsigma)=Q_h(\gradt\bsigma)=Q_hf_1$ on every $K\not\subset\omega_{x_0}$. On the patch, the interior edge fluxes in the construction of Section~\ref{subsec:patch} are shared by the two adjacent elements with opposite orientations, while on $\partial\omega_{x_0}$ the prescribed fluxes agree with the corresponding degrees of freedom of $\Pi_h^{\rm rt}\bsigma$. Therefore the interior and exterior fields fit together to give $\bsigma_I\in\RT_0(\cT_h)$. Moreover, for every $K\subset\omega_{x_0}$ the element balance equation gives $\int_{\partial K}\bsigma_I\cdot\bn_K\,ds=\int_KQ_hf_1\,dx$. Since $\gradt\bsigma_I$ is constant on $K$, $|K|\,\gradt\bsigma_I|_K=\int_KQ_hf_1\,dx=|K|\,Q_hf_1|_K$, and hence $\gradt\bsigma_I=Q_hf_1$ also on the patch. Thus $\bsigma_I\in\RT_0^{f_1}(\cT_h)$.

We next estimate the error on the patch. By \eqref{eq:patch-geometry} the patch has diameter at most $Ch_{x_0}$, so \eqref{eq:sigma-patch-size} gives $\|\bsigma\|_{L^p(\omega_{x_0})}\le CC_{\bsigma}h_{x_0}^{2/p-1}$.

It remains to estimate $\bsigma_I$ on the patch. Let $K\in\cT_h$ and $\btau_h\in\RT_0(K)$. On a fixed reference triangle $\widehat K$, finite-dimensional norm equivalence gives $\|\widehat{\btau}_h\|_{L^p(\widehat K)}\le C\sum_{\widehat e\subset\partial\widehat K}\left|\int_{\widehat e}\widehat{\btau}_h\cdot\widehat{\bn}_{\widehat e}\,d\widehat s\right|$. Under the contravariant Piola transformation the normal flux degrees of freedom are invariant, while shape regularity gives $\|\btau_h\|_{L^p(K)}\le Ch_K^{2/p-1}\|\widehat{\btau}_h\|_{L^p(\widehat K)}$. Consequently,
\begin{equation}\label{eq:patch-proof-piola}
\|\btau_h\|_{L^p(K)}\le Ch_K^{2/p-1}\sum_{e\subset\partial K}\left|\int_e\btau_h\cdot\bn_{K,e}\,ds\right|.
\end{equation}

We first bound the prescribed fluxes on the outer boundary of the patch. Let $e\subset\partial\omega_{x_0}$. Let $K^{+}$ be the element outside $\omega_{x_0}$ adjacent to $e$. By \eqref{eq:patch-geometry}, $\mathrm{dist}(K^{+},x_0)\simeq h_{x_0}$ and $|e|\le Ch_{x_0}$, so Assumption~\ref{ass:S} gives $\|\bsigma\|_{L^\infty(K^{+})}\le CC_{\bsigma}h_{x_0}^{-1}$; since $\bsigma|_{K^{+}}\in W^{s_{K^{+}},p}(K^{+})^2$ with $s_{K^{+}}p>1$ by \eqref{eq:away-regularity}, its trace on $e$ is a function and inherits that bound. No regularity inside the patch is used. Therefore $\left|\int_e\bsigma\cdot\bn_{\omega}\,ds\right|\le\int_e|\bsigma|\,ds\le C C_{\bsigma}h_{x_0}^{-1}|e|\le C C_{\bsigma}$. Thus every prescribed outer flux degree of freedom is bounded uniformly by $C C_{\bsigma}$.

We now consider the interior edge fluxes. Since the patch contains only a uniformly bounded number of elements, there are only finitely many possible balance systems up to connectivity, and the corresponding minimum-Euclidean-norm solution operators are uniformly bounded. Hence, for every interior edge $e$,
\begin{equation}\label{eq:patch-proof-interior-flux}
\left|\int_e\bsigma_I\cdot\bn_e\,ds\right|\le C\left(C_{\bsigma}+\sum_{K\subset\omega_{x_0}}\left|\int_KQ_hf_1\,dx\right|\right).
\end{equation}
Here the first term controls the prescribed outer fluxes, while the second term contains the element balance data.

For the latter, the mean-preserving property of $Q_h$ and Hölder's inequality give $\left|\int_KQ_hf_1\,dx\right|=\left|\int_Kf_1\,dx\right|\le |K|^{1-1/p}\|f_1\|_{L^p(K)}\le Ch_K^{2-2/p}\|f_1\|_{L^p(K)}$. Since the patch contains only a uniformly bounded number of elements and $h_K\le Ch_{x_0}$ there, $\sum_{K\subset\omega_{x_0}}\left|\int_KQ_hf_1\,dx\right|\le Ch_{x_0}^{2-2/p}\|f_1\|_{L^p(\omega_{x_0})}$. Combining this estimate with the bound for the outer fluxes and \eqref{eq:patch-proof-interior-flux}, every edge flux degree of freedom of $\bsigma_I$ on the patch satisfies
\begin{equation}\label{eq:patch-proof-all-fluxes}
\left|\int_e\bsigma_I\cdot\bn_e\,ds\right|\le C C_{\bsigma}+Ch_{x_0}^{2-2/p}\|f_1\|_{L^p(\omega_{x_0})}.
\end{equation}

Applying \eqref{eq:patch-proof-piola} to $\btau_h=\bsigma_I|_K$, using $h_K\le Ch_{x_0}$ and \eqref{eq:patch-proof-all-fluxes}, gives
\[
  \begin{aligned}
  \|\bsigma_I\|_{L^p(K)}
  &\le C C_{\bsigma}h_{x_0}^{2/p-1}
     +C h_{x_0}^{2/p-1}h_{x_0}^{2-2/p}\|f_1\|_{L^p(\omega_{x_0})}\\
  &=C C_{\bsigma}h_{x_0}^{2/p-1}
     +C h_{x_0}\|f_1\|_{L^p(\omega_{x_0})} .
  \end{aligned}
\]
Since $\omega_{x_0}$ contains only a uniformly bounded number of elements, it follows that
\begin{equation}\label{eq:patch-proof-sigmaI}
\|\bsigma_I\|_{L^p(\omega_{x_0})}\le C C_{\bsigma}h_{x_0}^{2/p-1}+C h_{x_0}\|f_1\|_{L^p(\omega_{x_0})}.
\end{equation}

Finally, using $\bsigma_I=\Pi_h^{\rm rt}\bsigma$ outside the patch and the triangle inequality on $\omega_{x_0}$, $\|\bsigma-\bsigma_I\|_{L^p(\O)}\le \|\bsigma-\Pi_h^{\rm rt}\bsigma\|_{L^p(\O\setminus\omega_{x_0})}+\|\bsigma-\bsigma_I\|_{L^p(\omega_{x_0})}\le \|\bsigma-\Pi_h^{\rm rt}\bsigma\|_{L^p(\O\setminus\omega_{x_0})}+\|\bsigma\|_{L^p(\omega_{x_0})}+\|\bsigma_I\|_{L^p(\omega_{x_0})}$. Combining this with \eqref{eq:patch-proof-sigmaI} proves \eqref{eq:patch-estimate}.
\end{proof}

When the leading term of Corollary~\ref{cor:global-lower} does not
vanish, the patch contribution $C_{\bsigma}h_{x_0}^{2/p-1}$ in
\eqref{eq:patch-estimate} has exactly the order that no field of
$\RT_0$ improves on.  The patch construction therefore repairs
definition and equilibrium, not accuracy, and in that regime it costs
nothing asymptotically.

\subsection{Global a priori rates}
\label{subsec:global-rate}

What has been proved so far already gives an estimate on any fitted
mesh.  No quasi-uniformity enters it: the quasi-best approximation
property of Theorem~\ref{thm:flux-quasibest} holds on every mesh, the
patch construction is local to $\omega_{x_0}$, and the canonical
interpolation estimate is elementwise.

\begin{thm}[Flux estimate on a fitted mesh]
\label{thm:flux-mesh}
Let $1<p<2$, let Assumption~\ref{ass:S} (the pointwise bound on
$\bsigma$), the elementwise regularity \eqref{eq:away-regularity} and
Hypothesis~(SP), \eqref{eq:Duran-stability}, hold, let $\cT_h$ be
shape-regular, and let $\omega_{x_0}\subset B_{\rho_0}(x_0)$, which holds
once $h_{x_0}$ is small enough.  Then
\begin{equation}\label{eq:flux-mesh}
\begin{split}
  \|\bsigma-\bsigma_h\|_{L^p(\O)}
  \le C(1+C_{\mathrm{SP}})
  \Bigl(&C_{\bsigma}\,h_{x_0}^{2/p-1}
   +h_{x_0}\|f_1\|_{L^p(\omega_{x_0})}\\
  &+\Bigl(\sum_{K\not\subset\omega_{x_0}}
      h_K^{s_Kp}\,|\bsigma|_{W^{s_K,p}(K)}^{p}\Bigr)^{1/p}\Bigr),
\end{split}
\end{equation}
where $C$ depends only on $p$, on the indices $s_K$, on the
shape-regularity parameter and on the constants in
\eqref{eq:patch-geometry}.  Under
\eqref{eq:Duran-two-term} the same argument gives \eqref{eq:flux-mesh}
with $1+C_{\mathrm{SP}}$ replaced by $1+C_0$ and the oscillation term of
Theorem~\ref{thm:flux-two-term} added.
\end{thm}

\begin{proof}
By Theorem~\ref{thm:flux-quasibest} and
Lemma~\ref{lem:patch-interpolant},
\[
  \|\bsigma-\bsigma_h\|_{L^p(\O)}
  \le(1+C_{\mathrm{SP}})\Bigl(
  \|\bsigma-\Pi_h^{\rm rt}\bsigma\|_{L^p(\O\setminus\omega_{x_0})}
  +CC_{\bsigma}h_{x_0}^{2/p-1}+Ch_{x_0}\|f_1\|_{L^p(\omega_{x_0})}\Bigr).
\]
Every $K\in\cT_h$ not contained in $\omega_{x_0}$ lies in the closure of
a single subdomain of $\cP$, so $\bsigma|_K\in W^{s_K,p}(K)^2$ by
\eqref{eq:away-regularity}.  Since $s_Kp>1$ the normal trace of
$\bsigma$ on $\p K$ lies in $L^1(\p K)$, so $\Pi_h^{\rm rt}\bsigma|_K$
is defined, and the reference-element argument, with polynomial
approximation in $W^{s,p}$ \cite{DupontScott1980} and Piola scaling,
gives $\|\bsigma-\Pi_h^{\rm rt}\bsigma\|_{L^p(K)}
\le Ch_K^{s_K}|\bsigma|_{W^{s_K,p}(K)}$.  Summing the $p$-th powers over
these elements gives \eqref{eq:flux-mesh}.  For the two-term version,
replace Theorem~\ref{thm:flux-quasibest} by
Theorem~\ref{thm:flux-two-term}.
\end{proof}

The smoothness index in \eqref{eq:flux-mesh} is elementwise, as in
\cite{CaiHeZhang2017, Zhang:20mixed}: the interpolation estimate behind
the third term is local, so it may be applied on each element with
whatever regularity $\bsigma$ has there.  The patch terms carry no
smoothness index at all.

Estimate \eqref{eq:flux-mesh} is the form in which the mesh, graded or
not, enters the flux error.  To read a rate off it one needs to know how
$\bsigma$ behaves between the patch and the far field, which is
derivative information away from the pole:
\begin{equation}\label{eq:away-derivative}
  |\nabla_{\cP}\bsigma(x)|\le C_{\bsigma}|x-x_0|^{-2}
  \quad\text{for a.e.\ }x\in B_{\rho_0}(x_0)\setminus\{x_0\},
  \qquad
  \nabla_{\cP}\bsigma\in L^p\bigl(\O\setminus\overline{B_{\rho_0}(x_0)}\bigr) .
\end{equation}
Both of these are needed.  The first bounds the elementwise
seminorms in \eqref{eq:flux-mesh} on the annulus
$B_{\rho_0}(x_0)\setminus\omega_{x_0}$, where a finite global $W^{1,p}$
seminorm is not available: $\int_0r^{-2p}r\,dr$ diverges for every
$p\ge1$, and it is the balance of $h_K$ against $r^{-2}$ that produces
the exponent $2/p-1$.  The second is what makes the far field contribute
$O(h)$.  Both are statements about $\bsigma$ relative to $\cP$ alone and
involve no mesh.

\begin{cor}[Rate on quasi-uniform meshes]
\label{cor:global-rate}
Let $1<p<2$, let Assumption~\ref{ass:S}, the elementwise regularity
\eqref{eq:away-regularity} with $s_K=1$, the derivative bound
\eqref{eq:away-derivative} and Hypothesis~(SP) hold, and let
$\{\cT_h\}$ be a quasi-uniform, shape-regular family.  Then
\begin{equation}\label{eq:global-upper}
\|\bsigma-\bsigma_h\|_{L^p(\O)}\le Ch^{2/p-1}.
\end{equation}
\end{cor}

\begin{proof}
We bound the three terms of \eqref{eq:flux-mesh}.  On a quasi-uniform
family $h_{x_0}\simeq h$, so the first is $O(h^{2/p-1})$ and the second
is $O(h)$.  For the third, split the elements not contained in
$\omega_{x_0}$ at $\p B_{\rho_0}(x_0)$.  Those meeting
$B_{\rho_0}(x_0)$ are at distance $\gtrsim h$ from $x_0$ by
\eqref{eq:patch-geometry}, so the first bound in
\eqref{eq:away-derivative} gives
\[
\Bigl(\sum h_K^{p}|\bsigma|_{W^{1,p}(K)}^{p}\Bigr)^{1/p}
\le Ch\Bigl(\int_{ch}^{\rho_0}r^{-2p}r\,dr\Bigr)^{1/p}
\le Ch^{2/p-1},
\]
and for the remaining elements the second bound gives an $O(h)$
contribution.  Using $2/p-1<1$ proves \eqref{eq:global-upper}.
\end{proof}

The exponent cannot be improved when the leading term does not vanish.
On a quasi-uniform family to which Corollary~\ref{cor:global-lower}
applies, and with $\bq$ not identically zero in
\eqref{eq:leading-singular-part}, no field of $\RT_0(\cT_h)$ does
better on the pole element; there $h_{K^{x_0}}\simeq h$, so the two bounds
meet, $\|\bsigma-\bsigma_h\|_{L^p(\O)}\simeq h^{2/p-1}$, and the rate is
$N^{-(1/p-1/2)}$ since $N\simeq h^{-2}$.  If $\bq\equiv0$ nothing here
forbids a faster rate (in the problem of
Remark~\ref{rem:num-matched} the flux error vanishes identically),
while \eqref{eq:global-upper} holds in either case.

\subsection{Local form and equidistribution}
\label{subsec:equidistribution}

Estimate \eqref{eq:flux-mesh} is local in the sense used for interface
problems in \cite{CaiHeZhang2017, Zhang:20mixed}: apart from the two
terms attached to the patch, each contribution involves only $h_K$ and
$\bsigma$ on the single element $K$, so that the bound is a sum of local
indicators and the question of which mesh makes them equal can be asked.
That is what an adaptive method needs, since it tells one both what the
grading the local bound suggests and the complexity such a grading
delivers.
The other half of those papers, robustness of the constants in the
coefficient, is not pursued here.

Write the terms of \eqref{eq:flux-mesh} as local indicators,
\begin{equation}\label{eq:local-indicators}
  \zeta_{x_0}:=C_{\bsigma}h_{x_0}^{2/p-1}
   +h_{x_0}\|f_1\|_{L^p(\omega_{x_0})},
  \qquad
  \zeta_K:=h_K^{s_K}\,|\bsigma|_{W^{s_K,p}(K)} ,
\end{equation}
so that \eqref{eq:flux-mesh} reads
$\|\bsigma-\bsigma_h\|_{L^p(\O)}\le C(1+C_{\mathrm{SP}})
\bigl(\zeta_{x_0}
 +(\sum_{K\not\subset\omega_{x_0}}\zeta_K^{p})^{1/p}\bigr)$.
The patch indicator carries no smoothness index at all: there is no
exponent $s$ and no seminorm of $\bsigma$ in $\zeta_{x_0}$, only
$C_{\bsigma}$ and the size of the patch, and under the nonvanishing
hypothesis of Corollary~\ref{cor:global-lower} its first term is not
merely an upper bound but the exact order of the local best
approximation.  On that
element a mesh can do one thing only, namely shrink it; everywhere else
it acts through the regularity of $\bsigma$.  An adaptive algorithm is
doing both at once, and \eqref{eq:flux-mesh} is what separates them.

Equidistribution of the dominant contributions to
\eqref{eq:local-indicators}, the singular patch term and the
elementwise terms $\zeta_K$, the source term on the patch turning out to
be of higher order, then determines the grading and the rate it
delivers.
They are not computable, being built from $\bsigma$ itself.

For the model singularity the answer is explicit, and it restores the
first-order complexity that Corollary~\ref{cor:global-rate} loses.

\begin{prop}[Graded balancing and first-order complexity]
\label{prop:equidistribution}
Let $1<p<2$, let the hypotheses of Theorem~\ref{thm:flux-mesh} with
$s_K=1$ and the derivative bound \eqref{eq:away-derivative} hold, and set
\[
  \mu:=\frac{2p}{p+2}\in(0,1).
\]
Let $\{\cT_c\}_{c>0}$ be fitted, shape-regular families graded so that
\begin{equation}\label{eq:grading}
  h_K\simeq c\,\bigl(\max\{|x_K-x_0|,h_{x_0}\}\bigr)^{\mu}
  \quad\text{for } K\subset B_{\rho_0}(x_0),
  \qquad
  h_K\simeq c\ \text{ otherwise},
\end{equation}
with $x_K$ the barycenter of $K$ and $h_{x_0}\simeq c^{(p+2)/(2-p)}$.
Then the patch approximation term of \eqref{eq:flux-mesh} and the bounds
on its elementwise terms all sit at the common scale $Cc^{1+2/p}$, the
source term on the patch is of higher order, the number of elements
satisfies $N\simeq c^{-2}$, and
\begin{equation}\label{eq:graded-rate}
  \|\bsigma-\bsigma_h\|_{L^p(\O)}\le CN^{-1/2}.
\end{equation}
\end{prop}

\begin{proof}
Let $K\subset B_{\rho_0}(x_0)$ with $K\not\subset\omega_{x_0}$ and write
$r_K:=|x_K-x_0|$, so that $r_K\gtrsim h_{x_0}$ and, by
\eqref{eq:away-derivative}, $\int_K|\nabla_{\cP}\bsigma|^p\,dx\le
CC_{\bsigma}^ph_K^2r_K^{-2p}$.  With \eqref{eq:grading},
\[
  \bigl(h_K|\bsigma|_{W^{1,p}(K)}\bigr)^{p}
  \le CC_{\bsigma}^p\,h_K^{p+2}r_K^{-2p}
  =CC_{\bsigma}^p\,c^{p+2},
\]
since $\mu(p+2)=2p$: the choice of $\mu$ is exactly the one that makes
this independent of $r_K$.  The pole term of \eqref{eq:flux-mesh}
contributes $\bigl(C_{\bsigma}h_{x_0}^{2/p-1}\bigr)^{p}
=C_{\bsigma}^ph_{x_0}^{2-p}=C_{\bsigma}^pc^{p+2}$, the same amount, which
is what fixes $h_{x_0}$.  Counting elements,
\[
  N\simeq\frac1{c^{2}}
   \Bigl(\int_{h_{x_0}}^{\rho_0}r^{1-2\mu}\,dr+|\O|\Bigr)
  \simeq c^{-2},
\]
the integral being bounded uniformly in $c$ because $\mu<1$.  Hence the
elements of $B_{\rho_0}(x_0)$ contribute at most $CNc^{p+2}\le Cc^{p}$
to the $p$-th power of the third term of \eqref{eq:flux-mesh}, while
outside $B_{\rho_0}(x_0)$ the second bound in
\eqref{eq:away-derivative} gives at most
$c^{p}\|\nabla_{\cP}\bsigma\|^p_{L^p(\O\setminus\overline{B_{\rho_0}(x_0)})}$.
Finally $h_{x_0}\|f_1\|_{L^p(\omega_{x_0})}\le Cc^{(p+2)/(2-p)}$ and
$C_{\bsigma}h_{x_0}^{2/p-1}=C_{\bsigma}c^{(p+2)/p}$ are both $o(c)$.
Collecting the terms gives
$\|\bsigma-\bsigma_h\|_{L^p(\O)}\le Cc$, and $c\simeq N^{-1/2}$.
\end{proof}

The rate \eqref{eq:graded-rate} is the first-order complexity of a
lowest-order method in two dimensions, and the graded family reaches it
for every $p\in(1,2)$; on a quasi-uniform family Corollary~\ref{cor:global-rate}
gives only $N^{-(1/p-1/2)}$, which is $N^{-1/3}$ at $p=1.2$.  The
grading exponent $\mu=2p/(p+2)$ tends to $1$ as $p\uparrow2$ and to
$2/3$ as $p\downarrow1$.  The elements of the pole patch are not graded
but resolved at the single scale $h_{x_0}\simeq c^{(p+2)/(2-p)}$, fixed
by balancing the patch contribution against the elementwise scale, which
is the condition for the term that carries no regularity at all.

\section{Estimates for the scalar variable}
\label{sec:scalar}

The estimates of Sections~\ref{sec:discrete} and \ref{sec:approx} concern
the flux alone.  We now transfer them to the potential.  Two routes are
given: an unconditional one, which converts the $L^p$-flux error directly
and needs no elliptic regularity, and a duality argument which gains the
factor $h^{2/p'}$ under a regularity hypothesis on the dual problem.  Both
rest on the behavior of the Raviart--Thomas interpolant and of the
divergence in the exponent $p'$ conjugate to that of the flux error, and
the section closes with the resulting rate for $u-u_h$.

One identity organizes all of it.  Both $Q_hu$ and $u_h$ lie in
$P_0(\cT_h)$ and $u-Q_hu$ is orthogonal to that space, so that
\begin{equation}\label{eq:pythagoras}
  \|u-u_h\|_{L^2(\O)}^2
  =\inf_{v_h\in P_0(\cT_h)}\|u-v_h\|_{L^2(\O)}^2
   +\|Q_hu-u_h\|_{L^2(\O)}^2 .
\end{equation}
The first term is the approximation error of the scalar space and has
nothing to do with the mixed method; the second is the error the method
itself makes.  Everything below acts on the second term only.  In
particular no estimate of it can bring $\|u-u_h\|_{L^2(\O)}$ below the
best approximation, so the duality argument of
Section~\ref{subsec:scalar-refined} does not improve an order: what it
does is to show that the second term is negligible against the first.
That is not automatic.  Theorem~\ref{thm:scalar-basic} passes to the
second term whatever rate is available for the flux, and on the
quasi-uniform meshes of Corollary~\ref{cor:global-rate} that is only
$\|Q_hu-u_h\|_{L^2(\O)}=O(h^{2/p-1})$, larger than the best
approximation in the singular regime analyzed below.  Bringing the
discrete component below the approximation scale is what the duality
argument is for.

\subsection{A discrete \texorpdfstring{$L^{p'}$}{Lp'} inf--sup condition
and the basic scalar estimate}
\label{subsec:scalar-basic}

Let $H^1(\cT_h):=\{\bv\in L^2(\O)^2:\bv|_K\in H^1(K)^2\ \forall K\in\cT_h\}$ and $|\bv|_{H^1(\cT_h)}^2:=\sum_K|\bv|_{H^1(K)}^2$. We shall use the following standard Raviart--Thomas interpolation estimate. On the reference element the edge moments are controlled by the trace theorem, $|\int_{\widehat e}\widehat\bv\cdot\widehat\bn\,d\widehat s|\le\|\widehat\bv\|_{L^1(\widehat e)}\le C\|\widehat\bv\|_{H^1(\widehat K)}$, so that $\widehat\Pi$ is bounded from $H^1(\widehat K)^2$ into $L^q(\widehat K)^2$ for every $q<\infty$.

Boundedness alone does not give an error estimate. Since $\widehat\Pi$ reproduces constant fields, $\widehat\bv-\widehat\Pi\widehat\bv=(\widehat\bv-\bc)-\widehat\Pi(\widehat\bv-\bc)$ for every constant $\bc$, so that boundedness together with the embedding $H^1(\widehat K)\hookrightarrow L^q(\widehat K)$ gives $\|\widehat\bv-\widehat\Pi\widehat\bv\|_{L^q(\widehat K)}\le C_q\inf_{\bc}\|\widehat\bv-\bc\|_{H^1(\widehat K)}\le C_q|\widehat\bv|_{H^1(\widehat K)}$, the last step by the Bramble--Hilbert lemma.

Piola scaling then gives, for $2\le q<\infty$ and $\bv\in H(\divvr;\O)\cap H^1(\cT_h)$, $\|\bv-\Pi_h^{\rm rt}\bv\|_{L^q(K)}\le C_qh_K^{2/q}|\bv|_{H^1(K)}$ on each $K\in\cT_h$. Summing the $q$-th powers, using $h_K\le h$ and $\ell^2\hookrightarrow\ell^q$ for $q\ge2$, gives
\begin{equation}\label{eq:Lq-interp}
\|\bv-\Pi_h^{\rm rt}\bv\|_{L^q(\O)}\le C_qh^{2/q}|\bv|_{H^1(\cT_h)},\qquad 2\le q<\infty.
\end{equation}

The following discrete lifting is the key ingredient for estimating the scalar variable in terms of the $L^p$-error of the flux.
\begin{lem}[Discrete $L^{p'}$ lifting of the divergence]
\label{lem:discrete-Lp-lifting}
Let $1<p<2$. For every $v_h\in P_0(\cT_h)$ there exists $\btau_h\in\RT_0(\cT_h)$ such that
\begin{equation}\label{eq:discrete-Lp-lifting}
\gradt\btau_h=v_h,\qquad
\|\btau_h\|_{L^{p'}(\O)}+\|\gradt\btau_h\|_{L^2(\O)}
\le C_{\mathrm L}\|v_h\|_{L^2(\O)},
\end{equation}
where $C_{\mathrm L}$ is independent of $h$.
\end{lem}

\begin{proof}
By \cite[Thm.~1.4.5]{BernardiEtAl2024} there is a bounded
operator $\cR_0:L^2_0(\O)\to H^1_0(\O)^2$ with $\gradt\cR_0v=v$ and
$|\cR_0v|_{H^1(\O)}\le C\|v\|_{L^2(\O)}$, the full norm being
controlled by the seminorm through Poincar\'e's inequality.  The mean
of $v_h$ need not vanish, so we split it off: let
$\bar v_h:=|\O|^{-1}\int_\O v_h\,dx$, fix any $x_c\in\Rone^2$, and set
$\cR v_h:=\cR_0(v_h-\bar v_h)+\tfrac12\bar v_h(\bx-x_c)$.
In two dimensions $\gradt\bigl(\tfrac12(\bx-x_c)\bigr)=1$, so
$\gradt\cR v_h=v_h$.
Since $\bar v_h$ is the $L^2$-projection of $v_h$
onto the constants, $\|v_h-\bar v_h\|_{L^2(\O)}\le\|v_h\|_{L^2(\O)}$
and $|\bar v_h|\le|\O|^{-1/2}\|v_h\|_{L^2(\O)}$, whence
$\|\cR v_h\|_{H^1(\O)}\le C\|v_h\|_{L^2(\O)}$ with $C$ depending only on
$\O$.

Set $\btau_h:=\Pi_h^{\rm rt}\cR v_h$.  The interpolant is defined
because $\cR v_h\in H^1(\O)^2$, so that its normal trace on each edge is
integrable by the trace theorem, and the commuting property of
Section~\ref{subsec:discrete-problem} gives
$\gradt\btau_h=Q_h(\gradt\cR v_h)=Q_hv_h=v_h$, the last equality because $v_h\in P_0(\cT_h)$.
Since $\gradt\cR v_h=v_h\in L^2(\O)$, the field $\cR v_h$ belongs to $H(\divvr;\O)\cap H^1(\cT_h)$, so \eqref{eq:Lq-interp} applies with $q=p'>2$; together with the triangle inequality and the two-dimensional embedding $H^1(\O)\hookrightarrow L^{p'}(\O)$ it gives
\[
  \|\btau_h\|_{L^{p'}(\O)}
  \le\|\cR v_h\|_{L^{p'}(\O)}
   +\|\cR v_h-\Pi_h^{\rm rt}\cR v_h\|_{L^{p'}(\O)}
  \le\bigl(C+C_{p'}h^{2/p'}\bigr)\|\cR v_h\|_{H^1(\O)} ,
\]
and $h\le\operatorname{diam}\O$ bounds the bracket independently of the
mesh.  Hence $\|\btau_h\|_{L^{p'}(\O)}\le C_{\mathrm L}\|v_h\|_{L^2(\O)}$,
and together with $\|\gradt\btau_h\|_{L^2(\O)}=\|v_h\|_{L^2(\O)}$ this
proves \eqref{eq:discrete-Lp-lifting}.
\end{proof}

\begin{cor}[Discrete $L^{p'}$ divergence inf--sup condition]
\label{cor:discrete-Lp-infsup}
There exists $\beta_p>0$, independent of $h$, such that
\begin{equation}\label{eq:discrete-Lp-infsup}
\sup_{0\neq\btau_h\in\RT_0(\cT_h)}
\frac{(\gradt\btau_h,v_h)}
{\|\btau_h\|_{L^{p'}(\O)}+\|\gradt\btau_h\|_{L^2(\O)}}
\ge \beta_p\|v_h\|_{L^2(\O)}
\qquad\forall\,v_h\in P_0(\cT_h).
\end{equation}
\end{cor}

\begin{proof}
For $v_h\neq0$, choose the lifting of Lemma~\ref{lem:discrete-Lp-lifting}. Then the numerator is $\|v_h\|_{L^2(\O)}^2$, whereas the denominator is bounded by $C_{\mathrm L}\|v_h\|_{L^2(\O)}$. Hence \eqref{eq:discrete-Lp-infsup} holds with $\beta_p=C_{\mathrm L}^{-1}$.
\end{proof}

\begin{thm}[Basic $L^2$-estimate for the scalar variable]
\label{thm:scalar-basic}
Let $1<p<2$. Then
\begin{equation}\label{eq:scalar-basic-projected}
\|Q_hu-u_h\|_{L^2(\O)}
\le\frac{1}{\alpha\beta_p}\|\bsigma-\bsigma_h\|_{L^p(\O)},
\end{equation}
and hence, by \eqref{eq:pythagoras},
\begin{equation}\label{eq:scalar-basic}
\|u-u_h\|_{L^2(\O)}
\le \inf_{v_h\in P_0(\cT_h)}\|u-v_h\|_{L^2(\O)}
+\frac{1}{\alpha\beta_p}\|\bsigma-\bsigma_h\|_{L^p(\O)}.
\end{equation}
In particular, no elliptic regularity or inverse inequality is required.

\end{thm}

\begin{proof}
Set $e_h:=Q_hu-u_h\in P_0(\cT_h)$. Although $u\notin H^1(\O)$, we do have
$u\in L^2(\O)$: in two dimensions $W^{1,p}(\O)\hookrightarrow
L^{2p/(2-p)}(\O)$ and $2p/(2-p)>2$ for every $p>1$, so $Q_hu$ is the
$L^2$-orthogonal projection of $u$ and $(\gradt\btau_h,u-Q_hu)=0$.

Since $\gradt\btau_h\in P_0(\cT_h)$ for every $\btau_h\in\RT_0(\cT_h)$, the definition of $Q_h$ and the first error equation give
$(\gradt\btau_h,e_h)=(\gradt\btau_h,u-u_h)=(A^{-1}(\bsigma-\bsigma_h),\btau_h)$.
Hence, by \eqref{eq:discrete-Lp-infsup} and H\"older's inequality,
\[
\|e_h\|_{L^2(\O)}
\le \frac{1}{\beta_p}
\sup_{0\neq\btau_h\in\RT_0(\cT_h)}
\frac{|(A^{-1}(\bsigma-\bsigma_h),\btau_h)|}
{\|\btau_h\|_{L^{p'}(\O)}+\|\gradt\btau_h\|_{L^2(\O)}}
\le \frac{1}{\alpha\beta_p}\|\bsigma-\bsigma_h\|_{L^p(\O)}.
\]
This is \eqref{eq:scalar-basic-projected}, and with
\eqref{eq:pythagoras} it gives \eqref{eq:scalar-basic}.
\end{proof}

\begin{remark}[The role of the discrete $L^{p'}$ inf--sup condition]
\label{rem:why-inverse-fails}
Combining the standard $H(\divvr)$ inf--sup condition with the inverse
estimate
\[
  \|\btau_h\|_{L^{p'}(\O)}\lesssim h^{-(2/p-1)}\|\btau_h\|_{L^2(\O)}
\]
would introduce the factor $h^{-(2/p-1)}$ in
\eqref{eq:scalar-basic}.  Since the flux error is of order
$h^{2/p-1}$, this would yield no convergence for the scalar variable.
The inf--sup condition \eqref{eq:discrete-Lp-infsup} avoids this
artificial loss by constructing the divergence lifting directly in the
natural $L^{p'}$ norm.
\end{remark}

\subsection{A refined scalar estimate under dual regularity}
\label{subsec:scalar-refined}

Theorem~\ref{thm:scalar-basic} requires no dual regularity but transfers
the $L^p$-flux error directly to the scalar variable.  A duality
argument in the spirit of Douglas and Roberts \cite{DouglasRoberts1985}
sharpens it.  By \eqref{eq:pythagoras} what such an argument can do is
to make the second term negligible against the first, and that is what
the argument below establishes.

\begin{assumption}[Dual regularity]
\label{ass:dual}
For every $\varphi\in L^2(\O)$, the solution $\psi\in H^1_0(\O)$ of $-\gradt(A\nabla\psi)=\varphi$ satisfies $A\nabla\psi\in H^1(\cT_h)$ and
$|A\nabla\psi|_{H^1(\cT_h)}\le C_{\rm dual}\|\varphi\|_{L^2(\O)}$, where
$C_{\rm dual}$ is independent of $h$.
\end{assumption}

The membership $\nabla\psi\in L^{p'}(\O)^2$ that the Green formula
needs follows from the assumption as it stands, as the proof of
Theorem~\ref{thm:scalar-refined} shows.

\begin{remark}[Scope of Assumption~\ref{ass:dual}]
\label{rem:dual-scope}
For $A=I$ the assumption is the usual $H^2$ regularity of the dual
problem, for which convexity of $\O$ is sufficient
\cite[Thm.~1.3.14]{BernardiEtAl2024}.  For a piecewise smooth
coefficient it becomes a piecewise $H^2$ regularity relative to the
coefficient interfaces, the mesh being again fitted to $\cP$.  It does
not involve $p$, and it is the only place where any regularity of the
dual problem is used.

On a nonconvex polygon the full $H^2$ bound is not available for general
$L^2$ data, a reentrant corner of opening $\omega$ carrying the singular
exponent $\pi/\omega<1$.
\end{remark}
\begin{thm}[Supercloseness of the projected potential]
\label{thm:scalar-refined}
Let $1<p<2$, assume $f_1\in L^2(\O)$, and let Assumption~\ref{ass:dual} hold. Then
\begin{equation}\label{eq:scalar-refined}
\|Q_hu-u_h\|_{L^2(\O)}
\le C\Bigl(h^{2/p'}\|\bsigma-\bsigma_h\|_{L^p(\O)}
+h\|f_1-Q_hf_1\|_{L^2(\O)}\Bigr),
\end{equation}
where $C$ is independent of $h$; with \eqref{eq:pythagoras} this bounds
$\|u-u_h\|_{L^2(\O)}$.
\end{thm}

\begin{proof}
Set $e_h:=Q_hu-u_h\in P_0(\cT_h)$ and let $\psi\in H^1_0(\O)$ solve $-\gradt(A\nabla\psi)=e_h$. Define $\bw:=-A\nabla\psi$. Then $\gradt\bw=e_h$ and, by Assumption~\ref{ass:dual}, $\bw\in H(\divvr;\O)\cap H^1(\cT_h)$, so $\Pi_h^{\rm rt}\bw$ is well defined. The commuting property gives
$\gradt\Pi_h^{\rm rt}\bw=Q_h(\gradt\bw)=e_h$.
Since $\gradt\Pi_h^{\rm rt}\bw\in P_0(\cT_h)$, the first error equation yields
\begin{align}
\|e_h\|_{L^2(\O)}^2
&=(u-u_h,\gradt\Pi_h^{\rm rt}\bw) \notag\\
&=(A^{-1}(\bsigma-\bsigma_h),\bw)
 +(A^{-1}(\bsigma-\bsigma_h),\Pi_h^{\rm rt}\bw-\bw)
=:T_1+T_2.                                      \label{eq:scalar-split}
\end{align}

On each element Assumption~\ref{ass:dual} gives $\bw\in H^1(K)^2$, and the two-dimensional embedding $H^1(K)\hookrightarrow L^{p'}(K)$ then puts $\nabla\psi=-A^{-1}\bw$ in $L^{p'}(\O)^2$ for the mesh at hand; this is what makes $T_1$ defined and the next step legitimate, and no bound on $\|\nabla\psi\|_{L^{p'}(\O)}$ is needed, the estimate below using only the energy bound. Since $A^{-1}\bw=-\nabla\psi$ and $\gradt(\bsigma-\bsigma_h)=f_1-Q_hf_1$, Green's formula gives
$T_1=(f_1-Q_hf_1,\psi)=(f_1-Q_hf_1,\psi-Q_h\psi)$.
Hence the approximation property of $Q_h$ and the standard energy estimate for the dual problem imply
$|T_1|\le Ch\|f_1-Q_hf_1\|_{L^2(\O)}\|\nabla\psi\|_{L^2(\O)}
\le Ch\|f_1-Q_hf_1\|_{L^2(\O)}\|e_h\|_{L^2(\O)}$.
For the second term, H\"older's inequality and \eqref{eq:Lq-interp} with $q=p'$ give
\[
|T_2|
\le \alpha^{-1}\|\bsigma-\bsigma_h\|_{L^p(\O)}
\|\Pi_h^{\rm rt}\bw-\bw\|_{L^{p'}(\O)}
\le Ch^{2/p'}\|\bsigma-\bsigma_h\|_{L^p(\O)}
|\bw|_{H^1(\cT_h)}.
\]
By Assumption~\ref{ass:dual},
$|\bw|_{H^1(\cT_h)}\le C_{\rm dual}\|e_h\|_{L^2(\O)}$. Inserting these estimates into \eqref{eq:scalar-split} and dividing by
$\|e_h\|_{L^2(\O)}$ proves \eqref{eq:scalar-refined}.
\end{proof}

\begin{remark}[Comparison of the scalar estimates]
\label{rem:scalar-compare}
Theorem~\ref{thm:scalar-basic} requires no dual regularity and gives $\|Q_hu-u_h\|_{L^2}=O(h^{2/p-1})$ when the flux error has its sharp rate. Under Assumption~\ref{ass:dual}, Theorem~\ref{thm:scalar-refined} gains the factor $h^{2/p'}$, and $h^{2/p'}h^{2/p-1}=h$.
Thus the flux contribution to the projected scalar error is first order, independently of $p$; the data term in \eqref{eq:scalar-refined} is also $O(h)$ for fixed $f_1\in L^2(\O)$.
\end{remark}

\begin{prop}[$P_0$ approximation of the potential]
\label{prop:P0-log}
Let $\cT_h$ be shape-regular and fitted, let Assumption~\ref{ass:S}
hold, and let $\omega_{x_0}\subset B_{\rho_0}(x_0)$, which holds once
$h_{x_0}$ is small enough.  Suppose that $u|_K\in H^{t_K}(K)$ with an
elementwise index $t_K\in(0,1]$ for every $K\in\cT_h$ not contained in
$\omega_{x_0}$.  Then
\begin{equation}\label{eq:P0-local}
  \|u-Q_hu\|_{L^2(\O)}
  \le C\Bigl((1+C_{\bsigma})h_{x_0}
   +\Bigl(\sum_{K\not\subset\omega_{x_0}}
     h_K^{2t_K}|u|_{H^{t_K}(K)}^{2}\Bigr)^{1/2}\Bigr).
\end{equation}
If in addition $\{\cT_h\}$ is quasi-uniform, \eqref{eq:away-derivative}
holds and $t_K=1$, then $\|u-Q_hu\|_{L^2(\O)}\le Ch|\log h|^{1/2}$; and
if moreover
\begin{equation}\label{eq:log-leading}
  u(x)=-\gamma\log|x-x_0|+v(x)\ \text{ on }B_{\rho_0}(x_0),
  \qquad \gamma\neq0,\quad \nabla v\in L^\infty(B_{\rho_0}(x_0)),
\end{equation}
then also $\|u-Q_hu\|_{L^2(\O)}\ge ch|\log h|^{1/2}$ for all $h$ small
enough.
\end{prop}

\begin{proof}
Since $Q_h$ is the elementwise mean,
$\|u-Q_hu\|_{L^2(\O)}^2=\sum_K\|u-Q_hu\|_{L^2(K)}^2$, and on an element
not contained in $\omega_{x_0}$ the Bramble--Hilbert lemma gives
$\|u-Q_hu\|_{L^2(K)}\le Ch_K^{t_K}|u|_{H^{t_K}(K)}$.  On the patch,
$\nabla u=-A^{-1}(\bsigma+\bff_2)$, where Assumption~\ref{ass:S} bounds
the first field and $|\bff_2(x)|=\frac1{2\pi}|x-x_0|^{-1}$ the second,
so that $|\nabla u(x)|\le C(1+C_{\bsigma})|x-x_0|^{-1}$ and
$\|\nabla u\|_{L^p(\omega_{x_0})}\le C(1+C_{\bsigma})h_{x_0}^{2/p-1}$;
hence
\[
  \|u-Q_hu\|_{L^2(\omega_{x_0})}
  \le Ch_{x_0}^{2-2/p}\|\nabla u\|_{L^p(\omega_{x_0})}
  \le C(1+C_{\bsigma})h_{x_0} ,
\]
which is \eqref{eq:P0-local}.  On a quasi-uniform family with $t_K=1$
the sum is $h^2\|\nabla u\|_{L^2(\O\setminus\omega_{x_0})}^2$, and
$\nabla u\in L^2$ outside $B_{\rho_0}(x_0)$ by the second bound in
\eqref{eq:away-derivative} and the two-dimen\-sional embedding
$W^{1,p}(\O_j)\hookrightarrow L^{2p/(2-p)}(\O_j)$, while the same bound
on $|\nabla u|$ gives
\[
  \|\nabla u\|_{L^2(B_{\rho_0}(x_0)\setminus\omega_{x_0})}
  \le C|\log h|^{1/2} .
\]
This gives the quasi-uniform upper bound.

For the lower bound, let $K$ be an element with
$r_K:=\mathrm{dist}(x_0,K)\in[\rho h,\rho_1]$, where $\rho$ and
$\rho_1$ are fixed below.  Write $u=\ell_K+R_K+v$ on $K$, with $\ell_K$
the first-order Taylor polynomial of $-\gamma\log|x-x_0|$ at the
barycenter of $K$.  On the two-dimensional space of affine functions
modulo constants the map $w\mapsto\|w-Q_hw\|_{L^2(K)}$ is a norm, so by
equivalence of norms on the reference element and scaling
$\|\ell_K-Q_h\ell_K\|_{L^2(K)}\ge c_0h_K|\nabla\ell_K||K|^{1/2}$ with
$c_0$ depending only on the shape-regularity parameter, and
$|\nabla\ell_K|\ge|\gamma|/(2r_K)$.  Moreover
$|R_K|\le Ch_K^2\sup_K|D^2\log|x-x_0||\le C|\gamma|h_K^2r_K^{-2}$ and
$\|v-Q_hv\|_{L^2(K)}\le Ch_K\|\nabla v\|_{L^\infty}|K|^{1/2}$.  Since
$Q_h$ is a projection,
\[
  \|u-Q_hu\|_{L^2(K)}
  \ \ge\ h_K|K|^{1/2}\Bigl(\frac{c_0|\gamma|}{2r_K}
   -\frac{C|\gamma|h_K}{r_K^{2}}-C\|\nabla v\|_{L^\infty}\Bigr)
  \ \ge\ \frac{c_0|\gamma|}{4}\,\frac{h_K}{r_K}\,|K|^{1/2} ,
\]
the last step on choosing $\rho$ so large that
$C h_K/r_K\le c_0/8$ and $\rho_1$ so small that
$C\|\nabla v\|_{L^\infty}\le c_0|\gamma|/(8\rho_1)$.  Summing the squares
over these elements and using quasi-uniformity,
\[
  \|u-Q_hu\|_{L^2(\O)}^2
  \ \ge\ c\,h^2\!\!\int_{2\rho h<|x-x_0|<\rho_1/2}\!\!\frac{dx}{|x-x_0|^2}
  \ =\ c\,h^2\log\frac{\rho_1}{4\rho h} ,
\]
which is $\ge ch^2|\log h|$ for $h$ small.
\end{proof}

\begin{thm}[Potential estimate on a fitted mesh]
\label{thm:scalar-mesh}
Let $1<p<2$, let $\cT_h$ be shape-regular and fitted, let
$f_1\in L^2(\O)$, let Assumption~\ref{ass:S} (the pointwise bound on
$\bsigma$) and Assumption~\ref{ass:dual} hold, let
$\omega_{x_0}\subset B_{\rho_0}(x_0)$, and let $u|_K\in H^{t_K}(K)$ with
an elementwise index $t_K\in(0,1]$ for every $K\in\cT_h$ not contained
in $\omega_{x_0}$.  Then
\begin{equation}\label{eq:scalar-mesh}
\begin{split}
  \|u-u_h\|_{L^2(\O)}
  \le C\Bigl(&(1+C_{\bsigma})h_{x_0}
   +\Bigl(\sum_{K\not\subset\omega_{x_0}}
      h_K^{2t_K}|u|_{H^{t_K}(K)}^{2}\Bigr)^{1/2}\\
  &+h^{2/p'}\|\bsigma-\bsigma_h\|_{L^p(\O)}
   +h\|f_1-Q_hf_1\|_{L^2(\O)}\Bigr).
\end{split}
\end{equation}
\end{thm}

\begin{proof}
By \eqref{eq:pythagoras} the squares of $\|u-Q_hu\|_{L^2(\O)}$ and
$\|Q_hu-u_h\|_{L^2(\O)}$ add up to $\|u-u_h\|_{L^2(\O)}^2$, and
$(a^2+b^2)^{1/2}\le a+b$.  The first term is bounded by
\eqref{eq:P0-local} and the second by \eqref{eq:scalar-refined}.
\end{proof}

Estimate \eqref{eq:scalar-mesh} is the counterpart of
\eqref{eq:flux-mesh} and is read the same way: a pole term of pure size
$(1+C_{\bsigma})h_{x_0}$, carrying no smoothness index, and the
elementwise regularity of $u$ everywhere else.  What it has in addition
is the third term, which transfers the flux error, and by
Theorem~\ref{thm:flux-mesh} that term is again of the same shape.  The
$1$ in the pole term is the contribution of $\bff_2$ itself, and it is
why that term does not vanish with $C_{\bsigma}$: matching removes the
leading singularity of the modified flux, not the logarithm of the
potential.  The next corollary is the quasi-uniform case.

\begin{cor}[The scalar error is asymptotically a best approximation]
\label{cor:scalar-rate}
Assume in addition that $f_1\in L^2(\O)$.  Then, under the hypotheses of
Corollary~\ref{cor:global-rate} and Assumption~\ref{ass:dual},
\[
  \|Q_hu-u_h\|_{L^2(\O)}\le Ch,
  \qquad
  \|u-u_h\|_{L^2(\O)}\le Ch|\log h|^{1/2} ,
\]
and if \eqref{eq:log-leading} holds as well, then
\begin{equation}\label{eq:scalar-rate}
  \|u-u_h\|_{L^2(\O)}
  =\Bigl(1+O\bigl(|\log h|^{-1}\bigr)\Bigr)
   \inf_{v_h\in P_0(\cT_h)}\|u-v_h\|_{L^2(\O)}
  \ \simeq\ h|\log h|^{1/2} .
\end{equation}
\end{cor}

\begin{proof}
By Theorem~\ref{thm:scalar-refined} and Corollary~\ref{cor:global-rate},
$h^{2/p'}\|\bsigma-\bsigma_h\|_{L^p(\O)}\le Ch^{2/p'}h^{2/p-1}=Ch$,
while $h\|f_1-Q_hf_1\|_{L^2(\O)}\le Ch$ because $f_1\in L^2(\O)$; this
is the first bound, and with Proposition~\ref{prop:P0-log} and
\eqref{eq:pythagoras} it gives the second.  Under
\eqref{eq:log-leading} the same identity reads
$\|u-u_h\|^2=a_h^2+b_h^2$ with $a_h\simeq h|\log h|^{1/2}$ the best
approximation and $b_h\le Ch$, so $b_h^2/a_h^2\le C|\log h|^{-1}$ and
\eqref{eq:scalar-rate} follows.
\end{proof}

\begin{remark}
\label{rem:scalar-log}
The logarithm in \eqref{eq:scalar-rate} is a property of the potential
and of $P_0(\cT_h)$, not of the mixed method: by
\eqref{eq:pythagoras} no discrete function does better, and the
proof of Proposition~\ref{prop:P0-log} shows where it comes from.  Every
dyadic annulus $|x-x_0|\simeq2^{-j}$ between the mesh scale $h$ and a
fixed radius independent of $h$ contributes the same amount $ch^2$ to
$\|u-Q_hu\|_{L^2(\O)}^2$, and there are $\simeq|\log h|$ of them; the
element carrying the pole contributes $O(h^2)$ to that same square, that
is $O(h)$ in the norm, and so counts for one annulus and no more.  So the two limitations of the analysis sit at opposite ends of
the scales: the flux rate is decided on one element
(Corollary~\ref{cor:global-lower}), the scalar rate by all of them at
once.  Grading removes both, the second because it turns
$\sum_jh_j^2$ into a convergent geometric series.

\end{remark}

\begin{remark}[The two errors compared]
\label{rem:two-rates}
Both rates are two-sided.  Under the hypotheses of
Corollary~\ref{cor:global-lower} and \eqref{eq:log-leading}, a
quasi-uniform family gives
$\|\bsigma-\bsigma_h\|_{L^p(\O)}\simeq h^{2/p-1}$ by
Corollaries~\ref{cor:global-lower} and~\ref{cor:global-rate}, and
$\|u-u_h\|_{L^2(\O)}\simeq h|\log h|^{1/2}$ by \eqref{eq:scalar-rate}.
Since $2/p-1<1$ for every $p>1$, the flux is the worse approximated of
the two, and what is compared here is the two errors and not two upper
bounds.  A nonvanishing leading-order mismatch is what produces the slow
flux rate: for a matched splitting the pole no longer limits the
modified flux, and under the corresponding global regularity
$\bsigma\in H(\divvr;\O)$ and the standard mixed theory applies with its
usual rates, whereas the logarithm of the
potential belongs to $u$ itself
and no splitting removes it (Section~\ref{subsec:phiA}).  It is also
where the cost is felt, the flux being the unknown a mixed method is
built around.
\end{remark}

\begin{remark}[The role of the exponent $p$]
\label{rem:p-tension}
The exponent $p$ is a parameter of the analysis only: it does not enter
the discrete system \eqref{eq:mixed-discrete}, and by
Remark~\ref{rem:scalar-compare} the refined scalar order is
$p$-independent.  What does depend on $p$ are the constants.  They do so through
Hypothesis~(SP), through the $L^{p'}$ lifting and interpolation
estimates, where $\beta_p^{-1}=C_{\mathrm L}$ of
Corollary~\ref{cor:discrete-Lp-infsup} and $C_{p'}$ of
\eqref{eq:Lq-interp} grow as $p'\to\infty$ through the two-dimensional
embedding $H^1(\O)\hookrightarrow L^{p'}(\O)$, and, in
Section~\ref{sec:aposteriori}, through the constant $C_{\rm reg}$ of
Assumption~\ref{ass:A3} at $q=p$.
\end{remark}

\begin{remark}[Graded and adaptive meshes]
\label{rem:graded}
The rates of Corollaries~\ref{cor:global-rate} and
\ref{cor:scalar-rate} are stated for quasi-uniform families.  Both come
from the point singularity, but in different ways: the flux rate is
fixed locally, by the size $h_{x_0}$ of the element carrying the pole,
while the logarithm in the scalar error is accumulated over the scales
between $h_{x_0}$ and the far field.  Grading toward $x_0$ reaches both,
making $h_{x_0}$ small independently of the far field and letting the
element sizes decrease geometrically across the intermediate scales.
Such families have therefore to be measured against the fitted-mesh
estimates \eqref{eq:flux-mesh} and \eqref{eq:scalar-mesh}, which hold on
any shape-regular fitted mesh, and not against the quasi-uniform
rates.

The separation made in Section~\ref{subsec:equidistribution}, pure
size at the pole and elementwise regularity everywhere else, holds for
the potential as well:
\eqref{eq:P0-local} has the pole term $(1+C_{\bsigma})h_{x_0}$, again
pure size, and elsewhere the elementwise regularity of $u$.  What is
quasi-uniform is only the evaluation of those sums:
Corollary~\ref{cor:scalar-rate} and the factor $|\log h|^{1/2}$ count
the dyadic scales between the mesh scale $h$ and a fixed radius
independent of $h$ at a common mesh size, and the
factor disappears as soon as the sizes decrease geometrically toward the
pole (Remark~\ref{rem:scalar-log}).  The numerical experiments of
Section~\ref{sec:numerics} test both points on adaptively refined
meshes.

\end{remark}

\section{A posteriori error estimation in \texorpdfstring{$L^p$}{Lp}}
\label{sec:aposteriori}

The estimator of this section is modeled on the a posteriori theory of
first-order system least-squares methods.  There the error in the
natural norm is equivalent to a computable residual, and no jump
residuals, bubble functions or mesh-dependent weights are needed; the
equivalence rests on the stability of the underlying second-order
problem.  We establish the counterpart of that equivalence in the
Lebesgue scale, where the Hilbert structure is unavailable, and then
turn it into an estimator for the mixed flux supplemented by an
recovered potential field.

\subsection{Residual norm equivalence}
\label{subsec:residual-equivalence}

For $1<q<\infty$ and $q'=q/(q-1)$ recall $H^q(\divvr;\O)$ from
\eqref{eq:Hr-div}, together with the norm \eqref{def:W-1norm-q} of $W^{-1,q}(\O)$ and the bound
\eqref{eq:div-Wm1q}, both stated at a general exponent in
Section~\ref{subsec:duran}.  In particular
$\gradt(A\nabla v)\in W^{-1,q}(\O)$ whenever $v\in W^{1,q}_0(\O)$, so
Assumption~\ref{ass:A3} below may be applied to it.

\begin{assumption}[$W^{1,q}$ well-posedness of the Dirichlet problem]
\label{ass:A3}
Let $1<q<\infty$. For every $f\in W^{-1,q}(\O)$, the problem
$-\gradt(A\nabla z)=f$ in $\O$ with $z=0$ on $\p\O$ has a unique
solution $z\in W_0^{1,q}(\O)$, and
$\|\nabla z\|_{L^q(\O)}\le C_{\rm reg}\|f\|_{W^{-1,q}(\O)}$ with
$C_{\rm reg}$ independent of $f$.
\end{assumption}

\begin{remark}
\label{rem:A3}
Assumption~\ref{ass:A3} is a standard \(W^{1,q}\) well-posedness of the Dirichlet problem. For \(A=I\), it holds on an arbitrary polygon, without restrictions on the angles, for \(4/3<q<4\), and on a convex polygon for every \(1<q<\infty\); see \cite[Sec.~1.3.1.2]{BernardiEtAl2024}. 
We use Assumption~\ref{ass:A3} only for the reliability estimate below.
\end{remark}

\begin{thm}[Residual norm equivalence]
\label{thm:norm-equivalence}
Let $1<q<\infty$ and let Assumption~\ref{ass:A3} hold.  For every
$\btau\in H^q(\divvr;\O)$ and $v\in W_0^{1,q}(\O)$,
\begin{equation}\label{eq:norm-equivalence-v}
\|\btau\|_{L^q(\O)}+\|v\|_{W^{1,q}(\O)}+\|\gradt\btau\|_{L^q(\O)}
\simeq
\|A^{-1/2}\btau+A^{1/2}\nabla v\|_{L^q(\O)}
+\|\gradt\btau\|_{L^q(\O)} ,
\end{equation}
with constants depending only on $\alpha,\beta,q,\O$, and $C_{\rm reg}$.
\end{thm}

\noindent
The left-hand side is the norm of the pair $(\btau,v)$ in which the
problem is well posed, and the right-hand side is the residual of the
first-order system.

\begin{proof}
The upper bound follows from the triangle inequality and the bounds
$\|A^{1/2}\|\le\beta^{1/2}$, $\|A^{-1/2}\|\le\alpha^{-1/2}$.
For the lower bound apply Assumption~\ref{ass:A3} to
$f:=-\gradt(A\nabla v)$, which
belongs to $W^{-1,q}(\O)$ by \eqref{eq:div-Wm1q} and acts by
$\langle f,\chi\rangle=(A\nabla v,\nabla\chi)$; the solution it returns
is $z=v$.  Together with \eqref{def:W-1norm-q}, Poincar\'e's inequality
and the bounds on $A^{1/2}$ and $A^{-1/2}$ this gives
\[
\|A^{1/2}\nabla v\|_{L^q(\O)}
\le C\sup_{0\ne\chi\in W_0^{1,q'}(\O)}
\frac{|(A\nabla v,\nabla\chi)|}{\|A^{1/2}\nabla\chi\|_{L^{q'}(\O)}}.
\]
For every $\chi\in W_0^{1,q'}(\O)$,
\[
(A\nabla v,\nabla\chi)
=\bigl(A^{-1/2}\btau+A^{1/2}\nabla v,A^{1/2}\nabla\chi\bigr)
+(\gradt\btau,\chi).
\]
Hence H\"older's and Poincar\'e's inequalities, together with the bound
$\|\nabla\chi\|_{L^{q'}(\O)}\le\alpha^{-1/2}
 \|A^{1/2}\nabla\chi\|_{L^{q'}(\O)}$, yield
\[
|(A\nabla v,\nabla\chi)|
\le C\Bigl(\|A^{-1/2}\btau+A^{1/2}\nabla v\|_{L^q(\O)}
+\|\gradt\btau\|_{L^q(\O)}\Bigr)\|A^{1/2}\nabla\chi\|_{L^{q'}(\O)}.
\]
It follows that
\[
\|A^{1/2}\nabla v\|_{L^q(\O)}
\le C\Bigl(\|A^{-1/2}\btau+A^{1/2}\nabla v\|_{L^q(\O)}
+\|\gradt\btau\|_{L^q(\O)}\Bigr).
\]
Finally,
$A^{-1/2}\btau=\bigl(A^{-1/2}\btau+A^{1/2}\nabla v\bigr)-A^{1/2}\nabla v$,
and
\[
  \begin{aligned}
  \|\btau\|_{L^q(\O)}&\le\beta^{1/2}\|A^{-1/2}\btau\|_{L^q(\O)},\\
  \|v\|_{W^{1,q}(\O)}&\le C\|\nabla v\|_{L^q(\O)}
   \le C\alpha^{-1/2}\|A^{1/2}\nabla v\|_{L^q(\O)},
  \end{aligned}
\]
the second by Poincar\'e's inequality, so that the same right-hand side
controls both terms.
\end{proof}

\begin{remark}[Relation to least-squares norm equivalence]
\label{rem:lsq}
Theorem~\ref{thm:norm-equivalence} is the $L^q$ counterpart of the $L^2$
norm equivalence used in first-order system least-squares methods.  The
source is again \cite{Zhang2023}: the equivalence
\cite[(4.16)]{Zhang2023} is stated there for exactly the first-order
system produced by the splitting of Section~\ref{subsec:splitting-def},
so the two ingredients this paper takes from that work, the
splitting and the norm equivalence, belong together.  See also
\cite[Thm.~16]{FuhrerHeuerKarkulik2022} for the same equivalence in
the presence of an oscillation term.  Indeed, when $q=2$, after taking square roots and
using the equivalence between the sum and the Euclidean combination of
the two residual norms, \eqref{eq:norm-equivalence-v} reduces to the
standard equivalence between the least-squares residual norm
$\|A^{-1/2}\btau+A^{1/2}\nabla v\|_{L^2(\O)}+\|\gradt\btau\|_{L^2(\O)}$
and the $H(\divvr;\O)\times H_0^1(\O)$ graph norm.  The proof follows
the same principle as in \cite{Zhang2023}: the stability of the
underlying second-order problem first controls the potential variable
through a supremum argument, after which the flux is recovered from the
constitutive residual.  Here the $W^{1,q}$ well-posedness in
Assumption~\ref{ass:A3} replaces the Hilbert-space stability estimate
and yields the corresponding norm equivalence in the Lebesgue scale.
\end{remark}

\subsection{The estimator}

In \eqref{eq:norm-equivalence-v} the flux slot is filled by the discrete
solution: $\btau:=\bsigma-\bsigma_h$ lies in $H^p(\divvr;\O)$, and the
residual it produces is computable, since
$\bsigma=-A\nabla u-\bff_2$ gives
\[
  A^{-1/2}(\bsigma-\bsigma_h)+A^{1/2}\nabla(u-w)
  =-\bigl[A^{-1/2}(\bsigma_h+\bff_2)+A^{1/2}\nabla w\bigr]
\]
for any $w$.  What remains to be supplied is the second slot, a
potential approximating $u$.  We therefore define the estimator for an
arbitrary such potential and produce one only afterwards.

The potentials we admit are those with
\begin{equation}\label{eq:admissible}
  w\in W^{1,p}(\O),\qquad w|_{\p\O}=g .
\end{equation}
Since the solution satisfies $u\in W^{1,p}(\O)$ with $u|_{\p\O}=g$, such
a $w$ gives $u-w\in W_0^{1,p}(\O)$ at once.  Both requirements in
\eqref{eq:admissible} are stated in spaces already fixed in
Section~\ref{sec:splitting}: the trace equality holds in
$W^{1-1/p,p}(\p\O)$, which is where the datum $g$ lives.  For each
$K\in\cT_h$ set
\begin{equation}\label{eq:estimator}
\eta_K^p:=
\|A^{-1/2}(\bsigma_h+\bff_2)+A^{1/2}\nabla w\|_{L^p(K)}^p
+\|f_1-\gradt\bsigma_h\|_{L^p(K)}^p,
\qquad
\eta:=\Bigl(\sum_{K\in\cT_h}\eta_K^p\Bigr)^{1/p}.
\end{equation}
The first residual vanishes for the exact pair because
$
A^{-1/2}(\bsigma+\bff_2)+A^{1/2}\nabla u=0.
$
Both variables are present in that residual, as they must be: it is the
constitutive law, and it pairs a flux with a potential gradient.  What
does not enter is $u_h$ itself, and for a structural reason, not for
a preference: $u_h$ is piecewise constant and carries no gradient, so
the potential side has to be supplied by a separate $w$, produced in
Section~\ref{subsec:potential-field}.
Moreover $\gradt\bsigma_h=Q_hf_1$ by
\eqref{eq:discrete-conservation}, so
$\|f_1-\gradt\bsigma_h\|_{L^p(K)}=\|f_1-Q_hf_1\|_{L^p(K)}$ on every
element.  The second term of \eqref{eq:estimator} is therefore a pure data
oscillation: it is computable from $f_1$ alone, involves no
post-processing of $\bsigma_h$, and vanishes identically whenever
$f_1$ is elementwise constant, in particular for the pure Dirac
problem $f_1\equiv0$, for which $\eta$ reduces to the constitutive
residual.  In the nonvanishing leading-order regime covered by
\eqref{eq:global-lower}, and on a quasi-uniform family, it is also of
higher order than the flux error: if $f_1|_K\in W^{1,p}(K)$ on every
element then $\|f_1-Q_hf_1\|_{L^p(\O)}\le
Ch\bigl(\sum_K|f_1|_{W^{1,p}(K)}^p\bigr)^{1/p}$, while
$\|\bsigma-\bsigma_h\|_{L^p(\O)}\ge ch^{2/p-1}$ and $2/p-1<1$ for every
$p>1$, so the oscillation is asymptotically negligible in the global
estimator.  Outside that
case only the upper bound $\|\bsigma-\bsigma_h\|_{L^p(\O)}\le
Ch^{2/p-1}$ is available, and an upper bound on the flux error says
nothing about how small it actually is: when the splitting is matched
the flux error can be far smaller, and in the situation of
Remark~\ref{rem:num-matched} it vanishes identically.  Since the
oscillation is computable, the comparison can in any case be made where
it matters.  It cannot be dropped from
\eqref{eq:norm-equivalence-v}, whose lower bound uses
$\|\gradt\btau\|_{L^p(\O)}$.

\begin{thm}[Reliability and efficiency for the augmented pair]
\label{thm:estimator}
Let $1<p<2$ and let Assumption~\ref{ass:A3} hold with $q=p$.  Then, for
every $w$ satisfying \eqref{eq:admissible},
\begin{equation}\label{eq:estimator-equiv}
\|\bsigma-\bsigma_h\|_{L^p(\O)}
+\|u-w\|_{W^{1,p}(\O)}
+\|f_1-\gradt\bsigma_h\|_{L^p(\O)}
\simeq \eta,
\end{equation}
with constants independent of $h$ and of $w$.
\end{thm}

\begin{proof}
Set $\btau:=\bsigma-\bsigma_h$ and $v:=u-w$.  Then
$v\in W_0^{1,p}(\O)$ by \eqref{eq:admissible}, while
$\gradt\btau=f_1-\gradt\bsigma_h\in L^p(\O)$, so
$\btau\in H^p(\divvr;\O)$.  The displayed identity above shows that
$A^{-1/2}\btau+A^{1/2}\nabla v$ is the first residual in
\eqref{eq:estimator} up to sign, and
\eqref{eq:estimator-equiv} follows from
Theorem~\ref{thm:norm-equivalence} together with the equivalence
between the sum and the $\ell^p$ combination of the two residual norms.
\end{proof}

\begin{remark}[Local efficiency for the augmented error]
For every $K\in\cT_h$,
\[
\eta_K\le C\Bigl(
\|\bsigma-\bsigma_h\|_{L^p(K)}
+\|\nabla(u-w)\|_{L^p(K)}
+\|f_1-\gradt\bsigma_h\|_{L^p(K)}
\Bigr).
\]
Thus the residual is locally efficient for the pair
$(\bsigma_h,w)$, with no bubble functions, jump residuals or
mesh-dependent weights; local efficiency for the mixed flux alone
depends, as above, on the quality of the recovery.  The global
reliability estimate, on the other hand, uses
Assumption~\ref{ass:A3}.
\end{remark}

\begin{remark}[What the estimator measures]
The estimator controls the error of the pair $(\bsigma_h,w)$, not that
of the mixed approximation alone.  The quality of $w$ therefore
enters, and a poor choice inflates both sides of
\eqref{eq:estimator-equiv} equally; the equivalence itself is
unaffected.  Reliability for the mixed flux alone is contained in it,
since $\|\bsigma-\bsigma_h\|_{L^p(\O)}\le C\eta$ for every $w$, however
it is produced.  Efficiency for the flux alone is not: it would need in
addition an estimate of the recovery, controlling $\|u-w\|_{W^{1,p}(\O)}$
by the mixed error and the data oscillation.  Nothing in
\eqref{eq:norm-equivalence-v} supplies one.
\end{remark}

\subsection{The potential field}
\label{subsec:potential-field}

Throughout this subsection $V_h^c$ denotes the continuous piecewise
linear finite element space on $\cT_h$, and we assume, to avoid an
inessential boundary-data term, that $g$ is represented exactly on
$\p\O$ by the trace of $V_h^c$.

Let $u_c\in V_h^c$ satisfy
\begin{equation}\label{eq:conforming-potential}
u_c=g\quad\text{on }\p\O,\qquad
(A\nabla u_c,\nabla v_h)=(f_1,v_h)+v_h(x_0)
\quad\forall\,v_h\in V_h^c\cap H_0^1(\O),
\end{equation}
and take $w:=u_c$.  The point source is well defined here because every
$v_h\in V_h^c$ is continuous at $x_0$; this is the escape of
Section~\ref{subsec:whyf2}, which is available to a conforming method
and not to a mixed one.  It satisfies \eqref{eq:admissible}, so
Theorem~\ref{thm:estimator} applies.  This is the potential used in the
computations of Section~\ref{sec:numerics}.  A conforming solve returns
the potential together with its gradient, so nothing here changes when
the operator carries lower-order terms and the residual sees $u$ itself.

\begin{remark}[Recovery without a potential]
\label{rem:hcurl}
For the operator treated here the field can also be produced without
forming a potential at all, in the $H(\mathrm{curl})$-conforming
N\'ed\'elec space; this is what the equilibrated estimators for mixed
methods of \cite{CaiCaiZhang2020} do.  Let $\cN_1(\cT_h)$ be the
N\'ed\'elec space of the first kind of index one, which is the rotation
$R\,\RT_0(\cT_h)$ of the flux space of the method itself, normal
continuity across an edge becoming tangential continuity and $\gradt$
becoming $\nabla\times$, and let
\[
  \cN_1^{g}(\cT_h):=\bigl\{\brho_h\in\cN_1(\cT_h):\
    \nabla\times\brho_h=0\ \text{in }\O,\quad
    \brho_h\cdot\mathbf t=\p_t g\ \text{on }\p\O\bigr\},
\]
with $\p_t$ the tangential derivative along $\p\O$ and $\cN_1^{0}(\cT_h)$
the same space for $g=0$.  Taking $\brho_h\in\cN_1^{g}(\cT_h)$ with
\begin{equation}\label{eq:hcurl-recovery}
  (A\brho_h,\btau_h)=-(\bsigma_h+\bff_2,\btau_h)
  \qquad\forall\,\btau_h\in\cN_1^{0}(\cT_h)
\end{equation}
selects one such field: the system is uniquely solvable, its matrix
being the $A$-weighted mass matrix on $\cN_1^{0}(\cT_h)$, and every entry
is finite because the discrete fields are bounded elementwise and
$\bff_2\in L^p(\O)^2$.

At this index the two constructions use the same discrete field space.  On a simply connected domain the
curl-free subspace of $\cN_1(\cT_h)$ is exactly $\nabla V_h^c$, by the
discrete exact sequence
$V_h^c\xrightarrow{\nabla}\cN_1(\cT_h)\xrightarrow{\nabla\times}P_0(\cT_h)$,
so every $\brho_h\in\cN_1^{g}(\cT_h)$ is $\nabla w_h$ with
$w_h\in V_h^c$; the prescribed tangential trace makes $w_h-g$ constant on
the connected boundary, and adjusting the additive constant, which does
not change $\brho_h$, gives $w_h|_{\p\O}=g$, which is
\eqref{eq:admissible}.  What differs is the
system solved on that set: \eqref{eq:conforming-potential} uses the data
and with them the point evaluation $v_h(x_0)$, whereas
\eqref{eq:hcurl-recovery} uses the computed flux and never touches
$\delta_{x_0}$.  Replacing $\bsigma_h$ by $\bsigma$ in
\eqref{eq:hcurl-recovery} and writing $\btau_h=\nabla v_h$ returns
\eqref{eq:conforming-potential} exactly, so the second construction is
the first with the exact data replaced by the discrete flux they
produced.

This route is tied to the operator treated here.  It returns a field and
no potential, which suffices because the constitutive relation of
\eqref{eq:first-order-system} involves $\nabla u$ alone; for an operator
carrying lower-order terms the potential itself enters both equations,
and a recovery producing no potential cannot form the residual.

Theorem~\ref{thm:norm-equivalence} may be read in the same way for this
operator.  Writing $\mathcal G_0^q(\O):=\{\nabla v:\ v\in W_0^{1,q}(\O)\}$,
every $\bxi\in\mathcal G_0^q(\O)$ is $\nabla v$ for a unique
$v\in W_0^{1,q}(\O)$ with $\|\nabla v\|_{L^q(\O)}\simeq\|v\|_{W^{1,q}(\O)}$,
so that \eqref{eq:norm-equivalence-v} is the same statement as
\begin{equation}\label{eq:norm-equivalence}
\|\btau\|_{L^q(\O)}+\|\bxi\|_{L^q(\O)}+\|\gradt\btau\|_{L^q(\O)}
\simeq
\|A^{-1/2}\btau+A^{1/2}\bxi\|_{L^q(\O)}
+\|\gradt\btau\|_{L^q(\O)}
\end{equation}
for $\btau\in H^q(\divvr;\O)$ and $\bxi\in\mathcal G_0^q(\O)$, and it is
\eqref{eq:norm-equivalence} that the field produced above is measured
against.  The reduction is the same one: only $\nabla v$ enters the
residual here.
\end{remark}

\begin{remark}[Local reconstructions]
\label{rem:local-reconstruction}
Both constructions have local counterparts, in which the global solve is
replaced by independent problems on vertex patches: for
\eqref{eq:conforming-potential} a reconstruction of the potential, for
\eqref{eq:hcurl-recovery} the patchwise minimization used in
\cite{CaiCaiZhang2020} to obtain estimators robust
in the contrast of $A$ within the $L^2$ theory.  Nothing in
Theorem~\ref{thm:estimator} depends on how $w$ is obtained, so any
local construction that returns a potential satisfying
\eqref{eq:admissible} may be used in it,
changing the cost of producing $w$ but not the equivalence.  The local construction will be treated in a separate
paper.
\end{remark}

\begin{remark}[A posteriori control is not enough by itself]
\label{rem:apriori-needed}
Theorem~\ref{thm:estimator} holds for every $\bsigma_h\in\RT_0^{f_1}(\cT_h)$
and every $w$ satisfying \eqref{eq:admissible}, whether or not they
solve anything; that
is what makes the estimator usable with a solver it was not derived
from.  The same freedom means that reliability and efficiency alone say
nothing about an adaptive loop driven by $\eta$.  If the pair is a poor
approximation, the estimator faithfully reports a large error and the
marking it produces refines where that pair happens to be bad, which
need not be where the solution is singular; the resulting meshes need
not reflect the approximation needs of the exact solution.  What is needed alongside is a priori
control of both fields, so that the pair converges and the indicators
localize the error of the exact solution.  This point, and a plain
convergence proof for adaptive algorithms built on a solve-and-recover
pair of exactly this kind, are given in \cite{LiZhang2025}, where the
two steps are treated as one combined problem in the framework of
\cite{Siebert2011}; see also \cite{FuhrerPraetorius2020}.

\end{remark}

\section{Numerical experiments}
\label{sec:numerics}

The experiments below test the points on which the analysis rests: that
the assembly of the singular field is accurate, that the modified flux
is genuinely an $L^p$ object and not an $L^2$ one in disguise, that the
method is indifferent to where $x_0$ sits relative to the mesh, and that
the residual estimator remains numerically effective when the canonical
interpolant used in the a priori analysis does not exist.
Section~\ref{subsec:num-problems} sets up three problems in which
$\bsigma\notin L^2(\O)^2$, Section~\ref{subsec:num-adaptive} reports
the adaptive runs, Section~\ref{subsec:num-uniform} the uniform ones,
where the exponent $p$ becomes visible, and
Section~\ref{subsec:num-extras} what the convergence orders do not
show.

\subsection{Setting}
\label{subsec:num-setting}

All computations use the $\RT_0$--$P_0$ pair \eqref{eq:mixed-discrete}
in two dimensions, and in every experiment $f_1=0$ and
$g=u_{\rm ex}|_{\p\O}$.  The adaptive runs and the meshes they produce
use $p=1.2$, so that $p'=6$; in Section~\ref{subsec:num-uniform} the
same discrete solutions are evaluated in the $L^{1.5}$ and $L^{1.8}$
norms as well, the discrete problem itself not involving $p$.  Two consequences of $f_1=0$ are
used throughout.  First, \eqref{eq:discrete-conservation} gives
$\gradt\bsigma_h=Q_hf_1=0$ exactly, so the discrete flux is exactly
equilibrated and the data oscillation term of \eqref{eq:estimator}
vanishes:
\begin{equation}\label{eq:eta-num}
  \eta=\bigl\|A^{-1/2}(\bsigma_h+\bff_2)+A^{1/2}\nabla u_c\bigr\|_{L^p(\O)} .
\end{equation}
Second, $\gradt\bsigma=0$, so the whole of the error lives in the
constitutive relation and, if Assumption~\ref{ass:A3} holds at $q=p$,
\eqref{eq:estimator-equiv} reads
\begin{equation}\label{eq:equiv-num}
  \|\bsigma-\bsigma_h\|_{L^p(\O)}+\|u-u_c\|_{W^{1,p}(\O)}\simeq\eta .
\end{equation}
The effectivity indices below test \eqref{eq:equiv-num} numerically for
the coefficients used here.
The potential is $w=u_c$ of
Section~\ref{subsec:potential-field}, the continuous
piecewise linear solution of \eqref{eq:conforming-potential} on the same
mesh; the datum $g$ is not piecewise linear, so $u_c$ carries its nodal
interpolant and \eqref{eq:admissible} holds with $g$ replaced by that
interpolant.  The
equivalence \eqref{eq:equiv-num} and the effectivity indices reported
below are therefore those of the interpolated-boundary problem; no
boundary-data oscillation term is included, and the difference between
the two problems is not quantified here.  Adaptive meshes are produced by D\"orfler marking on
the indicators $\eta_K$ with $\theta=0.25$ followed by newest-vertex
bisection.  Since the meshes are graded, orders are reported with
respect to $N^{-1/2}$, $N$ the number of degrees of freedom; on a
quasi-uniform family $N\simeq h^{-2}$, so an order $1$ in $N^{-1/2}$ is
an order $1$ in $h$.

The only non-polynomial integrals in the assembly are those against
$\bff_2$.  They are computed on a Duffy map with the pole at the
collapsed vertex, splitting the element that contains $x_0$ into the
three triangles with apex $x_0$ when $x_0$ is not a vertex.  Under that
map the integrand of \eqref{eq:eta-num} behaves like $r^{1-p}$, and the
substitution $r=t^{1/(2-p)}$ removes the endpoint singularity exactly;
twenty Gauss points in each direction are used throughout.

\begin{remark}[Verification on the matched problem]
\label{rem:num-matched}
For $A=I$ and $u_{\rm ex}=-\frac1{2\pi}\log|x-x_0|$ one has
$\bff_2=-\nabla u_{\rm ex}$, so $\bsigma\equiv0$ by
\eqref{eq:modified-flux} and the discrete solution is available in
closed form.  Subtracting \eqref{eq:mixed-discrete} from the same
identity satisfied by the exact pair gives
$(\bsigma_h,\btau_h)=(u_h-u,\gradt\btau_h)$ for every
$\btau_h\in\RT_0(\cT_h)$, whence $\bsigma_h=0$ and $u_h=Q_hu$,
both exactly.  The computation reproduces this: on the final mesh
$(\bsigma_h,\bsigma_h)^{1/2}=6.6\cdot10^{-15}$, so what the flux error
measures there is the quadrature of the singular load, against an exact
zero.  Moreover \eqref{eq:eta-num} reduces to
$\eta=\|\nabla(u-u_c)\|_{L^p(\O)}$, which coincides exactly with the
denominator of \eqref{eq:num-eff}, and the computed effectivity index is
$1$ to six digits.  This is a test of the assembly, not of the theory:
there is no singular flux left to approximate.  The experiments below
are the ones in which there is.
\end{remark}

\subsection{Three problems with \texorpdfstring{$\bsigma\notin L^2$}{sigma not in L2}}
\label{subsec:num-problems}

We use the three coefficients of Table~\ref{tab:num-problems}.  In each
case $f_1=0$, the exact solution is known, and the modified flux
\eqref{eq:modified-flux} takes the form
\begin{equation}\label{eq:num-common}
  \bsigma(r,\theta)=\frac1rS(\theta),\qquad S\not\equiv0 ,
\end{equation}
so that $\bsigma\in L^p(\O)^2$ for every $p<2$ and
$\bsigma\notin L^2(\O)^2$.  Two things follow at once.  The situation
\eqref{eq:sigma-Lp-only} is realized here, not just possible; and since
\eqref{eq:leading-singular-part} holds with $\bq=\bsigma$ and
$\widetilde{\bsigma}=0$ near the pole, the lower bound
\eqref{eq:global-lower} applies to these fields directly.  On a
quasi-uniform family it forbids any order better than $h^{2/p-1}$, and
the runs below are read against it.

\begin{table}[htbp]
\centering
\small
\begin{tabular}{@{}llll@{}}
\toprule
 & $A$ & $\O$, $x_0$ & mismatch caused by\\
\midrule
(I) & $\alpha I$, $\alpha=100$ and $1$ across $\Gamma$
    & $(-1,1)^2$, $x_0=(0,0)\in\Gamma$ & the jump in $\alpha$\\
(A) & $\begin{pmatrix}4&1\\1&2\end{pmatrix}$, constant
    & $(-1,1)^2$, $x_0=(0,0)$ or $(\tfrac23,0)$ & the anisotropy of $A$\\
(S) & \eqref{eq:sector-A} with $a=4$
    & $(0,1)^2$, $x_0=(\tfrac12,\tfrac12)$ & both\\
\bottomrule
\end{tabular}
\caption{The three test problems.  In (I) the interface is
$\Gamma=\{x_1=0\}$ and the pole lies on it; in (S) the coefficient is
$\mathrm{diag}(4,1)$ on $0<\theta<\pi/2$ and $I$ elsewhere, so that it
jumps across the two rays $\theta=0$ and $\theta=\pi/2$ issuing from the
pole.
Problem (A) is run twice, with $x_0$ at a mesh vertex and at the
barycenter of an initial element.}
\label{tab:num-problems}
\end{table}

In (I) the exact solution is $-\frac{1}{101\pi}\log|x-x_0|$ and
\begin{equation}\label{eq:num-I-sigma}
  \bsigma=\Bigl(\frac{\alpha(x)}{101\pi}-\frac{1}{2\pi}\Bigr)
          \frac{x-x_0}{|x-x_0|^2}
        =\pm\frac{99}{202\pi}\,\frac{x-x_0}{|x-x_0|^2} ,
\end{equation}
with the plus sign for $x_1<0$.  The coefficient is scalar on each side,
so this is the mildest configuration in which
$\bsigma\notin L^2(\O)^2$: the mismatch is caused by the jump alone.  In
(A), writing $r_A(x)^2=(x-x_0)^TA^{-1}(x-x_0)$ and taking the
anisotropic fundamental solution
$u_{\rm ex}=-\frac{1}{2\pi\sqrt{\det A}}\log r_A$,
\begin{equation}\label{eq:num-A-sigma}
  \bsigma=\frac{1}{2\pi}
   \Bigl(\frac{1}{\sqrt{\det A}\,r_A(x)^2}-\frac{1}{|x-x_0|^2}\Bigr)(x-x_0),
  \qquad
  S(\theta)=\frac{1}{2\pi}\Bigl(\frac{1}{\sqrt{\det A}\,q(\theta)}-1\Bigr)
            v(\theta),
\end{equation}
with $v(\theta)=(\cos\theta,\sin\theta)$ and $q=v^TA^{-1}v$.  Here
$S\equiv0$ would force $q$ constant, that is $A^{-1}$ a multiple of the
identity; the coefficient is not, so $S\not\equiv0$.  This is the
mechanism of Section~\ref{subsec:phiA} in its plainest form: $\bff_2$
carries the right Dirac mass and nothing else, and what it fails to
cancel is exactly the anisotropy of $A$ at the pole.  In (S) the exact
solution is \eqref{eq:sector-solution} translated to $x_0$, and
$u_{\rm ex}$, $\bsigma$ enter only through the boundary datum and the
error computation; the discrete problem uses $\bff_2$ alone in all three
cases.

The meshes are fitted to the interfaces of (I) and (S) without any
special construction.  In both, the longest edge of each initial element
has its midpoint on an interface ray, so the first bisection joins that
midpoint to $x_0$ and puts the interface on the skeleton, where
newest-vertex bisection keeps it.

\begin{figure}[htbp]
\centering
\begin{minipage}{0.44\textwidth}\centering
  \includegraphics[width=\linewidth]{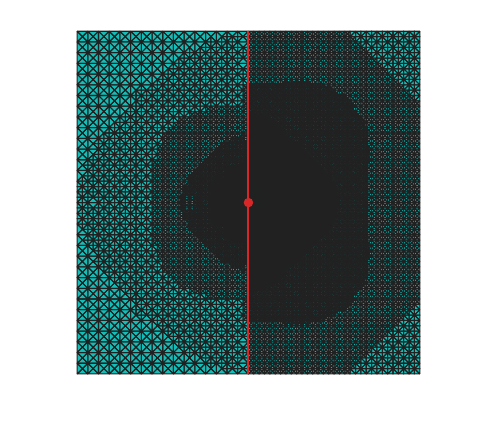}\\[2pt]
  {\scriptsize (I) interface}
\end{minipage}\hfill
\begin{minipage}{0.44\textwidth}\centering
  \includegraphics[width=\linewidth]{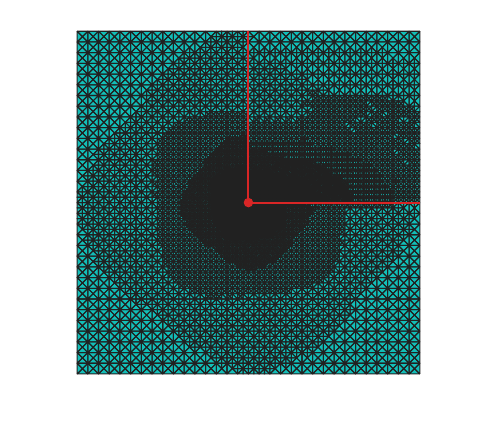}\\[2pt]
  {\scriptsize (S) angular jump}
\end{minipage}

\vspace{6pt}

\begin{minipage}{0.44\textwidth}\centering
  \includegraphics[width=\linewidth]{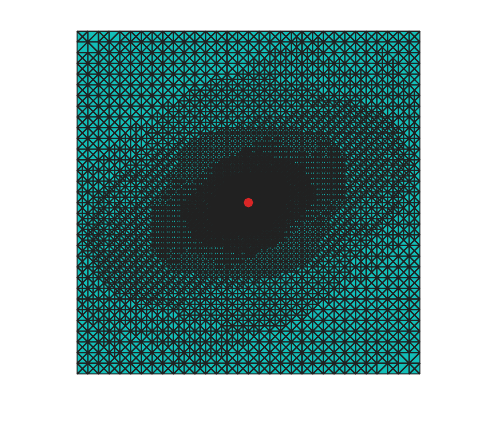}\\[2pt]
  {\scriptsize (A), $x_0$ a vertex, bisection}
\end{minipage}\hfill
\begin{minipage}{0.44\textwidth}\centering
  \includegraphics[width=\linewidth]{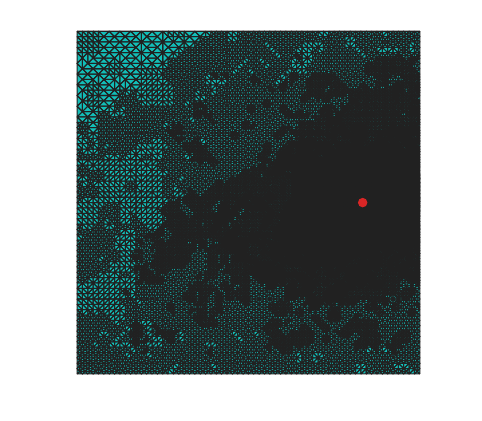}\\[2pt]
  {\scriptsize (A), $x_0$ interior, red--green--blue}
\end{minipage}
\caption{Adaptive meshes for the problems of
Table~\ref{tab:num-problems}, with the interfaces and the pole marked.
The two lower panels carry the same coefficient and differ in the
position of the pole and in the refinement rule; they produce the two
rows for (A) in Table~\ref{tab:num-summary}.  In both, the refined
region is tilted by $\pi/8$ and four-lobed, because
$A=R_{\pi/8}\,\mathrm{diag}(3+\sqrt2,3-\sqrt2)\,R_{\pi/8}^{T}$ exactly
and $S$ of \eqref{eq:num-A-sigma} vanishes on four rays, at
$\arctan\bigl((\lambda_2/\lambda_1)^{1/4}\bigr)=37.7^\circ$ from each principal
axis.  Nothing in the discrete problem refers to those directions:
$\bff_2$ is radially symmetric and the marking sees only $\eta_K$.}
\label{fig:num-meshes}
\end{figure}

\subsection{Adaptive results}
\label{subsec:num-adaptive}

Table~\ref{tab:num-summary} collects the four runs.  All orders lie
between $0.98$ and $1.07$, that is at $N^{-1/2}$, against the
quasi-uniform exponent $2/p-1=2/3$ of Corollary~\ref{cor:global-rate};
what removes the loss is the refinement at the pole, as
Remark~\ref{rem:graded} anticipates.

\begin{table}[htbp]
\centering
\small
\begin{tabular}{@{}lccccc@{}}
\toprule
& \multicolumn{4}{c}{order in $N^{-1/2}$} &\\
\cmidrule(r){2-5}
problem & $\eta$ & $\|A^{-1/2}(\bsigma-\bsigma_h)\|_{L^p}$
        & $\|A^{1/2}\nabla(u-u_c)\|_{L^p}$ & $\|u-u_h\|_{L^2}$
        & $\mathrm{eff}$\\
\midrule
(I)                 & $1.00$ & $0.99$ & $1.04$ & $1.07$ & $0.956$\\
(A), $x_0$ a vertex & $1.00$ & $0.98$ & $1.00$ & $1.01$ & $0.828$\\
(A), $x_0$ interior & $1.01$ & $1.00$ & $1.01$ & $0.99$ & $0.834$\\
(S)                 & $1.02$ & $1.01$ & $1.01$ & $1.02$ & $0.853$\\
\bottomrule
\end{tabular}
\caption{Adaptive runs, $p=1.2$, $\theta=0.25$.  Here
$\mathrm{eff}$ is the effectivity index \eqref{eq:num-eff}.  The
interior run of (A) uses red--green--blue refinement.}
\label{tab:num-summary}
\end{table}

The estimator is assembled as an $\ell^p$ combination of the two
residuals, so the error side is combined the same way.  For the reported
index we use the $A$-weighted gradient form of the augmented error,
equivalent to the left-hand side of \eqref{eq:equiv-num} by ellipticity
and Poincar\'e's inequality:
\begin{equation}\label{eq:num-eff}
  \mathrm{eff}
  :=\frac{\eta}
     {\bigl(\|A^{-1/2}(\bsigma-\bsigma_h)\|_{L^p(\O)}^p
           +\|A^{1/2}\nabla(u-u_c)\|_{L^p(\O)}^p\bigr)^{1/p}} .
\end{equation}
With $f_1=0$ the constitutive identity
$A^{-1/2}(\bsigma+\bff_2)+A^{1/2}\nabla u=0$ turns \eqref{eq:eta-num}
into
$\eta=\|A^{-1/2}(\bsigma-\bsigma_h)+A^{1/2}\nabla(u-u_c)\|_{L^p(\O)}$,
the two components entering pointwise, so by the triangle inequality and
$a+b\le2^{1-1/p}(a^p+b^p)^{1/p}$ the index cannot exceed $2^{1-1/p}$,
which is $1.12$ at $p=1.2$.  The reverse bound, that the two components
cannot cancel, is the content of Theorem~\ref{thm:estimator} and rests
on Assumption~\ref{ass:A3} at $q=p=1.2$, which is not established here
for these coefficients at that exponent.  The effectivity column of
Table~\ref{tab:num-summary} tests it: the index stays between $0.83$ and
$0.96$, so the two components do not cancel in any of the runs.  It is also stable in
$h$; the corresponding column of Table~\ref{tab:num-uniform} varies in
the third digit over a factor $16$ in the mesh size.

The two rows for (A) differ in the position of the pole and in the
refinement rule, and agree to $0.02$ in every reported order, with very
similar effectivity indices.  Neither configuration is neutral for a
method that carries $\delta_{x_0}$ in the conservation law, and they are
awkward for different reasons.  When $x_0$ is a vertex the point
evaluation $\delta_{x_0}(v_h)=v_h(x_0)$ is not canonical for a
discontinuous $v_h$, which has several values there; the a priori
analysis of \cite{HoustonWihler2012} is carried out under the hypothesis
that $x_0$ lies in the interior of an element, and the hybridizable
method of \cite{LengChen2022} replaces the load by the average of
$v_h|_K(x_0)$ over the elements containing $x_0$.  When $x_0$ is
interior the point value is canonical, but the residual estimator of
\cite{LengChen2022} then carries an extra term $h_K^{4-d}$ on the
element containing the pole, absent when $x_0$ is a node, because the
Lagrange interpolant of the dual solution is exact at nodes and not
elsewhere.  In \eqref{eq:mixed-discrete} and \eqref{eq:estimator} no
point value occurs and no term is added or removed: the same formulas
produce the two rows above.

\subsection{Uniform refinement and the exponent}
\label{subsec:num-uniform}

The rates of Corollaries~\ref{cor:global-rate} and
\ref{cor:scalar-rate} are statements about quasi-uniform families,
and on the graded meshes above they are invisible: every exponent
returns the order $1$ of Table~\ref{tab:num-summary}.  We therefore
repeat problem~(I) on a uniformly refined family, with $h$ halved at
each of five levels from $h=1/8$ to $h=1/128$.

The discrete problem does not involve $p$: by \eqref{eq:mixed-discrete}
the pair $(\bsigma_h,u_h)$, and with it $u_c$, is determined by the mesh
alone, and $p$ enters only the norms.  Each mesh is therefore solved
once and all three exponents are evaluated on that one solution.

\begin{table}[htbp]
\centering
\small
\begin{tabular}{@{}lccccc@{}}
\toprule
$p$ & $\|A^{-1/2}(\bsigma-\bsigma_h)\|_{L^p(\O)}$ & $2/p-1$
    & $\eta$ & $\|A^{1/2}\nabla(u-u_c)\|_{L^p(\O)}$ & $\mathrm{eff}$\\
\midrule
$1.2$ & $0.639$ & $0.667$ & $0.637$ & $0.632$ & $0.951$--$0.953$\\
$1.5$ & $0.332$ & $0.333$ & $0.332$ & $0.332$ & $0.9750$--$0.9752$\\
$1.8$ & $0.111$ & $0.111$ & $0.111$ & $0.111$ & $0.9911$\\
\bottomrule
\end{tabular}
\caption{Problem~(I) on uniform meshes: observed orders in $h$ against
the predicted flux exponent $2/p-1$, and the range of the effectivity
index over the five levels.  Each order is the mean of the last two
successive-level increments.  The scalar error is $p$-independent and
has order $1.00$.}
\label{tab:num-uniform}
\end{table}

\begin{figure}[htbp]
\centering
\begin{minipage}{0.46\textwidth}\centering
  \includegraphics[width=\linewidth]{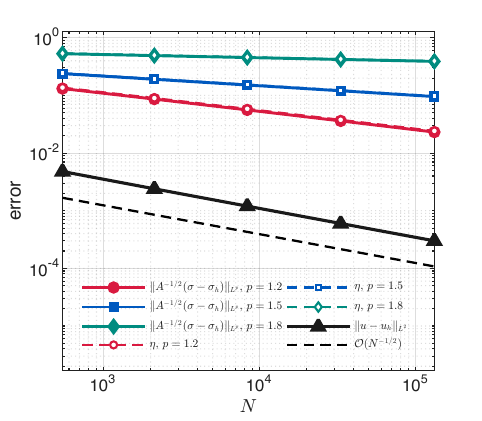}
\end{minipage}\hfill
\begin{minipage}{0.46\textwidth}\centering
  \includegraphics[width=\linewidth]{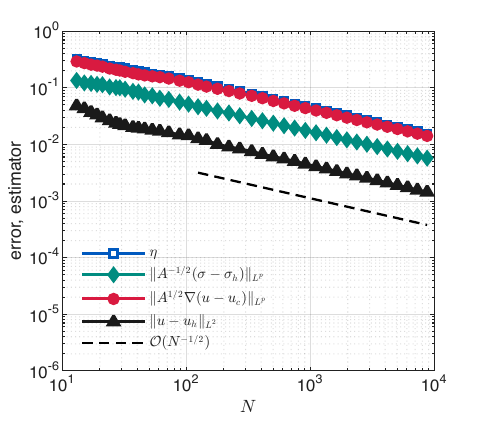}
\end{minipage}
\caption{Left: problem~(I) on uniform meshes.  For each exponent the
flux error and the estimator are drawn in the same colour and are barely
distinguishable; the three pairs separate according to
Table~\ref{tab:num-uniform}, while the scalar error, lowest curve, is
the same for all three.  Right: problem~(S) on adaptive meshes, with the
reference slope $N^{-1/2}$.}
\label{fig:num-convergence}
\end{figure}

At $p=1.5$ and $p=1.8$ the measured flux order agrees with $2/p-1$ to
three decimals; at $p=1.2$ it is $0.639$ against $0.667$ and rises along
the sequence, the level-by-level values being $0.607$, $0.623$, $0.635$,
$0.643$.  That the deficit is largest at the smallest exponent is
expected.  The singular and far-field contributions to the fitted-mesh
estimate enter the $p$-th power as $h^{2-p}$ and $h^{p}$, so that
relative to the singular one the smooth far-field correction is of order
$h^{2p-2}$, which is $h^{0.4}$ at $p=1.2$ and $h^{1.6}$ at $p=1.8$; the
scaling is consistent with the much slower approach to the asymptotic
slope at the smaller exponent.

The scalar variable behaves in the opposite way: its order is $1.00$
and does not depend on $p$.  This is what \eqref{eq:scalar-refined}
predicts once the flux rate is inserted, the cancellation being exact:
$h^{2/p'}\|\bsigma-\bsigma_h\|_{L^p(\O)}\simeq h^{2/p'}h^{2/p-1}=h$,
because $1/p+1/p'=1$.  The measured decay is moreover free of any
logarithm: over the last three levels the ratio of successive errors is
$0.500$ at both levels, whereas the best-approximation term
$\|u-Q_hu\|_{L^2(\O)}$, which carries $|\log h|^{1/2}$ by
Proposition~\ref{prop:P0-log}, would give $0.548$ and an apparent order
near $0.87$.  By \eqref{eq:pythagoras} the two contributions add in
squares, so what the measurement says is that in problem~(I) the
discrete one dominates at these mesh sizes.
Corollary~\ref{cor:scalar-rate} is what holds asymptotically; at the
mesh sizes reached here, problem~(I) measures the term the duality
argument estimates.
  This also separates the two scalar estimates: on a
quasi-uniform family Theorem~\ref{thm:scalar-basic} bounds
$\|Q_hu-u_h\|_{L^2(\O)}$ by $C\|\bsigma-\bsigma_h\|_{L^p(\O)}$ and so
gives only the weaker bound $O(h^{2/p-1})$, of order $0.667$ at
$p=1.2$, whereas Theorem~\ref{thm:scalar-refined} improves this to
$O(h)$; the measured order $1.00$ exhibits the sharper estimate.  No convergence at all is what the naive argument of
Remark~\ref{rem:why-inverse-fails} would predict, the inverse estimate
turning the flux order into $O(1)$; Lemma~\ref{lem:discrete-Lp-lifting}
and Corollary~\ref{cor:discrete-Lp-infsup} exist to avoid exactly that.

\subsection{What the orders do not show: the discrete flux diverges in
\texorpdfstring{$L^2$}{L2}}
\label{subsec:num-extras}

The same computation that produces a convergent
$\|\bsigma-\bsigma_h\|_{L^p(\O)}$ produces a divergent
$\|A^{-1/2}\bsigma_h\|_{L^2(\O)}$, and the exact flux says at what rate.
By \eqref{eq:num-A-sigma},
$\bsigma^{T}A^{-1}\bsigma=(2\pi r)^{-2}g(\theta)$ with
$g=\bigl(({\sqrt{\det A}\,q})^{-1}-1\bigr)^{2}q\ge0$, so that for any
fixed small $r_0$
\[
  \|A^{-1/2}\bsigma\|_{L^2(B_{r_0}(x_0)\setminus B_h(x_0))}^2
  =\frac{1}{4\pi^2}\int_0^{2\pi}g(\theta)\,d\theta\int_h^{r_0}\frac{dr}{r}
\]
and therefore
\begin{equation}\label{eq:num-L2-law}
  \|A^{-1/2}\bsigma\|_{L^2(\O\setminus B_h(x_0))}\simeq|\log h|^{1/2} ,
\end{equation}
provided $g\not\equiv0$.  It vanishes identically exactly when $q$ is
constant, that is in the matched case of
Remark~\ref{rem:num-matched}, where the norm stays at zero.

\begin{table}[htbp]
\centering
\begin{tabular}{@{}rccc@{}}
\toprule
$N$ & $h_{\min}$ & $\|A^{-1/2}\bsigma_h\|_{L^2(\O)}$
    & $\|A^{-1/2}\bsigma_h\|_{L^2(\O)}\,|\log h_{\min}|^{-1/2}$\\
\midrule
$113$   & $1.56\cdot10^{-2}$ & $0.1960$ & $0.0961$\\
$688$   & $6.91\cdot10^{-4}$ & $0.2531$ & $0.0938$\\
$4100$  & $3.05\cdot10^{-5}$ & $0.3009$ & $0.0933$\\
$8815$  & $7.63\cdot10^{-6}$ & $0.3196$ & $0.0931$\\
\bottomrule
\end{tabular}
\caption{Problem~(A), $x_0$ at a vertex: the $L^2$ norm of the discrete
flux, and the same norm divided by $|\log h_{\min}|^{1/2}$.}
\label{tab:num-L2}
\end{table}

Table~\ref{tab:num-L2} is consistent with the same logarithmic scale for
the discrete flux.  The norm itself grows by a factor of $1.6$ while
$h_{\min}$ falls by more than three orders of magnitude; divided by
$|\log h_{\min}|^{1/2}$ it moves by three percent over the same range,
the scale of \eqref{eq:num-L2-law}.  For one and
the same discrete flux, then, the $L^p$ error decays at the observed
order $0.98$ while the $L^2$ norm grows in step with the exact
logarithmic scale $|\log h_{\min}|^{1/2}$.  No error analysis carried
out in $L^2$ can describe this computation, and the divergence is not a
defect of the method: $\bsigma_h$ converges, in the norm the modified
flux possesses.

\section{Concluding remarks}
\label{sec:conclusions}

The change made in this paper is small and its consequences are not.
The divergence-form splitting \eqref{eq:f2}--\eqref{eq:modified-flux}
rewrites \eqref{eq:deltaeq} as \eqref{eq:divform}, an elliptic equation with data in divergence
form; the Dirac measure is gone and what stands in its place is a
standard divergence-form right-hand side for a second-order elliptic
problem.  The
identity is exact and $\bff_2$ is explicit, so nothing is approximated
and nothing is recomputed when the coefficient changes.  What is paid,
in the unmatched regime this analysis is written for, is that the
modified flux lies below the Hilbert scale and the error analysis has to
follow it there.

Under Hypothesis~(SP) and the local structure assumptions on $\bsigma$,
the modified flux satisfies an $L^p$ quasi-best approximation bound and
the rate $h^{2/p-1}$, sharp on quasi-uniform families when the leading
$|x-x_0|^{-1}$ mismatch does not vanish, and first-order complexity on
the graded family.  Under dual regularity the projected scalar error is
first order, and in the nonvanishing logarithmic regime the total scalar
error is asymptotically the best piecewise constant approximation of
$u$.  The residual norm equivalence yields a computable $L^p$ estimator,
reliable and locally efficient for the augmented flux--potential pair.

\subsection*{Other discretizations}

The scalar equation \eqref{eq:divform} is ready for other methods, as
Remark~\ref{rem:other-methods} says, and this is where we expect the
substitution to be worth more than it is worth here.  A conforming, a
discontinuous Galerkin, a hybridizable, a finite volume or a
nonconforming method applied to \eqref{eq:divform} is a method applied
to an equation with data in $W^{-1,p}$ of divergence form, a familiar
data class; the Dirac measure never enters, and neither does anything
built to accommodate it.  Method-specific stability and error analysis
are of course still required.  What
is missing is not the formulation but the analysis, and it is not the
same analysis in each case: for a broken scalar space the residual norm equivalence of
Section~\ref{subsec:residual-equivalence} acquires jump contributions
and the counterpart of Hypothesis~(SP) has to be identified for the
corresponding discrete kernel.  Those cases will be treated in separate
papers.

\subsection*{General second-order operators}

For the full operator
\[
  \mathcal Lu=-\gradt\bigl(A\nabla u+\bb\,u\bigr)+cu=f_1+\delta_{x_0} ,
  \qquad \bb\in L^\infty(\O)^2,\ c\in L^\infty(\O) ,
\]
the splitting itself is unaffected (it is stated for this operator
already in \cite[Sec.~4]{Zhang2023}), and this is where its
independence of the coefficient pays off: $\bff_2$ is required to carry
the Dirac mass, not to reproduce the singular structure of $\mathcal L$,
so setting $\bsigma_{\rm phys}:=-(A\nabla u+\bb\,u)$ and
$\bsigma:=\bsigma_{\rm phys}-\bff_2$ again
removes the measure from the conservation law, with no discrete delta
and nothing to recompute.  The a posteriori theory of
Section~\ref{sec:aposteriori} keeps its structure once the
corresponding $W^{1,p}$ stability is available: the singular lifting
enters the constitutive residual, while the balance residual stays an
ordinary $L^p$ quantity.  The full operator
will be treated in a separate paper.


\end{document}